\documentclass{amsart}

\usepackage{amssymb}
\usepackage{amsthm}
\usepackage{amsmath}
\usepackage{enumitem}
\usepackage[mathscr]{euscript}

\usepackage{cancel}
\usepackage{soul}

\usepackage[normalem]{ulem}

\usepackage[foot]{amsaddr}

\usepackage[usenames]{xcolor}

\usepackage{hyperref}

\usepackage{tikz-cd}

\theoremstyle{plain}
\newtheorem{proposition}{Proposition}[section] 
\newtheorem{theorem}[proposition]{Theorem}
\newtheorem{lemma}[proposition]{Lemma}  
\newtheorem{corollary}[proposition]{Corollary}
\theoremstyle{definition}
\newtheorem{example}[proposition]{Example} 
\newtheorem{definition}[proposition]{Definition}
\newtheorem{observation}[proposition]{Observation}
\newtheorem{conjecture}[proposition]{Conjecture}

\newtheorem{problem}[proposition]{Problem}
\theoremstyle{remark}
\newtheorem{remark}[proposition]{Remark}

\DeclareMathOperator{\Isom}{\mathsf{Isom}}

\DeclareMathOperator{\diam}{diam}

\DeclareMathOperator{\supp}{supp}

\DeclareMathOperator{\interior}{int}

\DeclareMathOperator{\id}{id}
\DeclareMathOperator{\rank}{rank}

\DeclareMathOperator{\dist}{d}

\DeclareMathOperator{\Cc}{\mathcal{C}}
\DeclareMathOperator{\Fc}{\mathcal{F}}

\DeclareMathOperator{\Nc}{\mathcal{N}}
\DeclareMathOperator{\Oc}{\mathcal{O}}

\DeclareMathOperator{\Tc}{\mathcal{T}}

\DeclareMathOperator{\Pc}{\mathcal{P}}

\DeclareMathOperator{\Hb}{\mathbb{H}}

\DeclareMathOperator{\Nb}{\mathbb{N}}
\DeclareMathOperator{\Pb}{\mathbb{P}}
\DeclareMathOperator{\Rb}{\mathbb{R}}
\DeclareMathOperator{\Sb}{\mathbb{S}}

\DeclareMathOperator{\Gsf}{\mathsf{G}}
\DeclareMathOperator{\Psf}{\mathsf{P}}

\DeclareMathOperator{\SL}{\mathsf{SL}}

\newcommand{\abs}[1]{\left|#1\right|}
\newcommand{\norm}[1]{\left\|#1\right\|}

\newcommand{\Ga}{\Gamma}
\newcommand{\ga}{\gamma}
\newcommand{\F}{\mathcal{F}}
\newcommand{\La}{\Lambda}

\newcommand{\Msf}{\mathsf{M}}
\newcommand{\Asf}{\mathsf{A}}

\newcommand{\Nsf}{\mathsf{N}}
\newcommand{\Ksf}{\mathsf{K}}

\newcommand{\Usf}{\mathsf{U}}

\newcommand{\fa}{\mathfrak{a}}
\newcommand{\mfa}{\mathfrak{a}}
\newcommand{\mfp}{\mathfrak{p}}

\newcommand{\Hsf}{\mathsf{H}}

\newcommand{\R}{\mathbb{R}}

\newcommand{\ba}{\backslash}

\newcommand{\opp}{\mathrm{i}}

\newcommand{\Mod}{\operatorname{Mod}}

\newcommand{\PMF}{\mathcal{PMF}}   

\newcommand{\UE}{\mathcal{UE}}

\begin{document}

\title[Rigidity for PS-systems, random walks, and entropy~rigidity]{Rigidity for Patterson--Sullivan systems with applications to random walks and entropy~rigidity}

\author[Kim]{Dongryul M. Kim}
\email{dongryul.kim97@gmail.com}
\address{Department of Mathematics, Yale University, USA}

\author[Zimmer]{Andrew Zimmer}
\email{amzimmer2@wisc.edu}
\address{Department of Mathematics, University of Wisconsin-Madison, USA}



\begin{abstract} In this paper we introduce Patterson--Sullivan systems,  which consist of a group action on a compact metrizable space and a quasi-invariant measure which behaves like a classical Patterson--Sullivan measure. For such systems we prove a generalization of Tukia's measurable boundary rigidity theorem. We then apply this generalization to (1) study the singularity conjecture for Patterson--Sullivan measures (or, conformal densities) and stationary measures of random walks on isometry groups of Gromov hyperbolic spaces, mapping class groups, and discrete subgroups of semisimple Lie groups; (2)  prove versions of Tukia's theorem for word hyperbolic groups, Teichm\"uller spaces, and higher rank symmetric spaces; and (3) in a companion paper prove an entropy rigidity result for  Anosov groups with Lipschitz limit sets.
\end{abstract}

\maketitle

\vspace{-3em}

\setcounter{tocdepth}{1}
\tableofcontents

\section{Introduction}  

Let $\Hb^n$ denote real hyperbolic $n$-space and let $\partial_\infty \Hb^n$ denote its boundary at infinity. Given a discrete subgroup $\Ga < \Isom(\Hb^n)$ and $\delta \ge 0$, a Borel probability measure $\mu$ on $\partial_{\infty} \Hb^n$ is called a Patterson--Sullivan measure (or conformal measure) for $\Ga$ of dimension $\delta$ if for any $\ga \in \Ga$ and Borel subset $E \subset \partial_{\infty} \Hb^n$,
\begin{equation} \label{eqn.conformalmeasure}
\mu(\ga E) = \int_E | \ga'|^{\delta} d \mu.
\end{equation}
These measures play a fundamental role in the study of geometry and dynamics of discrete subgroups of $\Isom(\Hb^n)$, or equivalently, of hyperbolic $n$-manifolds.

The celebrated rigidity theorem of Mostow \cite{Mostow_QC, Mostow_book} asserts that the geometry of a finite-volume hyperbolic $n$-manifold, $n \ge 3$, is determined by its fundamental group (see also \cite{Prasad_rigidity}). 
By considering Patterson--Sullivan measures, Tukia generalized Mostow's rigidity theorem to infinite-volume hyperbolic manifolds, as in the following theorem (which implies Mostow's rigidity).

\begin{theorem} \cite[Thm. 3C]{Tukia1989} \label{thm:Tukia} For $i=1,2$ let $\Gamma_i < \Isom(\Hb^{n_i})$ be a Zariski dense discrete subgroup and let $\mu_i$ be a Patterson--Sullivan measure for $\Gamma_i$ of dimension $\delta_i$. Suppose 
\begin{itemize}
\item $\sum_{\gamma \in \Gamma_1} e^{-\delta_1 \dist(o,\gamma o)} = +\infty$ for some $o \in \Hb^{n_1}$.  
\item There exists an onto homomorphism $\rho : \Gamma_1 \rightarrow \Gamma_2$ and a $\mu_1$-a.e. defined measurable $\rho$-equivariant injective boundary map $f : \partial_{\infty} \Hb^{n_1} \rightarrow \partial_{\infty} \Hb^{n_2}$.
\end{itemize}
If the measures $f_* \mu_1$ and $\mu_2$ are not singular, then $n_1=n_2$ and $\rho$ extends to an isomorphism $\Isom(\Hb^{n_1}) \rightarrow \Isom(\Hb^{n_2})$. 
\end{theorem}

Prior to Tukia's work,  Sullivan~\cite[Thm. 5]{Sullivan1982} proved the above theorem in the special case when $\delta_1 = \delta_2$ and $n_1 = n_2$. Later Yue~\cite{Yue1996} extended Tukia's theorem to discrete subgroups in isometry groups of negatively curved symmetric spaces.

In this paper, we define ``Patterson--Sullivan systems'' which consist of a group action and a quasi-invariant measure which behaves like a classical Patterson--Sullivan measure. More precisely, given a  compact metrizable space $M$ and a subgroup $\Gamma < \mathsf{Homeo}(M)$, a function $\sigma : \Ga \times M \to \mathbb{R}$ is called a \emph{$\kappa$-coarse-cocycle} if 
\begin{equation} \label{eqn.coarsecocycle intro}
   \left| \sigma(\ga_1 \ga_2, x) - \left( \sigma(\ga_1, \ga_2 x) + \sigma(\ga_2, x) \right) \right| \le \kappa
   \end{equation}
  for any $\ga_1, \ga_2 \in \Ga$ and $x \in M$.
Given such a coarse-cocycle and $\delta \ge 0$, a Borel probability measure $\mu$ on $M$ is called \emph{coarse $\sigma$-Patterson--Sullivan measure of dimension $\delta$} if there exists $C\geq1$ such that for any $\ga \in \Ga$ the measures $\mu, \gamma_*\mu$ are absolutely continuous and 
   \begin{equation} \label{eqn.coarse PS meaasure intro}
   C^{-1}e^{ - \delta \sigma(\ga^{-1}, x)} \le \frac{d \ga_* \mu}{d \mu}(x) \le Ce^{- \delta \sigma(\ga^{-1}, x)} \quad \text{for } \mu\text{-a.e. } x \in M.
   \end{equation} 
When $C = 1$ and hence equality holds in Equation \eqref{eqn.coarse PS meaasure intro}, we call $\mu$ a \emph{$\sigma$-Patterson--Sullivan measure}.
\begin{remark}
   We note that we \emph{do not assume anything on the support} of a Patterson--Sullivan measure (e.g. supported on a minimal set). Further, in specific settings these measures are sometimes called \emph{(quasi-)conformal densities}. 

\end{remark}

   A Patterson--Sullivan system consists of a coarse Patterson--Sullivan measure, a collection of open sets called shadows, and a choice of magnitude function all of which satisfy certain properties (see Section~\ref{sec:general framework} for the precise definition). The definition is quite robust and in Example~\ref{ex:PS systems} below we list a number of examples of Patterson--Sullivan systems.
   
   In a classical setting, $M$ is the boundary of real hyperbolic space, the coarse-cocycle is an actual cocycle (implicit in Equation~\eqref{eqn.conformalmeasure}), the shadows are the geodesic shadows, and the magnitude of an element is the distance it translates a fixed basepoint.

   For Patterson--Sullivan systems we prove a version of Tukia's measurable boundary rigidity theorem (Theorem~\ref{thm:Tukia}). Before stating our general theorem in Section~\ref{sec:general framework} below, we describe a number of applications.

\subsection{Random walks}

In this section we describe applications of our main theorem towards the singularity conjecture for Patterson--Sullivan measures and stationary measures of random walks in a variety of settings.  

One novelty in this work is the observation that the singularity conjecture can be studied via Tukia-type measurable boundary rigidity theorems.

\subsubsection{Random walks on Gromov hyperbolic spaces}\label{sec:random walks GH}  Suppose $(X, \dist_X)$  is a proper geodesic Gromov hyperbolic metric space and $\Gamma < \Isom(X)$ is a non-elementary discrete subgroup. Let $\mathsf{m}$ be a probability measure on $\Gamma$ whose support generates $\Gamma$ as a semigroup, i.e.
\begin{equation} \label{eqn.RW generates Gamma} 
\bigcup_{n \geq 1} [\supp \mathsf{m} ]^n = \Gamma.
\end{equation} 
Consider the random walk $W_n = \gamma_1 \cdots \gamma_n$ where the $\gamma_i$'s are independent identically distributed elements of $\Gamma$ each with distribution $\mathsf{m}$. Then, given $o \in X$, almost every sample path $W_n o \in X$ converges to a point in the Gromov boundary $\partial_\infty X$ \cite[Remark following Thm. 7.7]{Kaimanovich2000} (see also \cite{MaherTiozzo2018}). Further, 
\begin{equation} \label{eqn.hitting measure intro}
\nu(A) := {\rm Prob}\left( \lim_{n \rightarrow \infty} W_n o  \in A\right)
\end{equation}
defines a Borel probability measure $\nu$ on $\partial_\infty X$ called the \emph{hitting measure} (or \emph{harmonic measure}) for the random walk associated to $\mathsf{m}$, and is the unique \emph{$\mathsf{m}$-stationary measure} on $\partial_{\infty} X$, that is $\mathsf{m} * \nu = \nu$. 

Fixing a basepoint $o \in X$, the \emph{coarse Busemann cocycle} $\beta : \Gamma \times \partial_\infty X \rightarrow \Rb$ is the coarse-cocycle defined by 
\begin{equation} \label{eqn.Busemann cocycle intro}
\beta(g, x) = \limsup_{p \rightarrow x} \, \dist_X(p, g^{-1} o) - \dist_X(p,o).
\end{equation}
A \emph{coarse Busemann Patterson--Sullivan measure} on $\partial_\infty X$ is a coarse $\beta$-Patterson--Sullivan measure in the sense of Equation~\eqref{eqn.coarse PS meaasure intro}.

We will apply our generalization of Theorem \ref{thm:Tukia} to the following well-studied problem.

\begin{problem}[Singularity Problem]\label{problem:random walks GH} If $\mathsf{m}$ has finite support, determine when the $\mathsf{m}$-stationary measure $\nu$ is singular to some/any coarse Busemann Patterson--Sullivan measure for $\Gamma$ on $\partial_\infty X$.
\end{problem} 

In what follows, we will consider a slightly more general class of probability measures: The probability measure $\mathsf{m}$ has \emph{finite superexponential moment} if 
\begin{equation} \label{eqn.superexponential intro}
   \sum_{\ga \in \Ga} c^{\abs{\ga}} \mathsf{m}(\ga)< +\infty
\end{equation} for any $c > 1$, where $\abs{\cdot}$ is the distance from the identity with respect to a word metric on $\Ga$. 

We first present some applications of one of our main results (Theorem \ref{thm:random walks GH in intro}) towards Problem \ref{problem:random walks GH}. For \emph{any finitely generated Kleinian group}, we obtain the following, which was previously known only for geometrically finite groups \cite{GT2020}.

\begin{corollary}[corollary of Theorem~\ref{thm.singularity harm GH} and Corollary \ref{cor.singularity hitting PS symmetric space intro}]\label{cor.singularity hitting PS H3 intro} Suppose $X = \Hb^3$, $\Gamma < \Isom(\Hb^3)$ is a non-elementary finitely generated discrete subgroup, and $\mathsf{m}$ has finite superexponential moment. If $\Gamma$ is not convex cocompact, then the $\mathsf{m}$-stationary measure  $\nu$ is singular to every coarse Busemann Patterson--Sullivan measure of $\Gamma$ on $\partial_\infty \Hb^3$. In particular, if $\Gamma$ is not a cocompact lattice, then $\nu$ is singular to the Lebesgue measure class on $\partial_{\infty} \Hb^3$.
\end{corollary}

Our results for general $X$ involve relatively hyperbolic groups which is a class of finitely generated groups including word hyperbolic groups, whose definition we delay to Definition \ref{defn:RH}, and quasi-convex subgroups of $\Isom(X)$ which are discrete subgroups whose orbits are quasi-convex in $X$.

\begin{theorem}[corollary of Theorem~\ref{thm:random walks GH in intro}]\label{thm.singularity harm GH} Suppose  $\Gamma$ is relatively hyperbolic (as an abstract group) and $\mathsf{m}$ has finite superexponential moment. If $\Gamma$ is not a quasi-convex subgroup of $\Isom(X)$, then the $\mathsf{m}$-stationary measure $\nu$ is singular to every coarse Busemann Patterson--Sullivan measure on $\partial_\infty X$. 
\end{theorem}

\begin{remark}  

   In the special case when $\Gamma$ is word hyperbolic, $X$ admits a geometric group action, and $\mathsf{m}$ is symmetric, Theorem~\ref{thm.singularity harm GH} is due to Blach{\`e}re--Ha\"issinsky--Mathieu~\cite[Prop. 5.5]{BHM2011}. In the special case when $\Gamma$ acts geometrically finitely on $X$ (which implies it is relatively hyperbolic), Theorem~\ref{thm.singularity harm GH} is due to Gekhtman--Tiozzo~\cite[Coro. 4.2]{GT2020}.

\end{remark} 

Theorem~\ref{thm.singularity harm GH}, in full generality, is new even for negatively curved symmetric spaces. In this case, quasi-convex subgroups are  convex cocompact subgroups,  $\partial_\infty X$ has a smooth structure, and there is always a Busemann Patterson--Sullivan measure in the Lebesgue measure class. Using these facts, we will prove the following. 

\begin{corollary}[see Corollary \ref{cor.singularity hitting PS symmetric space} below]\label{cor.singularity hitting PS symmetric space intro} Suppose $X$ is a negatively curved symmetric space, $\Gamma$ is relatively hyperbolic (as an abstract group), and $\mathsf{m}$ has finite superexponential moment. If $\Gamma$ is not a cocompact lattice in $\Isom(X)$, then the $\mathsf{m}$-stationary measure $\nu$ is singular to the Lebesgue measure class on $\partial_{\infty} X$.
\end{corollary}

Corollary~\ref{cor.singularity hitting PS H3 intro} follows from Theorem \ref{thm.singularity harm GH} and Corollary \ref{cor.singularity hitting PS symmetric space intro}. Indeed, when $X=\Hb^3$, every finitely generated non-elementary discrete subgroup of $\Isom(\Hb^3)$ is relatively hyperbolic relative to some (possibly empty) collection of peripheral subgroups which are virtually abelian. This can be deduced by Scott core theorem \cite{Scott_compact} and Thurston's hyperbolization \cite{Thurston_BAMS} (see also \cite[Thm. 4.10]{MatsuzakiTaniguchi}).

\begin{remark} In the special case when $X = \Hb^n$ is real hyperbolic space, $n \geq 3$, and $\Gamma$ is a non-uniform lattice in $\Isom(\Hb^n)$, Corollary~\ref{cor.singularity hitting PS symmetric space intro}  is due to Randecker--Tiozzo~\cite{RT_cusp}. When $X = \Hb^2$,  this was obtained in different contexts \cite{GL_geodesic, DKN_circle, KP_matrix, GMT_word}. Further, Kosenko--Tiozzo \cite{KT_cocompact} explicitly constructed cocompact lattices of $\Isom(\Hb^2)$ such that hitting measures are singular to the Lebesgue measure class on $\partial_{\infty} \Hb^2$.

\end{remark}

In fact, we show that the non-singularity occurs precisely when any $\Ga$-orbit is roughly isometric to the Green metric associated to the random walk. 
The \emph{Green metric} on $\Gamma$ is defined by 
\begin{equation} \label{eqn.green metric intro}
\dist_G(g,h) = -\log \frac{G_{\mathsf{m}}(g,h)}{G_{\mathsf{m}}(\id,\id)} \quad \text{for } g, h \in \Ga
\end{equation}
where $G_{\mathsf{m}}(g,h) = \sum_{n=0}^\infty \mathsf{m}^{*n}(g^{-1}h)$ is the Green function.  When $\mathsf{m}$ has finite superexponential moment and $\Gamma$ is finitely generated and non-amenable,  the Green metric $\dist_G$ on $\Gamma$ is quasi-isometric to a word metric with respect to a finite generating set \cite[Prop. 7.8]{GT2020}. So Theorem~\ref{thm.singularity harm GH}  is a consequence of the following.

\begin{theorem}[see Theorem~\ref{thm:random walks GH} below]\label{thm:random walks GH in intro} Suppose  $\Gamma$ is relatively hyperbolic (as an abstract group), $\mathsf{m}$ has finite superexponential moment, $\nu$ is the $\mathsf{m}$-stationary measure, and $\mu$ is a coarse Busemann   Patterson--Sullivan measure for $\Ga$ on $\partial_\infty X$ of dimension $\delta$. Then the following are equivalent:
\begin{enumerate}
\item The measures $\nu$ and $\mu$ are not singular.
\item  The measures $\nu$ and $\mu$ are in the same measure class and the Radon--Nikodym derivatives are a.e. bounded from above and below by a positive number.
\item For any $o \in X$, $$\sup_{\gamma \in \Gamma} \abs{ \dist_G(\id, \gamma) - \delta \dist_X(o, \gamma o)} < +\infty.$$ In particular, $\Ga$ is quasi-convex and $\delta$ is the critical exponent of $\Gamma$.
\end{enumerate}
\end{theorem}

When $\Gamma$ is assumed to be a quasi-convex subgroup of $\Isom(X)$ (in particular, word hyperbolic) and $\mathsf{m}$ is symmetric, Theorem \ref{thm:random walks GH in intro} was obtained by Blach{\`e}re--Ha\"issinsky--Mathieu \cite[Thm. 1.5]{BHM2011}. In the special case when $\Gamma$ acts geometrically finitely on $X$ (which implies it is relatively hyperbolic), Theorem~\ref{thm:random walks GH in intro} is due to Gekhtman--Tiozzo~\cite[Thm. 4.1]{GT2020}.  For relatively hyperbolic groups, Dussaule--Gekhtman \cite{DG_entropy} proved an analogous statement for Patterson--Sullivan measure coming from a word metric on $\Ga$.

\subsubsection{Random walks on mapping class groups and Teichm\"uller spaces}\label{sec:RW on Tecihmuller space in intro}

Let $\Sigma$ be a closed connected orientable surface of genus at least two, $\Mod(\Sigma)$ denote the mapping class group of $\Sigma$, and  $(\Tc, \dist_{\Tc})$  denote the Teichm\"uller space of $\Sigma$ endowed with Teichm\"uller metric $\dist_{\Tc}$. 

Thurston~\cite{Thurston_geometry_dynamics} compactified $\Tc$ by the space $\PMF$ of projective measured foliations on $\Sigma$. This compactification is called Thurston's compactification and $\PMF$ is also referred to as Thurston's boundary.

Let $\Ga < \Mod(\Sigma)$ be a non-elementary subgroup (i.e. $\Ga$ is not virtually cyclic and contains a pseudo-Anosov element)  and $\mathsf{m}$ a probability measure on $\Ga$ whose support generates $\Gamma$ as a semigroup. Kaimanovich--Masur \cite{KM_MCG} showed that there exists a unique $\mathsf{m}$-stationary measure $\nu$ on $\PMF$ and the subset $\UE \subset \PMF$ of uniquely ergodic foliations has full $\nu$-measure. Further, for any $o \in \Tc$ the measure $\nu$ is the hitting measure for the associated random walk on the orbit $\Gamma(o) \subset \Tc$.

Analogous to Problem~\ref{problem:random walks GH}, Kaimanovich--Masur suggested the following. 

\begin{conjecture}[Kaimanovich--Masur {\cite[pg. 9]{KM_MCG}}] \label{conj.Teich}
   If $\mathsf{m}$ has finite support, then the $\mathsf{m}$-stationary measure  $\nu$ is singular to every Busemann Patterson--Sullivan measure for $\Gamma$.
\end{conjecture} 

For a special type of Patterson--Sullivan measure which is of Lebesgue measure class on $\PMF$, Gadre \cite{Gadre_singular} proved the singularity of $\mathsf{m}$-stationary measure for finitely supported $\mathsf{m}$. Later, Gadre--Maher--Tiozzo \cite{GMT_Teichmueller} extended this result to $\mathsf{m}$ with finite first moment with respect to a word metric as well. 
   
To the best of our knowledge, Conjecture \ref{conj.Teich} is only known for the Lebesgue measure class. We also note that many subgroups of $\Mod(\Sigma)$ have limit sets with Lebesgue measure zero (e.g. handlebody groups \cite{Masur_handlebody,Kerckhoff_handlebody}), which automatically implies that the stationary measure is singular to the Lebesgue measure class. See (\cite{Masur_interval}, \cite[Sect. 3.3]{KL_subgroups}) for more discussion on the nullity of limit sets.

As an application of our generalization of Tukia's theorem, we prove Conjecture \ref{conj.Teich} for a certain class of subgroups of $\Mod(\Sigma)$, showing the singularity of $\mathsf{m}$-stationary measure and \emph{any} Busemann Patterson--Sullivan measure.  Before presenting the theorem, we first define Patterson--Sullivan measures in this context.

Gardiner--Masur \cite{GM_boundary} introduced another compactification by $\partial_{GM} \Tc$, called \emph{Gardiner--Masur boundary} of $\Tc$, and proved that $\PMF$ is a proper subset of $\partial_{GM} \Tc$.
Liu--Su  \cite{LS_horofunction} showed that $\partial_{GM} \Tc$ is the horofunction boundary of $(\Tc, \dist_{\Tc})$. Hence, after fixing $o \in \Tc$, one can define a cocycle $\beta : \Mod(\Sigma) \times \partial_{GM} \Tc \to \Rb$ by
\begin{equation*} 
   \beta(g, x) = \lim_{p \rightarrow x} \, \dist_{\Tc}(p, g^{-1} o) - \dist_{\Tc}(p,o)
\end{equation*}
where $p \in \Tc$ converges to $x \in \partial_{GM}\Tc$.
A \emph{Busemann Patterson--Sullivan measure} on $\partial_{GM} \Tc$ is a $\beta$-Patterson--Sullivan measure in the sense of Equation~\eqref{eqn.coarse PS meaasure intro}. These measures have been constructed and studied by several authors, including Coulon \cite{Coulon_PS} and Yang \cite{Yang_conformal}.

We also note that Athreya--Bufetov--Eskin--Mirzakhani \cite{ABEM_MCG} constructed a Patterson--Sullivan measure for $\Mod(\Sigma)$ on $\PMF$ using Thurston measure, and Gekhtman \cite{Gekhtman_CC} constructed Patterson--Sullivan measures for convex cocompact subgroups of $\Mod(\Sigma)$ on $\UE$. Since  the identity map $\Tc \to \Tc$ continuously extends to a topological embedding $\UE \hookrightarrow \partial_{GM} \Tc$ \cite{Miyachi_UE}, the Patterson--Sullivan measures in~\cite{Gekhtman_CC} are Patterson--Sullivan measures on $\partial_{GM} \Tc$. Further, by works of Masur \cite{Masur_interval} and Veech \cite{Veech_Gauss}, the Patterson--Sullivan measure constructed in \cite{ABEM_MCG} gives a full measure on $\UE$, and therefore can be identified with a Busemann Patterson--Sullivan measure on $\partial_{GM}\Tc$. 

Finally, since the $\mathsf{m}$-stationary measure $\nu$ also gives a full measure on $\UE$, we can view $\nu$ as a measure on $\partial_{GM} \Tc$. Moreover, any measure on $\PMF$ is non-singular to $\nu$ on $\PMF$ if and only if its restriction on $\UE$ is non-singular to $\nu$ viewed as measures on $\partial_{GM} \Tc$.

We now state our  contribution towards Conjecture \ref{conj.Teich}.

\begin{theorem}[see Corollary \ref{cor.singularity hitting PS MCG multitwist} below] \label{thm.singularity hitting PS MCG multitwist intro} Suppose  $\Gamma$ is relatively hyperbolic (as an abstract group) and $\mathsf{m}$ has finite superexponential moment. If $\Gamma$ contains a multitwist, then the $\mathsf{m}$-stationary measure $\nu$ is singular to every Busemann Patterson--Sullivan measures on $\partial_{GM} \Tc$. 
\end{theorem}

As explained above, Theorem \ref{thm.singularity hitting PS MCG multitwist intro} implies the same statement for Patterson--Sullivan measures on $\PMF$, such as the measures constructed in \cite{ABEM_MCG,Gekhtman_CC}. 
Note also that Patterson--Sullivan measures under consideration do not have any assumptions on their supports. We also remark that in Theorem \ref{thm.singularity hitting PS MCG multitwist intro}, the multitwist in $\Ga$ does not necessarily belong to a peripheral subgroup of $\Ga$.

There are many examples of subgroups of $\Mod(\Sigma)$ which are relatively hyperbolic and containing multitwists, so Theorem \ref{thm.singularity hitting PS MCG multitwist intro} applies to. For instance, the combination theorem for Veech subgroups by Leininger--Reid \cite{LR_combination} produces closed surface subgroups in $\Mod(\Sigma)$ with multitwists, and so-called parabolically geometrically finite subgroups introduced by Dowdall--Durham--Leininger--Sisto \cite{DDLS_PGF} are relatively hyperbolic and contain multitwists in their peripheral subgroups. Many examples of parabolically geometrically finite subgroups were also constructed by  Udall \cite{Udall_PGF}, Aougab et al.~\cite{aougab2025constructing}, and Loa~\cite{Loa_PGF}. Finally, in their proof of the purely pseudo-Anosov surface subgroup conjecture, Kent--Leininger \cite{KL_atoroidal} constructed a type-preserving homomorphism from a finite index subgroup of the fundamental group of the figure-8 knot complement into $\Mod(\Sigma)$ when $\Sigma$ has genus at least 4. The image of such a homomorphism is relatively hyperbolic and contains a multitwist.

Theorem \ref{thm.singularity hitting PS MCG multitwist intro} will be a consequence of the following.

   \begin{theorem}[see Theorem \ref{thm.RW Teich body} below] \label{thm.singularity hitting PS MCG intro} Suppose  $\Gamma$ is relatively hyperbolic (as an abstract group), $\mathsf{m}$ has a finite superexponential moment with the $\mathsf{m}$-stationary measure $\nu$, and $\mu$ is a Busemann Patterson--Sullivan measure for $\Ga$ on $\partial_{GM} \Tc$ of dimension $\delta$. If the measures $\nu$ and $\mu$ are not singular, then:
      \begin{enumerate}
      \item For any $o \in \Tc$, $$\sup_{\gamma \in \Gamma} \abs{ \dist_G(\id, \gamma) - \delta \dist_{\Tc}(o, \ga o)} < +\infty.$$ In particular, $\delta$ is the critical exponent of $\Ga$ and $\sum_{\ga \in \Ga} e^{-\delta \dist_{\Tc}(o, \ga o)} = + \infty$.
      \item If   $\dist_w$ is a word metric on $\Gamma$ with respect to a finite generating set, then  the map 
      $$
      \gamma \in (\Gamma, \dist_w) \mapsto \gamma o \in (\Tc, \dist_{\Tc})
      $$
      is a quasi-isometric embedding.
      \end{enumerate} 
      \end{theorem} 

For parabolically geometrically finite subgroups, we will also prove the converse of Theorem~\ref{thm.singularity hitting PS MCG intro}, see Theorem  \ref{thm.RW Teich PGF} below.

\subsubsection{Random walks on discrete subgroups of Lie groups} \label{subsec.singularity Liegp}

Let $\Gsf$ be a connected semisimple Lie group without compact factors and with finite center. Suppose $\Gamma < \Gsf$ is a Zariski dense discrete subgroup, and $\mathsf{m}$ is a probability measure on $\Gamma$ whose support generates $\Gamma$ as a semigroup. Fix a minimal parabolic subgroup $\Psf$ and let $\Fc:=\Gsf/\Psf$ denote the Furstenberg boundary. Then there is a unique \emph{$\mathsf{m}$-stationary measure} $\nu$ on $\Fc$ \cite{Furstenberg_boundary, GR_Furstenberg}. The measure $\nu$ is also referred to as the \emph{Furstenberg measure}.

 In this section we consider the following well-known conjecture, cf. \cite{KP_matrix}.

\begin{conjecture}[Singularity conjecture] \label{conj.singularity Liegp} If $\mathsf{m}$ has finite support, then the $\mathsf{m}$-stationary measure $\nu$ is singular to the Lebesgue measure class on $\Fc$. 
\end{conjecture}

In \cite{KZ2025}, we give an affirmative answer to the singularity conjecture when $\Gsf$ has Kazhdan's property (T) and $\Ga$ is not a lattice. In this case, it is not necessary to assume any moment condition on $\mathsf{m}$, and it suffices to have that $\supp \mathsf{m}$ generates $\Ga$ as a group, not necessarily as a semigroup.

In this paper,  we consider the singularity conjecture for a more general class of measures, the ``Iwasawa Patterson--Sullivan measures'' introduced by Quint~\cite{Quint_PS}, and for general semisimple Lie groups.

Delaying precise definitions until Section \ref{sec.Liegps}, we fix a Cartan decomposition $\mathfrak{g} = \mathfrak{k} + \mathfrak{p}$ 
of the Lie algebra of $\Gsf$, a Cartan subspace $\mfa \subset \mfp$, and a positive closed Weyl chamber $\mathfrak{a}^+\subset\mathfrak a$. Then let $\Delta \subset \mathfrak{a}^*$ be the corresponding system of simple restricted roots, and let $\kappa:\Gsf\to\mathfrak{a}^+$ denote the associated Cartan projection.

For the usage in later sections, we consider general flag manifolds.
Given a non-empty subset $\theta \subset \Delta$, we let $\Psf_\theta < \Gsf$ denote the associated parabolic subgroup and let $\Fc_\theta=\Gsf / \Psf_\theta$ denote the associated partial flag manifold. We denote by $B_{\theta} : \Gsf \times \Fc_{\theta} \to \fa_{\theta}$ the \emph{partial Iwasawa cocycle}, a vector valued cocycle whose image lies in a subspace $\fa_{\theta} \subset \fa$ associated to $\theta$.

Given a functional $\phi \in \mfa_\theta^*$ and a subgroup $\Gamma < \Gsf$, a Borel probability measure $\mu$ on $\Fc_\theta$ is called a \emph{coarse $\phi$-Patterson--Sullivan measure for $\Gamma$} if it is a  coarse  $(\phi \circ B_\theta)$-Patterson--Sullivan measure for $\Gamma$ in the sense of Equation~\eqref{eqn.coarse PS meaasure intro}. We refer to these measures as \emph{coarse Iwasawa Patterson--Sullivan measures}.

In the case when $\Gsf = \Isom_0(\Hb^n)$, $\Delta =\{\alpha\}$ consists of a single simple restricted root  and $\Fc_{\alpha}$ naturally identifies with $\partial_\infty\Hb^n$. Employing the ball model for $\Hb^n$ with $o \in \Hb^n$ as the center of the ball so that $\partial_{\infty} \Hb^n = \Sb^{n-1}$, 
$$
| g'(x) |_{\partial_{\infty} \Hb^n} = e^{-(\alpha \circ B_{\alpha})(g, x)}
$$
for all $g \in \Gsf$ and $x \in \partial_{\infty} \Hb^n$. So the above definitions encompasses the classical case described  in Equation \eqref{eqn.conformalmeasure}.

As $\Fc = \Fc_{\Delta}$ always supports a Iwasawa Patterson--Sullivan measure in the Lebesgue measure class \cite[Lem. 6.3]{Quint_PS},  it is natural to consider the following generalization of Conjecture \ref{conj.singularity Liegp}.

\begin{conjecture}[generalized Singularity conjecture] \label{conj.generalized Singularity} If $\mathsf{m}$ has finite support, then the $\mathsf{m}$-stationary measure $\nu$ is singular to every coarse Iwasawa Patterson--Sullivan measure on $\Fc$. 
\end{conjecture} 

We prove that non-singularity implies strong restrictions on how a discrete subgroup embeds in $\Gsf$.

\begin{theorem}[see Theorem~\ref{thm:random walks Zdense} below]\label{thm:random walks Zdense in intro} Suppose  $\Gamma$ is relatively hyperbolic (as an abstract group), $\mathsf{m}$ has finite superexponential moment, and $\mu$ is a coarse $\phi$-Patterson--Sullivan measure on $\Fc$ of dimension $\delta$. If the measures $\nu$ and $\mu$ are not singular, then: 
\begin{enumerate}
\item $\sup_{\gamma \in \Gamma} \abs{ \dist_G(\id, \gamma) - \delta \phi(\kappa(\gamma))} < +\infty$. In particular,  $$\sum_{\ga \in \Ga} e^{-\delta \phi(\kappa(\gamma))} =~+ \infty$$ and $\delta \phi \in \fa^*$ is tangent to the growth indicator of $\Ga$.
\item If  $\dist_w$ is a word metric on $\Gamma$ with respect to a finite generating set, $(X,\dist_X)$ is the symmetric space associated to $\Gsf$, and $x_0 \in X$, then  the map 
$$
\gamma \in (\Gamma, \dist_w) \mapsto \gamma x_0 \in (X,\dist_X)
$$
is a quasi-isometric embedding.
\end{enumerate} 
\end{theorem}

For some classes of groups, it is easy to verify that the map in part (2) cannot be a quasi-isometric embedding.

\begin{corollary}[see Corollary \ref{cor:RW in G wh and unipotent body} below] \label{cor:RW in G wh and unipotent}  Suppose  $\Gamma$ is word hyperbolic (as an abstract group) and $\mathsf{m}$ has finite superexponential moment. If $\Gamma$ contains a unipotent element of $\Gsf$, then the $\mathsf{m}$-stationary measure  $\nu$ is singular to every coarse Iwasawa Patterson--Sullivan measure on $\Fc$. \end{corollary}  

More generally, Corollary~\ref{cor:RW in G wh and unipotent} holds when $\Gamma$ is relatively hyperbolic and contains an element $u$ which is unipotent (as an element of $\Gsf$) and the stable translation length of $u$ is positive on a Cayley graph of $\Ga$ (e.g. $u$ is loxodromic \cite[Prop. 7.8]{DG_entropy}).

\subsection{Tukia's measurable boundary rigidity theorem} 

In this section we describe special cases of our main theorem in a variety of settings.

\subsubsection{Tukia's theorem for word hyperbolic groups} 

We establish a version of Tukia's theorem for word metrics on word hyperbolic groups, which implies that any measurable isomorphism between Gromov boundaries with respect to coarse Patterson--Sullivan measures always extends to a \emph{homeomorphism}.

\begin{theorem}[see Theorem~\ref{thm:GH case meas=>homeo} below]\label{thm:tukia for word hyperbolic in intro} For $i=1,2$ suppose $\Gamma_i$ is a non-elementary word hyperbolic group endowed with a word metric $\dist_i$ with respect to a finite generating set and $\mu_i$ is a coarse Busemann Patterson--Sullivan measure for $\Gamma_i$ of dimension $\delta_i$ on $\partial_{\infty} \Ga_i$. Assume there exist
\begin{itemize}
\item a homomorphism $\rho : \Gamma_1 \rightarrow \Gamma_2$ with non-elementary image and 
\item a $\mu_1$-almost everywhere defined measurable $\rho$-equivariant injective map $f : \partial_\infty \Gamma_1 \rightarrow \partial_\infty \Gamma_2$.
\end{itemize} 
If $f_* \mu_1$ and $\mu_2$ are not singular, then $\ker \rho$ is finite, $\rho(\Gamma_1) < \Gamma_2$ has finite index, 
$$
\sup_{\gamma_1,\gamma_2 \in \Gamma_1} \abs{\delta_1\dist_1(\gamma_1,\gamma_2) - \delta_2\dist_2(\rho(\gamma_1), \rho(\gamma_2))} < +\infty,
$$
and there exists a $\rho$-equivariant homeomorphism $\tilde f : \partial_\infty \Gamma_1 \rightarrow \partial_\infty \Gamma_2$ such that 
\begin{enumerate}
\item $\tilde f = f$ $\mu_1$-a.e., 
\item $\tilde f_* \mu_1$, $\mu_2$ are in the same measure class and the Radon--Nikodym derivatives are a.e. bounded from above and below by a positive number.
\end{enumerate} 
\end{theorem} 

In fact we prove Theorem \ref{thm:tukia for word hyperbolic in intro} for Patterson--Sullivan measures associated to a more general class of cocycles introduced in \cite{BCZZ2024a}, see Definition \ref{defn.expanding coarse cocycle} and Theorem~\ref{thm:GH case meas=>homeo}.

\begin{remark} \label{rmk.convergence group bdr map}
   Given two minimal convergence group actions $\Ga_1 \curvearrowright M_1$ and $\Ga_2 \curvearrowright M_2$ and an onto homomorphism $\rho : \Ga_1 \to \Ga_2$, it is known that any continuous $\rho$-equivariant map $f : M_1 \to M_2$ is injective on the so-called Myrberg limit set of $\Ga_1$ \cite[Prop. 7.5.2]{Gerasimov_floyd} (see also \cite[Lem. 10.5]{Yang_conformal}). Moreover, for a word hyperbolic group, the Myrberg limit set on its Gromov boundary is of full measure with respect to any coarse Busemann Patterson--Sullivan measure \cite[Thm. 1.14]{Yang_conformal} (see also \cite[Cor. 7.3]{Coornaert_PS}). Hence, any continuous equivariant maps between Gromov boundaries of word hyperbolic groups satisfies the condition in Theorem \ref{thm:tukia for word hyperbolic in intro}.
\end{remark}

\subsubsection{Tukia's theorem for Teichm\"uller spaces}

We establish a version of Tukia's theorem for Teichm\"uller spaces. 

\begin{theorem}[corollary to Theorems~\ref{thm:main rigidity theorem in intro} and~\ref{thm.Teichmueller is wellbehavedPS}]\label{thm.Tukia Teichmueller intro} For $i = 1, 2$, let $\Sigma_i$ be a closed connected orientable surface of genus at least two and $\Tc_i$ its Teichm\"uller space. Let $\Ga_i < \Mod(\Sigma_i)$ be a non-elementary subgroup and $\mu_i$ a Busemann Patterson--Sullivan measure for $\Ga_i$ of dimension $\delta_i$ on $\partial_{GM} \Tc_i$. Suppose
   \begin{itemize}
      \item $\displaystyle \sum_{\ga \in \Ga_1} e^{-\delta_1 \dist_{\Tc_1}(o_1, \ga o_1)} = + \infty$ for $o_1 \in \Tc_1$.
      \item There exists an onto homomorphism $\rho : \Ga_1 \to \Ga_2$ and a $\mu_1$-almost everywhere defined measurable $\rho$-equivariant injective map $f : \partial_{GM} \Tc_1 \to \partial_{GM} \Tc_2$.
   \end{itemize}
   If $f_*\mu_1$ and $\mu_2$ are not singular, then for any $o_2 \in \Tc_2$, the orbit map $\ga o_1 \mapsto \rho(\ga) o_2$ is a rough isometry after scaling, i.e.,
   $$
   \sup_{\ga \in \Ga_1} \abs{ \delta_1 \dist_{\Tc_1}(o_1, \ga o_1) - \delta_2 \dist_{\Tc_2}(o_2, \rho(\ga) o_2)} < + \infty.
   $$
\end{theorem}

\begin{remark}
   As shown by Yang \cite{Yang_conformal}, $\sum_{\ga \in \Ga_1} e^{-\delta_1 \dist_{\Tc_1}(o_1, \ga o_1)} = + \infty$ implies that $\UE(\Sigma_1) \subset \partial_{GM}\Tc_1$ has a full $\mu_1$-measure. Hence,  the boundary map $f$ and measure $\mu_1$ can be regarded to be defined on $\PMF(\Sigma_1)$, i.e. Thurston's boundary.
\end{remark}
 
For a convex cocompact $\Ga < \Mod(\Sigma)$, there exists a unique $\Ga$-minimal subset of $\PMF$, called the limit set of $\Ga$, and is the image of a $\Ga$-equivariant embedding of $\partial_{\infty} \Ga$ into $\UE$ \cite[Prop. 3.2]{FM_CC}. Moreover, if $\mu$ is a Patterson--Sullivan measure for $\Ga$ of dimension $\delta$ and $\sum_{\ga \in \Ga} e^{-\delta \dist_{\Tc}(o, \ga o)} = + \infty$, then $\mu$ is supported on the limit set of $\Ga$ \cite{Gekhtman_CC} (see also \cite{Coulon_PS, Yang_conformal}). Hence, the boundary map $f$ as in Theorem \ref{thm.Tukia Teichmueller intro} always exists for two isomorphic convex cocompact subgroups. See also Remark \ref{rmk.convergence group bdr map}.

\subsubsection{Tukia's theorem in higher rank} 

Using the Iwasawa Patterson--Sullivan measures introduced in Section~\ref{subsec.singularity Liegp}, we extend Tukia's theorem to a class of discrete subgroups in higher rank semisimple Lie groups called transverse groups, which can be viewed as a higher rank analogue of Kleinian groups. This class  is defined in Section \ref{sec.Liegps} and includes the Anosov and relatively Anosov subgroups and their subgroups. Further, any discrete subgroup of a rank one non-compact simple Lie group is transverse.

\begin{theorem}[see Corollary \ref{cor.Tukia in higher rank} below] \label{thm:Tukia in higher rank} 
   Let $\Gsf_1, \Gsf_2$ be  non-compact simple Lie groups with trivial centers. Let $\Ga < \Gsf_1$ be a Zariski dense $\Psf_{\theta_1}$-transverse subgroup, $\mu$ a coarse $\phi$-Patterson--Sullivan measure for $\Ga$ of dimension $\delta \ge 0$ on $\F_{\theta_1}$, and $\rho : \Ga \to \Gsf_2$ a representation with Zariski dense image. Suppose
   \begin{itemize}
      \item $\displaystyle \sum_{\ga \in \Ga} e^{-\delta \phi(\kappa(\ga))} = + \infty$.
      \item There exists a $\mu$-almost everywhere defined measurable $\rho$-equivariant injective map $f : \Fc_{\theta_1} \rightarrow \Fc_{\theta_2}$.
   \end{itemize}
   If $f_*\mu$ is not singular to some coarse Iwasawa Patterson--Sullivan measure for $\rho(\Ga)$, then $\rho$ extends to a Lie group isomorphism $\Gsf_1 \rightarrow \Gsf_2$. 
   \end{theorem}

\begin{remark}
   \
   \begin{enumerate}
      \item As in Margulis' superrigidity theorem, the representation $\rho$ is not assumed to be discrete in Theorem \ref{thm:Tukia in higher rank}, in contrast to Theorem \ref{thm:Tukia} and Yue's generalization \cite{Yue1996}. 
      
      \item Theorem \ref{thm:Tukia in higher rank} follows from a more general statement  (Corollary \ref{cor.Tukia in higher rank}) about  a non-elementary transverse subgroup of a semisimple Lie group and its irreducible representation into a semisimple Lie group. 
      \item See Remark \ref{rmk.directional} for a version of the theorem for non-transverse Zariski dense discrete subgroups. 
   \end{enumerate}
\end{remark}

\begin{remark} \label{rmk.prevworks} Theorem~\ref{thm:Tukia in higher rank} was previously established in a variety of special cases. In all of these previous works, the representation $\rho$ was assumed to be discrete faithful and the boundary map was assumed to be a topological embedding. 
   \begin{itemize}
      \item Kim--Oh \cite{KO_ConfMR} considered the cases when either 
      \begin{enumerate}
      \item $\Gsf_1$ is rank one, $\rho$ is faithful, and $\rho(\Ga)$ is $\Psf_{\Delta_2}$-divergent.
      \item $\Ga$ is $\Psf_{\Delta_1}$-Anosov, $\rho$ is faithful, and $\rho(\Gamma)$ is $\Psf_{\Delta_2}$-Anosov.
      \end{enumerate} 
      \item Kim \cite{Kim_ConfMR} considered the case where $\Gamma$ is $\Psf_{\theta_1}$-hypertransverse ($\Psf_{\theta_1}$-transverse with an extra assumption), $\rho$ is faithful, and $\rho(\Ga)$ is $\Psf_{\theta_2}$-divergent.
      \item Blayac--Canary--Zhu--Zimmer \cite{BCZZ2024a} considered the case where $\Ga$ is $\Psf_{\theta_1}$-transverse, $\rho$ is faithful, and $\rho(\Gamma)$ is $\Psf_{\theta_2}$-transverse. 
         \end{itemize}
In contrast to these previous works, in Theorem~\ref{thm:Tukia in higher rank}, $\rho$ does not need to be discrete or faithful, and the boundary map \emph{does not even need to be continuous}.  Further, in many natural settings the boundary maps will not be a topological embedding (e.g. Cannon--Thurston maps \cite{CT_CannonThurston_map}), continuous, or even defined everywhere (e.g. maps between limit sets of isomorphic geometrically finite groups \cite{Tukia_limit_map}).
\end{remark}

\subsection{Entropy rigidity for Anosov groups with Lipschitz limit sets}

The \emph{Hilbert critical exponent} of a discrete subgroup $\Gamma < \SL(d,\Rb)$ is 
$$
\delta_{\rm Hil}(\Gamma) : = \limsup_{R \rightarrow \infty} \frac{1}{R} \log \# \left\{ \gamma \in \Gamma : \frac{1}{2} \log \norm{\gamma} \norm{\gamma^{-1}} \leq R\right\}. 
$$
Notice that if $\Gamma$ preserves a properly convex domain $\Omega \subset \Pb(\Rb^d)$, then $\delta_{\rm Hil}(\Gamma)$ coincides with the critical exponent of $\Gamma$ with respect to the Hilbert metric on $\Omega$ (this follows from \cite[Proposition 10.1]{DGK_cc}).

In a companion paper, we use Corollary \ref{cor.Tukia in higher rank} (the general version of Theorem~\ref{thm:Tukia in higher rank}) to prove a rigidity result for this quantity for Anosov groups with Lipschitz limit sets. We briefly and informally describe this result here, for more details see ~\cite{KimZimmer_Furstenberg}. Pozzetti--Sambarino--Wienhard \cite{PSW2023} proved that 
$$
\delta_{\rm Hil}(\Gamma) \leq p-1
$$
when $\Gamma < \SL(d,\Rb)$ is projective Anosov whose projective limit set in $\Pb(\Rb^d)$  is a $p$-dimensional Lipschitz  manifold and $\Gamma$ acts irreducibly on $\Rb^d$ and $\wedge^{p+1} \Rb^d$, for some $p \le d-2$. Under this hypothesis, in~\cite[Theorem 1.13]{KimZimmer_Furstenberg} we prove that 
  $$
\delta_{\rm Hil}(\Gamma) = p-1
$$
if and only if $p=d-2$ and up to conjugation $\Gamma$ is a uniform lattice in $\mathsf{SO}(d-1,1)$. 

This generalizes a result of Crampon~\cite{Crampon2009} who proved that if $\Gamma < \SL(d,\Rb)$ is a discrete group which acts cocompactly on a strictly convex domain $\Omega \subset \Pb(\Rb^d)$, then 
$$
\delta_{\rm Hil}(\Gamma) \leq d-2 
$$
with equality if and only if $\Omega$ is an ellipsoid (and hence up to conjugation $\Gamma$ is a uniform lattice in $\mathsf{SO}(d-1,1)$). Under Crampon's hypothesis, $\Gamma$ is $\Psf_1$-Anosov and acts irreducibly on $\Rb^d$ and hence also $\wedge^{d-1} \Rb^d$, see~\cite[Section 6.2]{GW2012}. Further, $\Lambda_1(\Gamma)$ coincides with $\partial\Omega$ and hence is a Lipschitz $(d-2)$-manifold. 

In the proof of~\cite[Theorem 1.13]{KimZimmer_Furstenberg} we construct a measurable boundary map from the projective limit set into the partial flag manifold $\Fc_{1,p+1,d-p-1, d-1}(\Rb^d)$ and then show that two Patterson--Sullivan measures are non-singular. Then Corollary \ref{cor.Tukia in higher rank} constrains the Jordan projection and we use a result of Benoist~\cite{Benoist_properties} to determine the Zariski closure. 

\subsection{Patterson--Sullivan systems}\label{sec:general framework}

We now define Patterson--Sullivan systems and then state our generalization of Tukia's theorem.  In the classical setting of real hyperbolic geometry, ``geodesic shadows'' play a fundamental role in the study of  Patterson--Sullivan measures and our definition of  Patterson--Sullivan systems attempts to extract the key properties of these sets.

As in the beginning of the introduction, let $M$ be a compact metric space and let $\Gamma < \mathsf{Homeo}(M)$ be a subgroup. Recall that coarse-cocycles and coarse Patterson--Sullivan measures  were introduced in Equations \eqref{eqn.coarsecocycle intro} and \eqref{eqn.coarse PS meaasure intro}.

\begin{definition}\label{defn:PS systems}
A \emph{Patterson--Sullivan-system (PS-system) of dimension $\delta$} consists of
\begin{itemize}
\item a coarse-cocycle $\sigma : \Gamma \times M \rightarrow \Rb$, 
\item coarse $\sigma$-Patterson--Sullivan measure (PS-measure) $\mu$ of dimension $\delta$,
\item for each $\gamma \in \Gamma$, a number $\norm{\gamma}_\sigma \in \R$ called the \emph{$\sigma$-magnitude of $\gamma$}, and
\item for each $\gamma \in \Gamma$ and $R > 0$, a non-empty open set $\Oc_R(\gamma) \subset M$ called the \emph{$R$-shadow of $\gamma$}
\end{itemize}
such that:
\begin{enumerate}[label=(PS\arabic*)]
\item\label{item:coycles are bounded} For any $\ga \in \Ga$, there exists $c=c(\gamma) > 0$ such that $\abs{\sigma(\ga, x)}\leq  c(\gamma)$ for  $\mu$-a.e. $x\in M$. 
\item\label{item:almost constant on shadows} For every $R> 0$ there is a constant $C=C(R) > 0$ such that
$$
 \norm{\gamma}_\sigma - C  \leq \sigma(\gamma,x) \leq \norm{\gamma}_\sigma + C 
$$
for all $\gamma \in \Gamma$ and $\mu$-a.e. $x \in \gamma^{-1} \Oc_R(\gamma)$.
\item\label{item:empty Z intersection} If $\{\gamma_n\} \subset \Gamma$, $R_n \rightarrow +\infty$, $Z \subset M$ is compact, and $[M \smallsetminus \gamma_n^{-1}\Oc_{R_n}(\gamma_n)] \rightarrow Z$ with respect to the  Hausdorff distance, then for any $x \in Z$, there exists $g \in \Ga$ such that 
$$gx \notin Z.$$
\end{enumerate} 
We call the PS-system \emph{well-behaved} with respect to a collection
$$
\mathscr{H} := \{ \mathscr{H}(R) \subset \Ga : R \ge 0 \}
$$
 of non-increasing subsets of $\Ga$ if the following additional properties hold:
\begin{enumerate}[label=(PS\arabic*)]
\setcounter{enumi}{3}
\item\label{item:properness} $\Gamma$ is countable and for any $T > 0$, the set $\{ \gamma \in \mathscr{H}(0) : \norm{\gamma}_\sigma \leq T\}$ is finite. 

\item\label{item:baire} If $\{\gamma_n\} \subset \Gamma$, $R_n \rightarrow +\infty$, $Z \subset M$ is compact, and $[M \smallsetminus \gamma_n^{-1}\Oc_{R_n}(\gamma_n)] \rightarrow Z$ with respect to the  Hausdorff distance, then for any $h_1, \ldots, h_m \in \Ga$ and $x \in Z$, there exists $g \in \Ga$ such that
$$
g x \notin \bigcup_{i = 1}^m h_i Z.
$$

\item\label{item:shadow inclusion} If $R_1 \leq R_2$ and $\gamma \in \mathscr{H}(0)$, then $\Oc_{R_1}(\gamma) \subset \Oc_{R_2}(\gamma)$. 

\item\label{item:intersecting shadows} For any $R > 0$ there exist $C>0$ and $R'> 0$ such that: if $\alpha, \beta \in \mathscr{H}(R)$, $\norm{\alpha}_\sigma \leq \norm{\beta}_\sigma$, and $\Oc_R(\alpha) \cap \Oc_R(\beta) \neq \emptyset$, then 
$$
\Oc_R(\beta) \subset \Oc_{R'}(\alpha)
$$ 
and  
$$
\abs{\norm{\beta}_\sigma - (\norm{\alpha}_\sigma + \norm{\alpha^{-1}\beta}_\sigma)} \leq C. 
$$
\item \label{item:diam goes to zero ae} For every $R > 0$, there exists a set $M' \subset M$ of full $\mu$-measure such that 
$$
\lim_{n \rightarrow\infty} {\rm diam} \Oc_R(\gamma_n) = 0
$$
whenever $\{\gamma_n\} \subset \mathscr{H}(R)$ is an escaping sequence and 
$$
x \in M' \cap  \bigcap_{n =1}^\infty \Oc_R(\gamma_n).
$$
\end{enumerate} 
We call the collection $\mathscr{H}$ the \emph{hierarchy} of the Patterson--Sullivan system. 
\end{definition} 

\begin{remark} Property \ref{item:empty Z intersection} and the stronger Property  \ref{item:baire} can be viewed as saying the action of $\Gamma$ on $Z$ is ``irreducible'' and ``strongly irreducible'' respectively. 

\end{remark}

For a well-behaved Patterson--Sullivan system $(M, \Ga, \sigma, \mu)$ with respect to a hierarchy $\mathscr{H} = \{\mathscr{H}(R) \subset \Ga : R \ge 0\}$, we  define the $\mathscr{H}$-conical limit set:
\begin{equation} \label{eqn.conical hierarchy}
\La^{\rm con}(\mathscr{H}) := \left\{x \in M : \begin{matrix}
\exists \ R > 0, \ \ga \in \Ga, \ \text{an escaping sequence } \{ \ga_n \in \mathscr{H}(n) \} \\
\text{s.t. } x \in \ga \Oc_R(\ga_n) \text{ for all } n \ge 1
\end{matrix}
 \right\}.
\end{equation}
We now state our generalization of Tukia's rigidity theorem (Theorem \ref{thm:Tukia}) to PS-systems.

\begin{theorem}[see Theorem~\ref{thm:main rigidity theorem} below] \label{thm:main rigidity theorem in intro}
 Suppose 
\begin{itemize} 
\item $(M_1, \Gamma_1, \sigma_1, \mu_1)$ is a well-behaved PS-system of dimension $\delta_1$ with respect to a hierarchy $\mathscr{H}_1 = \{ \mathscr{H}_1(R) \subset \Ga_1  : R \ge 0 \}$ and
$$\mu_1(\La^{\rm con}(\mathscr{H}_1)) = 1.$$
\item $(M_2, \Gamma_2, \sigma_2, \mu_2)$ is a  PS-system of dimension $\delta_2$.
\item There exists an onto homomorphism $\rho : \Gamma_1 \rightarrow \Gamma_2$ and a $\mu_1$-a.e. defined measurable $\rho$-equivariant injective map $f : M_1 \rightarrow M_2$.
\end{itemize} 
If the measures $f_*\mu_1$ and $\mu_2$ are not singular, then 
$$
\sup_{\gamma \in \Gamma_1} \abs{\delta_1 \norm{\gamma}_{\sigma_1} - \delta_2 \norm{\rho(\gamma)}_{\sigma_2}} < +\infty.
$$
\end{theorem}

\begin{remark} \label{rmk.divergence implies full} Although formulated differently, Theorem~\ref{thm:main rigidity theorem in intro} contains Tukia's theorem as a special case. Under the hypothesis of Theorem~\ref{thm:Tukia}, the Patterson--Sullivan measures $\mu_i$ are part of a well-behaved PS-system with respect to a trivial hierarchy $\mathscr{H}_i(R) \equiv \Gamma_i$ and with magnitude function 
$$
\gamma \mapsto \dist_{\Hb^{n_i}}(o_i, \gamma o_i)
$$
where $o_i \in \Hb^{n_i}$ is a basepoint. Further, the conical limit set defined in Equation~\eqref{eqn.conical hierarchy} coincides with the classical conical limit set in hyperbolic geometry. The classical Hopf--Tsuji--Sullivan dichotomy then implies that $\mu_1(\Lambda^{\rm con}(\mathscr{H}_1)) = 1$ and hence Theorem \ref{thm:main rigidity theorem in intro} implies that 
$$
\sup_{\gamma \in \Gamma_1} \abs{ \delta_1\dist_{\Hb^{n_1}}( o_1, \gamma o_1) - \delta_2\dist_{\Hb^{n_2}}( o_2, \rho(\gamma) o_2)} < +\infty.
$$
It then follows from marked length spectrum rigidity that $n_1=n_2$ and $\rho$ extends to an isomorphism $\Isom(\Hb^{n_1}) \rightarrow \Isom(\Hb^{n_2})$, as in Theorem \ref{thm:Tukia}. Similarly, Theorem \ref{thm:tukia for word hyperbolic in intro}, Theorem \ref{thm.Tukia Teichmueller intro}, and Theorem \ref{thm:Tukia in higher rank} are consequences of Theorem \ref{thm:main rigidity theorem in intro}.
\end{remark}

\begin{example}[PS-systems]\label{ex:PS systems} Our abstract setting encompasses the following: 
\begin{enumerate}
\item  Stationary measures on the Bowditch boundary of a relatively hyperbolic group associated to random walks with finite superexponential moments are contained in  well-behaved PS-systems (see Section~\ref{sec.randomwalks}).  
\item Coarse Busemann PS-measures on the Gromov boundary of a proper geodesic Gromov hyperbolic space are contained in well-behaved PS-systems. More generally, coarse PS-measures associated to expanding coarse-cocycles (introduced in  \cite{BCZZ2024a}) are contained in well-behaved PS-systems (see Section~\ref{sec.convergenceaction}). 
\item Coarse Iwasawa PS-measures on a partial flag manifold associated to Zariski dense subgroups (more generally ``$\Psf_{\theta}$-irreducible'' subgroups)  are always contained in  PS-systems. When the subgroup is transverse and the measure is supported on the limit set, they are  contained in well-behaved PS-systems (see Section~\ref{sec.Liegps}; see also Theorem \ref{thm.directionalps} for general Zariski dense discrete subgroups). 
 \item Busemann PS-measures 
 \begin{itemize}
 \item on the Gardiner--Masur boundary  $\partial_{GM} \Tc$ of Teichm\"uller space for non-elementary subgroups of a mapping class group, 
 \item on the geodesic boundary of a ${\rm CAT}(0)$-space for discrete groups of isometries with rank one elements, 
 \end{itemize} 
  are contained in well-behaved PS-systems (see Section~\ref{sec.contracting} for a general discussion on group actions with contracting isometries). 
   \end{enumerate}

\end{example}

\subsection*{Acknowledgements} 
We would like to thank Harry Bray, Dick Canary, Sebastian Hurtado-Salazar, Yair Minsky, and Hee Oh for valuable conversation. Kim expresses his special gratitude to his Ph.D. advisor Hee Oh for her encouragement and guidance.

Kim thanks the University of Wisconsin--Madison for hospitality during a visit in October 2024 where work on this project started. Zimmer was partially supported by a Sloan research fellowship and grant DMS-2105580 from the National Science Foundation.

\section{Preliminaries}

\subsection{Possibly ambiguous notation/terminology} We briefly define any possible ambiguous notation and terminology. 
\begin{enumerate} 
\item A sequence $\{y_n\}$ in a countable set $Y$  is \emph{escaping} if it eventually leaves every finite set, i.e. if $F \subset Y$ is finite, then $\#\{ n : y_n \in F\}$ is finite. 
\item Any connected semisimple Lie group $\Gsf$ with trivial center is real algebraic \cite[Prop. 3.1.6]{Zimmer_book}. Hence, Zariski density is defined for $\Hsf < \Gsf$, in the sense that no finite index subgroup of $\Hsf$ is contained in a proper connected closed subgroup of $\Gsf$.
\item Given a proper metric space $X$ we endow the isometry group $\Isom(X)$ with the compact open topology. Then a subgroup $\Gamma < \Isom(X)$ is discrete if and only if it is countable and acts properly on $X$. 
\end{enumerate}

\subsection{The Hausdorff distance} Suppose $(M,\dist)$ is a compact metric space. Given a subset $C \subset M$ and $\epsilon > 0$, let $\Nc_\epsilon(C)$ denote the open $\epsilon$-neighborhood of $C$ with respect to $\dist$. The \emph{Hausdorff distance} between two compact subsets $C_1, C_2 \subset M$ is 
$$
\dist^{\rm Haus}(C_1,C_2) : = \inf\{ \epsilon : C_1 \subset \Nc_\epsilon(C_2) \text{ and } C_2 \subset \Nc_\epsilon(C_1)\}. 
$$ 
Notice that for the empty set we have 
$$
\dist^{\rm Haus}(\emptyset, C) = \dist^{\rm Haus}(C,\emptyset)= \begin{cases} 0 & \text{if } C = \emptyset \\ +\infty & \text{if } C \neq \emptyset \end{cases}. 
$$ 
This metric induces a compact topology on the space of compact subsets of $M$ where $C_n \rightarrow C$ if and only if
$$
\lim_{n \rightarrow \infty} \dist^{\rm Haus}(C_n,C) =0.
$$
Notice that the empty set is an isolated point: $C_n \rightarrow \emptyset$ if and only if $C_n = \emptyset$ for all $n$ sufficiently large.

\subsection{Relatively hyperbolic groups} There are several equivalent definitions of relatively hyperbolic groups and we state the definition we use in this paper. 

Suppose $\Gamma < \mathsf{Homeo}(M)$ is a convergence group. 
\begin{itemize} 
\item A point $x \in M$ is a \emph{conical limit point} of $\Ga$ if there are $a,b \in M$ distinct and $\{\gamma_n\} \subset \Gamma$ such that $\gamma_n(x) \rightarrow a$ and $\gamma_n(y) \rightarrow b$ for all $y \in M \smallsetminus \{x\}$. 
\item An element $\gamma \in \Gamma$ is \emph{parabolic} if it has infinite order and fixes exactly one point in $M$.
\item A point $x \in M$ is a \emph{parabolic fixed point} of $\Ga$ if the stabilizer ${\rm Stab}_{\Gamma}(x)$ is infinite and every infinite order element in ${\rm Stab}_{\Gamma}(x)$ is parabolic. A \emph{bounded parabolic fixed point} $x \in M$ is a parabolic fixed point where the quotient ${\rm Stab}_{\Gamma}(x) \backslash (M- \{x\})$ is compact. 
\item $\Gamma$ is a \emph{geometrically finite convergence group} if every point in $M$ is either a conical limit point or a bounded parabolic fixed point of $\Ga$.
\end{itemize} 

\begin{definition}\label{defn:RH}
Given a finitely generated group $\Gamma$ and a collection $\Pc$ of finitely generated infinite subgroups, we say that $(\Gamma,\Pc)$ is \emph{relatively hyperbolic}, if $\Gamma$ acts on a compact perfect metrizable space $M$ as a geometrically finite convergence group and the maximal parabolic subgroups are exactly the set
$$
 \{ \gamma P \gamma^{-1} : P \in \Pc, \gamma \in \Gamma\}.
$$
\end{definition}

Given a relatively hyperbolic group $(\Gamma,\Pc)$, any two compact perfect metrizable spaces satisfying Definition~\ref{defn:RH} are $\Gamma$-equivariantly homeomorphic (see \cite[Thm.\ 9.4]{Bowditch_relhyp}). This unique topological space is then denoted by $\partial(\Gamma,\Pc)$ and called the \emph{Bowditch boundary of $(\Gamma, \Pc)$}. 

\begin{remark} Note that by definition we assume that a relatively hyperbolic group is non-elementary, finitely generated, and has finitely generated peripheral subgroups. \end{remark}

\part{Abstract PS-systems}

\section{Basic properties of PS-systems}

In this section we observe some immediate consequences of the definitions introduced in Section~\ref{sec:general framework}.

\begin{proposition}[Shadow Lemma] \label{prop.shadowlemma} Let $(M, \Ga, \sigma, \mu)$ be a PS-system of dimension $\delta \ge 0$. For any $R > 0$ sufficiently large there exists $C=C(R) > 1$ such that 
$$
\frac{1}{C} e^{-\delta \norm{\gamma}_\sigma} \leq \mu( \Oc_R(\gamma)) \leq C e^{-\delta \norm{\gamma}_\sigma} 
$$
for all $\ga \in \Ga$
and 
\begin{equation*}
\inf_{\gamma \in \Gamma} \mu( \gamma^{-1}\Oc_R(\gamma))>0.
\end{equation*}
\end{proposition} 

\begin{proof} We first show that for any $R > 0$ sufficiently large,
\begin{equation}\label{eqn:translate of shadow has positive mass}
\inf_{\gamma \in \Gamma} \mu( \gamma^{-1}\Oc_R(\gamma))>0.
\end{equation}
Suppose not. Then for every $n \geq 1$ there exists $\gamma_n \in \Gamma$ with 
$$
\mu( \gamma_n^{-1}\Oc_{n}(\gamma_n)) < \frac{1}{n}.
$$

Fix a metric on $M$ which generates the topology. Passing to a subsequence, we can suppose that $M \smallsetminus \gamma_n^{-1}\Oc_{n}(\gamma_n)$ converges to a compact set $Z$ with respect to the  Hausdorff distance (note it is possible for $Z$ to be the empty set, in which case $M \smallsetminus \gamma_n^{-1}\Oc_{n}(\gamma_n)$ is also empty for $n$ sufficiently large). 

Fix $\epsilon > 0$. Then  $M \smallsetminus \gamma_n^{-1}\Oc_{n}(\gamma_n) \subset \Nc_\epsilon(Z)$ for $n$ sufficiently large and hence 
$$
\mu(\Nc_\epsilon(Z)) \geq \lim_{n \rightarrow \infty} \mu(M \smallsetminus \gamma_n^{-1}\Oc_{n}(\gamma_n)) = 1. 
$$
Since $\epsilon > 0$ is arbitrary and $Z$ is closed,  
$$
\mu(Z) = \lim_{n \rightarrow \infty} \mu(\Nc_{1/n}(Z)) =1. 
$$
On the other hand, by Property \ref{item:empty Z intersection}, $\bigcap_{\ga \in \Ga} \ga Z = \emptyset$. Since $M$ is compact, there exist finitely many $\ga_1, \dots, \ga_n \in \Ga$ such that $\bigcap_{i = 1}^n \ga_i Z = \emptyset$, which is a contradiction to $\mu(Z) = 1$ and the $\Ga$-quasi-invariance of $\mu$. Thus Equation~\eqref{eqn:translate of shadow has positive mass} is true for sufficiently large $R > 0$. 

Fix $R > 0$ satisfying Equation~\eqref{eqn:translate of shadow has positive mass} and let $\epsilon_0 : = \inf_{\gamma \in \Gamma} \mu( \gamma^{-1}\Oc_R(\gamma))$. 
Since 
$$
\mu(\Oc_R(\ga)) =  \int_{\gamma^{-1}\Oc_R(\gamma)} \frac{d\gamma^{-1}_*\mu}{d\mu} d\mu,
$$
by Property~\ref{item:almost constant on shadows}, there exists $C = C(R) > 1$ such that
\begin{align*}
\frac{\epsilon_0}{C}e^{-\delta \norm{\gamma}_\sigma} & \le \mu(\Oc_R(\ga))  \le C e^{-\delta \norm{\gamma}_\sigma}. \qedhere
\end{align*}
\end{proof}
  
  We will use the following version of the Vitali covering lemma.

\begin{lemma} \label{lem:V covering} Let $(M, \Ga, \sigma, \mu)$ be a well-behaved PS-system with respect to a hierarchy $\mathscr{H} = \{ \mathscr{H}(R) \subset \Ga : R \ge 0 \}$. Let $R > 0$ and let $R' > 0$ be the constant satisfying Property \ref{item:intersecting shadows} for $R$. Then  for any $I \subset \mathscr{H}(R)$, there exists $J \subset I$ such that 
$$
\bigcup_{\gamma \in I} \Oc_R(\gamma) \subset \bigcup_{\gamma \in J} \Oc_{R'}(\gamma)
$$
and the shadows $\{ \Oc_R(\gamma) : \gamma \in J\}$ are pairwise disjoint. 
\end{lemma} 

\begin{proof} By Property~\ref{item:properness} we can enumerate $I = \{\gamma_n\}$ so that 
$$
\norm{\gamma_1}_\sigma \leq \norm{\gamma_2}_\sigma \leq \norm{\gamma_3}_\sigma \leq \cdots
$$
Now we define indices $j_1 < j_2 < \cdots$ as follows. First let $j_1 = 1$. Then supposing $j_1,\dots, j_{k}$ have been selected, let $j_{k+1}$ be the smallest index greater than $j_k$ such that 
$$
\Oc_R(\gamma_{j_{k+1}}) \cap \bigcup_{i=1}^k \Oc_R(\gamma_{j_{i}}) = \emptyset. 
$$
(This process could terminate after finitely many steps). 

We claim that $J=\{\gamma_{j_k}\}$ has the desired properties. By construction, the shadows $\{ \Oc_R(\gamma) : \gamma \in J\}$ are pairwise disjoint. For any  $\gamma_n \in I \smallsetminus J$, we can pick $k$ such that $j_k < n$ and that $k$ is the maximal index with this property. Since $\gamma_n \notin J$, we must have 
 $$
\Oc_R(\gamma_{n}) \cap \bigcup_{i=1}^k \Oc_R(\gamma_{j_{i}}) \neq \emptyset
$$
and so 
$$
\Oc_R(\gamma_{n}) \subset  \bigcup_{i=1}^k \Oc_{R'}(\gamma_{j_{i}})
$$
by Property \ref{item:intersecting shadows}.
Thus
\begin{equation*}
\bigcup_{\gamma \in I} \Oc_R(\gamma) \subset \bigcup_{\gamma \in J} \Oc_{R'}(\gamma). \qedhere
\end{equation*}
\end{proof}

We will crucially use the following diagonal covering lemma several times in the arguments that follow. It applies in the case when $\Gamma < \mathsf{Homeo}(M_1)$ is part of a well-behaved PS-system and $\rho(\Gamma) < \mathsf{Homeo}(M_2)$ is part of a PS-system.

\begin{lemma}\label{lem.keylemma} Let $M_1,M_2$ be compact metrizable spaces. Suppose $\Gamma < \mathsf{Homeo}(M_1)$ and $\rho : \Gamma \rightarrow \mathsf{Homeo}(M_2)$ is a homomorphism. If 
\begin{itemize}
\item $Z_1 \subset M_1$, $Z_2 \subset M_2$ are compact, 
\item for any finitely many $h_1, \ldots, h_m \in \Ga$ and $x \in Z_1$, there exists $g \in \Ga$ such that
$$g x \notin \bigcup_{i = 1}^m h_i Z_1,$$
and
\item for any $y \in Z_2$, there exists $h \in \Ga$ such that $\rho(h) y \notin Z_2$,
\end{itemize}
then we have
$$
M_1 \times M_2 = \bigcup_{\gamma \in \Gamma} (M_1 \smallsetminus \gamma Z_1) \times (M_2 \smallsetminus \rho(\gamma)Z_2).
$$
\end{lemma} 

\begin{proof} 
   The third hypothesis implies that $\bigcap_{\ga \in \Ga} \rho(\ga) Z_2 = \emptyset$.
   Since $Z_2$ is compact, there exist finitely many elements $h_1, \ldots, h_m \in \Ga$ such that 
$$
\rho(h_1) Z_2 \cap \cdots \cap \rho(h_m) Z_2 = \emptyset.
$$

Now suppose to the contrary that
\begin{align*}
C & := M_1 \times M_2 \smallsetminus \bigcup_{\gamma \in \Gamma} (M_1 \smallsetminus \gamma Z_1) \times (M_2 \smallsetminus \rho(\gamma)Z_2) \\
& = \bigcap_{\gamma \in \Gamma} \big( M_1 \times \rho(\gamma)Z_2 \big)\cup\big( \gamma Z_1 \times M_2\big)
\end{align*}
is non-empty. Let $(x, y) \in C$. Since $C$ is invariant under the action of $\{ (\ga, \rho(\ga)) : \ga \in \Ga \}$, we have
$$
(\ga x, \rho(\ga) y) \in C \quad \text{for all } \ga \in \Ga.
$$

By the choice of $\{\rho(h_1), \ldots, \rho(h_m) \}$, we have for some $j \in \{1, \ldots, m\}$ that $y \notin \rho(h_j) Z_2$, and hence
$$
(x, y) \in h_j Z_1 \times M_2.
$$
In other words,
$$
(h_j^{-1} x, \rho(h_j)^{-1} y) \in Z_1 \times M_2.
$$

By the second hypothesis, there exists $g \in \Ga$ such that
$$
g (h_j^{-1} x) \notin \bigcup_{i = 1}^m h_i Z_1.
$$
On the other hand, there exists $i \in \{ 1, \ldots, m\}$ such that $\rho(g) \rho(h_j)^{-1}y \notin \rho(h_i) Z_2$. Since $(g h_j^{-1} x, \rho(g) \rho(h_j)^{-1} y) \in C$, we have
$$
(g h_j^{-1} x, \rho(g) \rho(h_j)^{-1} y) \in h_i Z_1 \times M_2.
$$
In particular,
$$
g (h_j^{-1} x) \in  h_i Z_1.
$$
This is a contradiction to the choice of $g \in \Ga$.
\end{proof}

\section{An analogue of the conical limit set}\label{sec:conical limit set}    

Let $(M, \Ga, \sigma, \mu)$ be a PS-system of dimension $\delta \ge 0$. In this section we introduce an analogue of the conical limit set and relate its measure to the divergence of the Poincar\'e series.

Given a subset $H \subset \Gamma$, let $\Lambda_R(H) \subset M$ be the set of points $x \in M$ where there exists an escaping sequence $\{\gamma_n\} \subset H$ and $R >0$ such that 
$$
x \in \bigcap_{n \geq 1} \Oc_{R}(\gamma_n).
$$
Using this notation, the  $\mathscr{H}$-conical limit set  defined in Equation \eqref{eqn.conical hierarchy} can be rewritten as
$$
\La^{\rm con}(\mathscr{H}) = \Ga \cdot \bigcup_{R > 0} \bigcap_{n \ge 1} \Lambda_R(\mathscr{H}(n)).
$$
For simplicity, we denote by $\La^{\rm con}(\Ga)$ the conical limit set of the trivial hierarchy $\mathscr{H}(R) \equiv \Ga$.

\begin{theorem}\label{thm:conical limit set has full measure} \ 
\begin{enumerate} 
\item If $\mu(\Lambda_R(H)) > 0$ for some $H \subset \Ga$ and $R > 0$, then $\sum_{\gamma \in H} e^{-\delta \norm{\gamma}_\sigma} = +\infty$. 
\item If $(M, \Ga, \sigma, \mu)$ is well-behaved with respect to the trivial hierarchy $ \mathscr{H}(R) \equiv \Ga$ and $\sum_{\gamma \in \Gamma} e^{-\delta \norm{\gamma}_\sigma} = +\infty$, then $\mu \left(\Lambda^{\rm con}(\Gamma)\right)=1.$
\end{enumerate} 
 \end{theorem}

 \begin{remark} 
 In many examples, the shadows have the following additional property: for any $\alpha \in \Gamma$ and $R > 0$, there exists $R'>0$ such that 
      $$
      \alpha\Oc_R(\gamma)\subset \Oc_{R'}(\alpha \gamma)
      $$
      for all $\gamma \in \Gamma$. In this case, one has $\Lambda^{\rm con}(\Gamma) = \bigcup_{R > 0} \Lambda_R(\Gamma)$. 
 \end{remark}
 
 \subsection{Proof of Theorem \ref{thm:conical limit set has full measure} part (1)} 
 By Property~\ref{item:almost constant on shadows},  there exists $C = C(R) > 0$ such that for any $\ga \in \Ga$,
 $$
 \mu(\Oc_R(\gamma))= \gamma_*^{-1}\mu(\gamma^{-1}\Oc_R(\gamma)) \le C e^{-\delta \norm{\ga}_{\sigma}}.
 $$
 Now suppose $\sum_{\ga \in H} e^{-\delta \norm{\ga}_{\sigma}} < + \infty$. Then $H$ is countable and enumerating $H =\{ \gamma_n\}$, we have  
 $$
 \Lambda_R(H) \subset \bigcup_{n \geq N} \Oc_R(\gamma_n) \quad \text{for all } N > 0.
 $$
 Therefore, $\mu(\La_R(H)) = 0$, which is a contradiction.

\subsection{Proof of Theorem \ref{thm:conical limit set has full measure} part (2)}

The proof  is exactly the same as the proof of~\cite[Prop. 7.1]{BCZZ2024a}, which itself is similar to an earlier argument of Roblin~\cite{roblin}. Since the proof is short, we include it here.

We use the following variant of  Borel--Cantelli Lemma. 

\begin{lemma}[Kochen--Stone Lemma~\cite{KochenStone}]\label{lem: KS BC lemma}
Let $(X,\nu)$ be a finite measure space. If $\{ A_n\} \subset X$ is a sequence of measurable sets where 
$$
\sum_{n =1}^\infty \nu(A_n) = +\infty \quad \text{and} \quad \liminf_{N \rightarrow\infty} \frac{ \sum_{n, m = 1}^N \nu(A_n \cap A_m)}{\left(\sum_{n=1}^N \nu(A_n)\right)^2} < +\infty,
$$
then  
$$
\nu\left( \{ x \in X : x \text{ is contained in infinitely many of } A_1, A_2, \dots \}\right) > 0. 
$$
\end{lemma}

Using the Shadow Lemma (Proposition~\ref{prop.shadowlemma}), fix $R>0$ and $C_1 > 1$ such that 
\begin{equation}\label{eqn:use of shadow lemma in KS}
\frac{1}{C_1} e^{-\delta \norm{\gamma}_\sigma} \leq \mu\Big( \Oc_R(\gamma) \Big) \leq C_1 e^{-\delta \norm{\gamma}_\sigma}
\end{equation} 
for all $\gamma \in \Gamma$. Using Property~\ref{item:properness}, we can fix an enumeration $\Gamma = \{ \gamma_n\}$ such that 
$$
\norm{\gamma_1}_\sigma \leq \norm{\gamma_2}_\sigma \leq \cdots.
$$
We will show that the sets $A_n :=\Oc_R(\gamma_n)$ satisfy the hypothesis of the Kochen--Stone Lemma. 

The first estimate follows immediately from the divergence of the Poincar\'e series 
$$ 
\sum_{n=1}^\infty \mu(A_n) \geq \frac{1}{C_1} \sum_{\gamma \in \Gamma} e^{-\delta \norm{\gamma}_\sigma} = +\infty.
$$
The other estimate is only slightly more involved. Using Property~\ref{item:intersecting shadows},  there exists $C_2' > 0$ such that: if $1 \leq n \leq m$ and $A_n \cap A_m \neq \emptyset$, then 
$$
\norm{\gamma_n}_\sigma +\norm{\gamma_n^{-1}\gamma_m}_\sigma \leq \norm{\gamma_m}_\sigma+C_2'. 
$$
Hence, in this case, $\norm{\gamma_n^{-1}\gamma_m}_\sigma \leq \norm{\gamma_m}_\sigma+C_2$ where $C_2 = C_2' - \norm{\ga_1}_{\sigma}$ and 
$$
\mu( A_n \cap A_m) \leq \mu(A_m) \leq C_1 e^{-\delta \norm{\gamma_m}_\sigma}\leq  C_3e^{-\delta \norm{\gamma_n}_\sigma}e^{-\delta \norm{\gamma_n^{-1}\gamma_m}_\sigma}
$$
where $C_3 := C_1e^{\delta C_2'}$. 

Let $f(N) := \max\{ n : \norm{\gamma_n}_{\sigma} \leq \norm{\gamma_N}_{\sigma} +C_2\}$, which is finite by Property~\ref{item:properness}. Then
\begin{align*}
\sum_{m,n=1}^N \mu(A_n \cap A_m) & \leq 2 \sum_{1 \leq n \leq m \leq N} \mu(A_n \cap A_m) \leq 2C_3  \sum_{1 \leq n \leq m \leq N} e^{-\delta \norm{\gamma_n}_\sigma}e^{-\delta \norm{\gamma_n^{-1}\gamma_m}_\sigma} \\
& \leq 2C_3 \sum_{n=1}^N e^{-\delta \norm{\gamma_n}_\sigma} \sum_{n=1}^{f(N)}  e^{-\delta \norm{\gamma_n}_\sigma}.
\end{align*}

Thus to apply the Kochen--Stone lemma, it suffices to observe the following. 

\begin{lemma} There exists $C_4 > 0$ such that: 
$$
\sum_{n=1}^{f(N)}  e^{-\delta \norm{\gamma_n}_\sigma} \leq C_4 \sum_{n=1}^N e^{-\delta \norm{\gamma_n}_\sigma}
$$
for all $N \geq 1$. 
\end{lemma}

\begin{proof} Notice if $N < n \leq m \leq f(N)$ and $A_n \cap A_m \neq \emptyset$, then 
$$
\norm{\gamma_n^{-1}\gamma_m}_\sigma \leq \norm{\gamma_m}_\sigma-\norm{\gamma_n}_\sigma +C_2' \leq C_2 + C_2'. 
$$
Let  $D:=\#\{ \gamma \in \Gamma: \norm{\gamma}_\sigma \leq C_2 + C_2'\}$, which is finite by Property~\ref{item:properness}. Then  
$$
\sum_{n=N+1}^{f(N)}  e^{-\delta \norm{\gamma_n}_\sigma} \leq C_1\sum_{n=N+1}^{f(N)}  \mu(A_n) \leq C_1 D \mu\left( \bigcup_{n=N+1}^{f(N)} A_n \right) \leq C_1 D
$$
where Equation~\eqref{eqn:use of shadow lemma in KS} is applied in the first inequality.
Hence 
\begin{equation*}
 \sum_{n=1}^{f(N)}  e^{-\delta \norm{\gamma_n}_\sigma} \leq \left(1+ C_1D e^{\delta \norm{\gamma_1}_\sigma} \right) \sum_{n=1}^N e^{-\delta \norm{\gamma_n}_\sigma}. \qedhere
\end{equation*}
\end{proof} 

So by the Kochen--Stone lemma, the set 
$$
\Lambda_R(\Gamma) = \{ x \in M : x \text{ is contained in infinitely many of } A_1, A_2, \dots \}
$$
has positive $\mu$-measure. Hence $\mu(\Lambda^{\rm con}(\Gamma)) > 0$. 

Suppose for a contradiction that $\mu(\Lambda^{\rm con}(\Gamma) ) < 1$. Then 
$$
\mu' (\cdot) : = \frac{1}{\mu(\Lambda^{\rm con}(\Gamma)^c)} \mu\left( \Lambda^{\rm con}(\Gamma)^c \cap \cdot \right)
$$
is a $\sigma$-PS measure of dimension $\delta$, and so by the argument above we must have $\mu'(\Lambda^{\rm con}(\Gamma)) > 0$, which is impossible. Hence $\mu(\Lambda^{\rm con}(\Gamma) ) = 1$.
\qed

\section{An analogue of the Lebesgue differentiation theorem} 

Let $(M, \Ga, \sigma, \mu)$ be a well-behaved PS-system of dimension $\delta \ge 0$ with respect to a hierarchy $\mathscr{H} = \{ \mathscr{H} (R) \subset \Ga : R \ge 0\}$. Fix $R_0 > 0$ such that any $R \geq R_0$ satisfies the Shadow Lemma (Proposition~\ref{prop.shadowlemma}).

In this section we prove the following analogue of the Lebesgue differentiation theorem (which is known to hold for many particular PS-systems).

\begin{theorem}\label{thm:Leb diff} Fix $R \ge R_0$. 
 If $h \in L^1(M, \mu)$, then for $\mu$-a.e. $x \in M$ we have
   $$
0 = \lim_{n \rightarrow \infty} \frac{1}{\mu(\gamma\Oc_R(\gamma_n))} \int_{\gamma\Oc_R(\gamma_n)} |h(y) - h(x)| d\mu(y),
$$
and hence
$$
h(x) = \lim_{n \rightarrow \infty} \frac{1}{\mu(\gamma\Oc_R(\gamma_n))} \int_{\gamma\Oc_R(\gamma_n)} h(y) d\mu(y),
$$
whenever $x \in \bigcap_{n \ge 1} \ga \Oc_R(\ga_n)$ for some $\ga \in \Ga$ and escaping sequence $\{ \ga_n \} \subset \mathscr{H}(R)$.

\end{theorem} 

Delaying the proof of the theorem, we state several corollaries. We will use Theorem~\ref{thm:Leb diff} to prove that $\Gamma$ acts ergodically. 

\begin{corollary}\label{cor:ergodicity}  
   If $\mu(\La^{\rm con}(\mathscr{H})) = 1$, then the $\Gamma$-action on $(M,\mu)$ is ergodic.
   In particular, if the hierarchy is trivial (i.e. $\mathscr{H}(R) \equiv \Gamma$) and $\sum_{\gamma \in \Gamma} e^{-\delta \norm{\gamma}_\sigma} = +\infty$, then the $\Gamma$-action on $(M,\mu)$ is ergodic.
\end{corollary}

Corollary~\ref{cor:ergodicity} is a  consequence of Theorem~\ref{thm:Leb diff} and the following lemma (which is itself a corollary of Theorem \ref{thm:Leb diff}).

\begin{lemma} \label{lem:Leb Density} Fix $R \ge R_0$. 
 If $E \subset M$ is measurable, then for $\mu$-a.e. $x \in E$ we have 
$$
0= \lim_{n \to \infty} \mu \left( \ga_{n}^{-1} \Oc_{R}(\ga_{n}) \smallsetminus \ga_{n}^{-1}\ga^{-1}E\right)
$$
whenever $x \in \bigcap_{n \ge 1} \ga \Oc_R(\ga_n)$ for some $\ga \in \Ga$ and escaping sequence $\{ \ga_n \} \subset \mathscr{H}(R)$.
\end{lemma} 

\begin{remark}Lemma~\ref{lem:Leb Density} can be viewed as an analogue of the Lebesgue density theorem. \end{remark}

For use in Section~\ref{sec.convergenceaction} we also record the following corollary about approximate continuity of maps into separable metric spaces.

\begin{corollary}\label{cor:approx continuity} 
Fix $R \ge R_0$. If $F : M \rightarrow (Y, \dist_Y)$ is a Borel measurable map into a separable metric space, then for $\mu$-a.e. $x \in M$ we have 
$$
0 = \lim_{n \rightarrow \infty} \frac{1}{\mu(\gamma\Oc_R(\gamma_n))}\mu\left( \left\{ y \in \gamma\Oc_R(\gamma_n) : \dist_Y(F(x), F(y)) > \epsilon\right\} \right)
$$
for all $\epsilon > 0$ whenever $x \in \bigcap_{n \ge 1} \ga \Oc_R(\ga_n)$ for some $\ga \in \Ga$ and escaping sequence $\{ \ga_n \} \subset \mathscr{H}(R)$.
\end{corollary} 

The rest of the section is devoted to the proof of the theorem and the three corollaries.

\subsection{Proof of Theorem~\ref{thm:Leb diff}} Recall that any $R \geq R_0$ satisfies the Shadow Lemma (Proposition \ref{prop.shadowlemma}) and recall that $\La_R(\mathscr{H}(R))$ is the set of points $x \in M$ such that $x \in \bigcap_{n \ge 1} \Oc_R(\ga_n)$ for some escaping sequence $\{ \ga_n \} \subset \mathscr{H}(R)$.

Fix $R \ge R_0$ and $h \in L^1(M,\mu)$. For $\alpha \in \Gamma$, define functions $\mathcal{A}_{\alpha} h, \mathcal{B}_{\alpha} h : M \rightarrow [0,+\infty]$ by 
$$\begin{aligned}
   \mathcal{A}_{\alpha} h(x) & =  \begin{cases}
\lim\limits_{T\to\infty}\sup\limits_{\substack{\ga \in \mathscr{H}(R) \\ \norm\gamma_\sigma\geq T\\ x\in\alpha\Oc_{R}(\gamma)}} \frac{1}{\mu(\alpha\Oc_{R}(\gamma))} \int_{\alpha\Oc_{R}(\gamma)} \abs{h(y)-h(x)} d\mu(y) & \text{if } x \in  \alpha\Lambda_{R}(\mathscr{H}(R)) \\
0 & \text{else}
\end{cases} \\
\text{and} & \\
\mathcal{B}_{\alpha} h(x) & =  \begin{cases}
\lim\limits_{T\to\infty}\sup\limits_{\substack{\ga \in \mathscr{H}(R) \\ \norm\gamma_\sigma\geq T\\ x\in\alpha\Oc_{R}(\gamma)}} \frac{1}{\mu(\alpha\Oc_{R}(\gamma))} \int_{\alpha\Oc_{R}(\gamma)} \abs{h(y)} d\mu(y) & \text{if } x \in \alpha \Lambda_{R}(\mathscr{H}(R)) \\
0 & \text{else}
\end{cases}. 
\end{aligned}
$$

\begin{lemma} \label{lem.vanishingAf}
   If  $\alpha \in \Gamma$, then $\mathcal{A}_{\alpha} h = 0$ $\mu$-a.e. 
\end{lemma}

\begin{proof} It suffices to show that $\mu(\{ x : \mathcal{A}_{\alpha} h(x) > c\})=0$ for any $c > 0$. To that end, fix $c, \epsilon > 0$ and a continuous function $g : M \rightarrow \Rb$ with 
$$
\int_M \abs{h-g}d\mu < \epsilon. 
$$
Then 
$$
 \mathcal{A}_{\alpha} h(x) \leq \mathcal{B}_{\alpha}(h-g)(x)  +\abs{h(x)-g(x)} +  \mathcal{A}_{\alpha} g(x).
 $$ 
 Hence 
 $$
 \{ x:  \mathcal{A}_{\alpha} h(x) > c\} \subset N_1 \cup N_2 \cup N_3 
 $$
 where 
 \begin{align*}
N_1 & : = \{x :  \mathcal{B}_{\alpha}(h-g)(x)> c/3\}; \\
N_2 &:= \{ x: \abs{h(x)-g(x)}  > c/3\}; \\
 N_3 & : = \{x:   \mathcal{A}_{\alpha} g(x) > c/3\}. 
 \end{align*}
 Since $g$ is continuous, Property \ref{item:diam goes to zero ae} implies that $\mu(N_3)=0$. Further, 
 $$
 \mu(N_2) \leq \frac{3}{c}\int_M \abs{h-g}d\mu < \frac{3}{c}\epsilon. 
 $$
 
 To bound $\mu(N_1)$, we use Lemma~\ref{lem:V covering}. For any $x \in N_1$ there exists $\gamma_x \in \mathscr{H}(R)$ such that $x \in \alpha \Oc_R(\gamma_x)$ and 
 $$
 \int_{\alpha \Oc_R(\gamma_x)} \abs{h-g} d\mu > \frac{c}{4} \mu(\alpha \Oc_R(\gamma_x)). 
 $$
 By Lemma~\ref{lem:V covering} there exist $N_1' \subset N_1$ and $R'>R$ such that 
 $$
 N_1 \subset \bigcup_{x \in N_1} \alpha \Oc_R(\gamma_x) \subset  \bigcup_{x \in N_1'} \alpha \Oc_{R'}(\gamma_x)
$$
and the shadows $\{ \alpha \Oc_{R}(\gamma_x) : x \in N_1'\}$ are disjoint. By Property~\ref{item:coycles are bounded}, there exists $C_\alpha > 1$ such that 
$$
C_\alpha^{-1} \mu \leq \alpha^{-1}_* \mu \leq C_\alpha \mu.
$$
Then by the Shadow Lemma (Proposition \ref{prop.shadowlemma}), there exists $C =C(\alpha, R, R')> 1$ such that 
$$
\mu( \alpha \Oc_{R'}(\gamma)) \leq C \mu( \alpha \Oc_{R}(\gamma))
$$
for all $\gamma \in \Gamma$. Then 
\begin{align*}
 \mu(N_1) & \leq \sum_{x \in N_1'} \mu(\alpha \Oc_{R'}(\gamma_x)) \leq C \sum_{x \in N_1'} \mu(\alpha \Oc_R(\gamma_x)) < \frac{4C}{c} \sum_{x \in N_1'}\int_{ \alpha \Oc_{R}(\gamma_x)} \abs{h-g} d\mu \\
 & \leq \frac{4C}{c}\int_M \abs{h-g}d\mu < \frac{4C}{c}\epsilon. 
\end{align*}

Thus 
$$
\mu( \{ x:  \mathcal{A}_{\alpha} h(x) > c\} ) \leq \mu(N_1)+\mu(N_2)+\mu(N_3) < \frac{4C}{c}\epsilon+\frac{3}{c}\epsilon + 0. 
$$
Since $\epsilon > 0$ was arbitrary, we see that $\{ x:  \mathcal{A}_{\alpha} h(x) > c\}$ is $\mu$-null. Then since $c > 0$ was arbitrary, $ \mathcal{A}_{\alpha} h =0$ $\mu$-a.e.
\end{proof} 

We now finish the proof of Theorem \ref{thm:Leb diff}.
Fix $h \in L^1(M,\mu)$ and set 
$$
M':= \bigcap_{\alpha \in \Gamma} 
\{ x:  \mathcal{A}_{\alpha} h(x) =0\}.
$$
Then $\mu(M') = 1$ by Lemma \ref{lem.vanishingAf}.

Fix $x \in M'$ and suppose that $$x \in \bigcap_{n \ge 1} \ga \Oc_R(\ga_n)$$ for some $\ga \in \Ga$ and an escaping sequence  $\{\gamma_n\} \subset \mathscr{H}(R)$. Then
$$
\limsup_{n \to \infty} \frac{1}{\mu(\gamma\Oc_R(\gamma_n))}  \int_{\gamma\Oc_R(\gamma_n)} \abs{h(y)-h(x)} d\mu(y) \le \mathcal{A}_{ \ga}h(x) = 0,
$$
completing the proof.
\qed

\subsection{Proof of Lemma~\ref{lem:Leb Density} } Fix $R \ge R_0$ and a measurable set $E \subset M$. 

   For each $g \in \Ga$, consider the function $\mathbf{1}_{g^{-1}E}$. Then by Theorem \ref{thm:Leb diff}, we have a measurable subset $M_g \subset M$ such that $\mu(M_g) = 1$ and for any $y \in M_g \cap g^{-1}E$, 
   $$
   1 = \lim_{n \to \infty} \frac{ \mu( g^{-1}E \cap \ga \Oc_R(\ga_n))}{\mu(\ga \Oc_{R}(\ga_n))}
   $$
   whenever $y \in \bigcap_{n \ge 1} \ga \Oc_R(\ga_n)$ for some  $\ga \in \Ga$ and escaping  sequence $\{ \ga_n \} \subset \mathscr{H}(R)$.
 Set
   $$
   M_E := \bigcap_{g \in \Ga} g M_g.
   $$
   Since $\mu$ is $\Ga$-quasi-invariant, $\mu(M_E) = 1$.

Fix $x \in E \cap M_E$ and suppose that $x \in \bigcap_{n \ge 1} \ga \Oc_R(\ga_n)$ for some $\ga \in \Ga$ and escaping sequence $\{ \ga_n \} \subset \mathscr{H}(R)$. We then have $\ga^{-1} x \in M_{\ga} \cap \ga^{-1}E$ and moreover $\ga^{-1} x \in \bigcap_{n \ge 1} \Oc_R(\ga_n)$. Therefore
   $$
   1 = \lim_{n \to \infty} \frac{\mu(\ga^{-1} E \cap \Oc_R(\ga_n))}{\mu(\Oc_R(\ga_n))} = \lim_{n \to \infty} \frac{({\ga_n^{-1}}_*\mu)(\ga_n^{-1} \ga^{-1} E \cap \ga_n^{-1} \Oc_R(\ga_n))}{({\ga_n^{-1}}_*\mu)(\ga_n^{-1}\Oc_R(\ga_n))}.
   $$
   In particular,
   $$
\lim_{n \to \infty} \frac{({\ga_n^{-1}}_*\mu)(\ga_n^{-1} \ga^{-1} E^c \cap \ga_n^{-1} \Oc_R(\ga_n))}{({\ga_n^{-1}}_*\mu)(\ga_n^{-1}\Oc_R(\ga_n))}  = 0.
   $$
   By Property~\ref{item:almost constant on shadows}, there exists $C = C(R) > 1$ such that 
   $$
   C e^{-\delta\norm{\gamma_n}} \le \frac{d{\ga_n^{-1}}_*\mu}{d\mu} \le C e^{-\delta\norm{\gamma_n}} \quad \mu\text{-a.e.}
   $$
   on $\ga_n^{-1} \Oc_R(\ga_n)$. So 
   $$
   \lim_{n \to \infty} \frac{\mu(\ga_n^{-1} \ga^{-1} E^c \cap \ga_n^{-1} \Oc_R(\ga_n))}{\mu(\ga_n^{-1}\Oc_R(\ga_n))}  = 0.
   $$
Since $\mu(\ga_n^{-1}\Oc_R(\ga_n)) \leq 1$, we then have 
      $$
   \lim_{n \to \infty} \mu(\ga_n^{-1} \ga^{-1} E^c \cap \ga_n^{-1} \Oc_R(\ga_n)) = 0,
   $$
   which implies that 
   \begin{align*}
0= & \lim_{n \to \infty} \mu ( \ga_{n}^{-1} \Oc_{R}(\ga_{n}) \smallsetminus \ga_{n}^{-1}\ga^{-1}E).
   \end{align*} \qedhere
   \qed

\subsection{Proof of Corollary~\ref{cor:ergodicity}}
Once we show the first statement, the second follows from Theorem~\ref{thm:conical limit set has full measure}.

Recall that $\La^{\rm con}(\mathscr{H}) = \Ga \cdot \bigcup_{R > 0} \bigcap_{n \ge 1} \La_R(\mathscr{H}(n))$, which is assumed to have full $\mu$-measure.
We show that the $\Gamma$-action on $(M,\mu)$ is ergodic using Lemma \ref{lem:Leb Density}.
Let $E \subset M$ be a $\Ga$-invariant measurable set with $\mu(E) > 0$.
Since the sequence $\Ga \cdot \bigcap_{n \ge 1} \La_R(\mathscr{H}(n))$ is non-decreasing in $R$ by Property \ref{item:shadow inclusion}, there exists $R \ge R_0$ such that $\mu( E \cap \Ga \cdot \bigcap_{n \ge 1} \La_R(\mathscr{H}(n))) > 0$.

Fix a sequence $R_k \to + \infty$. For each $k \ge 1$, let $M_k \subset M$ a full measure set satisfying Lemma \ref{lem:Leb Density}. We then set $M_E := \bigcap_{k \ge 1} M_k$ which is of $\mu$-full measure.

Fix $x \in E \cap M_E \cap \Ga \cdot \bigcap_{n \ge 1} \La_R(\mathscr{H}(n))$. Then there exist $\ga \in \Ga$ and an escaping sequence $\{ \ga_n \in \mathscr{H}(n) \}$ such that
$$
x \in \bigcap_{n \ge 1} \ga \Oc_R(\ga_n).
$$
Since the hierarchy $\mathscr{H}$ consists of a non-increasing sequence of subsets of $\Ga$, for each $k \ge 1$, we have $\ga_n \in \mathscr{H}(R_k)$ for all large $n \ge 1$.
Then by Property \ref{item:shadow inclusion}, Lemma \ref{lem:Leb Density},  and the $\Ga$-invariance of $E$,
$$
0 = \lim_{n \to \infty} \mu( \ga_{n}^{-1} \Oc_{R_k}(\ga_{n}) \smallsetminus E).
$$
Hence, after passing to a subsequence of $\{ \ga_n \}$, we have
$$
0 = \lim_{n \to \infty} \mu ( \ga_{n}^{-1} \Oc_{R_n}(\ga_{n}) \smallsetminus E).
$$

Fix a metric on $M$ which generates the topology. Passing to a subsequence, we can suppose that $M \smallsetminus  \ga_{n}^{-1} \Oc_{R_n}(\ga_{n}) $ converges to some compact set $Z \subset M$ with respect to the Hausdorff distance (it is possible for $Z = \emptyset$, in which case $M \smallsetminus  \ga_{n}^{-1} \Oc_{R_n}(\ga_{n}) =\emptyset$ for $n$ sufficiently large). 

Then for each $j \ge 1$, 
$$
M \smallsetminus  \ga_{n}^{-1} \Oc_{R_n}(\ga_{n}) \subset \Nc_{1/j}(Z)
$$ 
when $n$ is sufficiently large. Therefore
\begin{align*}
\mu((M \smallsetminus Z) \smallsetminus E)  
& \le \mu(  (M \smallsetminus \Nc_{1/j}(Z)) \smallsetminus E) + \mu( \Nc_{1/j}(Z) \smallsetminus Z) \\
& \le \lim_{n \rightarrow \infty} \mu(\ga_{n}^{-1} \Oc_{R_n}(\ga_{n}) \smallsetminus  E) + \mu( \Nc_{1/j}(Z) \smallsetminus  Z) \\
& = \mu( \Nc_{1/j}(Z) \smallsetminus Z).
\end{align*}

Since $Z$ is closed, $\Nc_{1/j}(Z) \smallsetminus Z$ is a decreasing sequence of sets whose limit is the empty set. Therefore, taking $j \to +\infty$, we have 
$$
\mu((M \smallsetminus Z) \smallsetminus E) = 0.
$$
By Property~\ref{item:empty Z intersection}, $M = \bigcup_{\ga \in \Ga} M \smallsetminus \ga Z$. Therefore, it follows from the $\Ga$-invariance of $E$ and the $\Ga$-quasi-invariance of $\mu$ that
$$
\mu(M \smallsetminus E) \le \sum_{\ga \in \Ga} \mu( (M \smallsetminus \ga Z) \smallsetminus E) =  \sum_{\ga \in \Ga} \gamma_*^{-1}\mu( (M \smallsetminus  Z) \smallsetminus E) =0.
$$
This shows $\mu(E) = 1$, finishing the proof.
\qed

\subsection{Proof of Corollary~\ref{cor:approx continuity}} Fix $R \ge R_0$ and fix a countable dense subset $D = \{ z_n\} \subset Y$. For $k \in \Nb$ define $f_k : M \rightarrow \Nb$ by letting 
$$
f_k(x) = \min\{ n : \dist_Y(F(x), z_n) < 1/k\}. 
$$
Then for $K \in \Nb$ let $h_{k,K}(x) = \min\{ f_k(x), K\}$. Each $h_{k,K}$ is bounded and hence in $L^1(M, \mu)$. Then there exists a full measure set $M'$ such that Theorem~\ref{thm:Leb diff} holds for every $x \in M'$ and every $h_{k,K}$, for our given $R \ge R_0$.

Now fix $x \in M'$ and $\epsilon > 0$. Then fix $k \in \Nb$ with $\frac{1}{2k} < \epsilon$ and fix $K \in \Nb$ with $f_k(x) < K$. Then for $y \in M$,
$$
\dist_Y(F(x), F(y)) > \epsilon \Rightarrow \abs{h_{k,K}(x)-h_{k,K}(y)} \geq 1. 
$$
So whenever $x \in \bigcap_{n \ge 1} \ga \Oc_R(\ga_n)$ for some  $\ga \in \Ga$ and escaping sequence $\{ \ga_n \} \subset \mathscr{H}(R)$, we have
\begin{align*}
0 & \le \lim_{n \rightarrow \infty} \frac{1}{\mu(\gamma\Oc_R(\gamma_n))}\mu\left( \left\{ y \in \gamma\Oc_R(\gamma_n) : \dist_Y(F(x), F(y)) > \epsilon\right\} \right) \\
& \leq \lim_{n \rightarrow \infty} \frac{1}{\mu(\gamma\Oc_R(\gamma_n))} \int_{\gamma\Oc_R(\gamma_n)}\abs{h_{k,K}(x)-h_{k,K}(y)}  d\mu(y)=0. 
\end{align*} 
\qed

\section{Mixed Shadows and a Shadow Lemma} 

For the rest of the section suppose 
\begin{itemize}
\item $(M_1, \Ga_1, \sigma_1, \mu_1)$ is a well-behaved PS-system of dimension $\delta_1$ with respect to a hierarchy $\mathscr{H}_1 = \{\mathscr{H}_1(R) \subset \Ga_{1} : R \ge 0\}$. 
\item  $(M_2, \Ga_2, \sigma_2, \mu_2)$ is a PS-system of dimension $\delta_2$.
\item  There exists an onto homomorphism $\rho : \Gamma_1 \rightarrow \Gamma_2$ and a measurable $\rho$-equivariant  map $f : Y \rightarrow M_2$ where $Y \subset M_1$ is a $\Gamma_1$-invariant subset of full $\mu_1$-measure.
\end{itemize} 
In this section we introduce mixed shadows, which play a key role in our main rigidity result, and prove a version of the Shadow Lemma. 

\begin{definition} 
For $R > 0$ and $\ga \in \Ga$, the associated \emph{mixed shadow} is 
$$
\Oc_R^f(\ga) := \Oc_R(\ga) \cap f^{-1}(\Oc_R(\rho(\ga))) \cap Y \subset M_1.
$$
\end{definition}

\begin{theorem}[Mixed Shadow Lemma] \label{thm.mixedshadow}\

   \begin{enumerate}
      \item For any sufficiently large $R > 0$, there exists $C=C(R) >1$ such that 
      $$
      \frac{1}{C} e^{-\delta_1 \norm{\gamma}_{\sigma_1}} \leq \mu_1\left( \Oc_R^f(\gamma) \right) \leq C e^{-\delta_1 \norm{\gamma}_{\sigma_1}}
      $$
      for all $\gamma \in \Gamma$. 
      \item Suppose, in addition, that $f$ maps Borel subsets of $Y$ to Borel subsets of $M_2$ and $\mu_2(f(Y))>0$. Then for any sufficiently large $R>0$, there exists $C=C(R) >1$ such that 
      $$
      \frac{1}{C} e^{-\delta_2 \norm{\rho(\gamma)}_{\sigma_2}} \leq \mu_2\left( f\left(\Oc_R^f(\gamma) \right)\right) \leq C e^{-\delta_2 \norm{\rho(\gamma)}_{\sigma_2}}
      $$
      for all $\gamma \in \Gamma$. 
      \end{enumerate}
\end{theorem}

Delaying the proof of the theorem for a moment, we establish the following corollary. 

\begin{theorem} \label{thm:Leb diff mixed}
There exists $R_0 > 0$ such that:  if $R \ge R_0$ and 
    $h \in L^1(M_1, \mu_1)$, then for $\mu_1$-a.e. $x \in M_1$ we have
   $$
   h(x) = \lim_{n \to \infty} \frac{1}{\mu_1 \left( \Oc_R^f(\ga_n) \right)} \int_{ \Oc_R^f(\ga_n) } h(y) d\mu_1(y)
   $$
   whenever $x \in \bigcap_{n \ge 1} \Oc_R(\ga_n)$ for some escaping sequence $\{ \ga_n \} \subset \mathscr{H}_1(R)$.
\end{theorem}

\begin{proof}
 Fix $R_0 > 0$ such that any $R \ge R_0$ satisfies Proposition \ref{prop.shadowlemma} for $(M_1, \Ga_1, \sigma_1,\mu_1)$ and Theorem \ref{thm.mixedshadow} part (1). Fix $R \ge R_0$ and $h  \in L^1(M_1, \mu_1)$. Let $M_1' \subset M_1$ be a full $\mu_1$-measure set satisfying Theorem \ref{thm:Leb diff} for $h$ and $R$. 
 
 Now fix $x \in M_1'$ and an escaping sequence $\{ \ga_n \} \subset \mathscr{H}_1(R)$ where $x \in \bigcap_{n \ge 1} \Oc_R(\ga_n)$. By Theorem \ref{thm:Leb diff}, 
   $$
   0 = \lim_{n \to \infty} \frac{1}{\mu_1 \left( \Oc_R(\ga_n) \right)} \int_{ \Oc_R(\ga_n) } |h(y) - h(x)| d\mu_1(y).
   $$
   By our choice of $R_0$, there exists $C=C(R) > 1$ such that 
   $$
  \mu_1 \left( \Oc_R^f(\ga_n) \right) \geq C \mu_1 \left( \Oc_R(\ga_n) \right).
   $$
  Then, since $\Oc_R^f(\ga_n) \subset \Oc_R(\ga_n)$, 
   $$\begin{aligned}
   0 \le \frac{1}{\mu_1 \left( \Oc_R^f(\ga_n) \right)} & \int_{ \Oc_R^f(\ga_n) } |h(y) - h(x)| d\mu_1(y) \\
   & \le  \frac{1}{\mu_1 \left( \Oc_R^f(\ga_n) \right)} \int_{ \Oc_R(\ga_n) } |h(y) - h(x)| d\mu_1(y) \\
   & \le \frac{C}{\mu_1 \left( \Oc_R(\ga_n) \right)} \int_{ \Oc_R(\ga_n) } |h(y) - h(x)| d\mu_1(y) \to 0.
   \end{aligned}
   $$
   Therefore,
   $$\begin{aligned}
   & \left| h(x)  - \lim_{n \to \infty} \frac{1}{\mu_1 \left( \Oc_R^f(\ga_n) \right)} \int_{ \Oc_R^f(\ga_n) } h(y) d\mu_1(y) \right| \\
   & \qquad \qquad \le \lim_{n \to \infty} \frac{1}{\mu_1 \left( \Oc_R^f(\ga_n) \right)} \int_{ \Oc_R^f(\ga_n) } |h(y) - h(x)| d\mu_1(y) = 0.
   \end{aligned}
   $$
\end{proof}

\subsection{Proof of Theorem~\ref{thm.mixedshadow}} Fix metrics on $M_1, M_2$ which induce their topologies. 
As in the proof of the classical Shadow Lemma, we start by proving lower bounds for translates of shadows.

\begin{lemma} \label{lem.mixedshadow1}
   For any sufficiently large $R > 0$, 
   $$
   \inf_{\ga \in \Ga_1} \mu_1 \left( \ga^{-1} \Oc_R^f(\ga) \right) > 0.
   $$
\end{lemma}

\begin{proof}
Suppose not. Then there exist sequences $R_n \to +\infty$ and  $\{\ga_n\} \subset \Ga_1$ such that
$$
\mu_1 \left(\ga_n^{-1} \Oc_{R_n}^f(\ga_n) \right) < \frac{1}{n} \quad \text{for all} \quad n \ge 1.
$$
Since $\mu_1(Y) = 1$ and $f$ is $\rho$-equivariant, 
$$
\mu_1 \Big( \ga_n^{-1}\Oc_{R_n}(\ga_n) \cap f^{-1} \big( \rho(\ga_n)^{-1} \Oc_{R_n}(\rho(\ga_n))\big)\Big) < \frac{1}{n}.
$$
Note that
\begin{align*}
M_1 & \smallsetminus \Big( \ga_n^{-1}\Oc_{R_n}(\ga_n) \cap f^{-1} \big( \rho(\ga_n)^{-1} \Oc_{R_n}(\rho(\ga_n))\big) \Big)\\
& = \Big(M_1 \smallsetminus  \ga_n^{-1}\Oc_{R_n}(\ga_n) \Big) \cup \Big(M_1 \smallsetminus f^{-1} \big( \rho(\ga_n)^{-1} \Oc_{R_n}(\rho(\ga_n))\big)\Big) \\
& = \Big(M_1 \smallsetminus  \ga_n^{-1}\Oc_{R_n}(\ga_n) \Big) \cup f^{-1} \Big( M_2 \smallsetminus \rho(\ga_n)^{-1} \Oc_{R_n}(\rho(\ga_n)) \Big).
\end{align*}

After passing to a subsequence,  we can assume that 
$$
[M_1 \smallsetminus  \ga_n^{-1}\Oc_{R_n}(\ga_n)]  \to Z_1
$$
for some (possibly empty) compact subset $Z_1 \subset M_1$ with respect to the Hausdorff distance and 
$$
[M_2 \smallsetminus \rho(\ga_n)^{-1} \Oc_{R_n}(\rho(\ga_n))] \to Z_2
$$
for some (possibly empty) compact subset $Z_2 \subset M_2$ with respect to the Hausdorff distance.

For any $\epsilon > 0$ and $n \ge 1$ sufficiently large (depending on $\epsilon$), 
$$
M_1 \smallsetminus \ga_n^{-1}\Oc_{R_n}(\ga_n) \subset \Nc_{\epsilon} (Z_1) \quad \text{and} \quad M_2 \smallsetminus \rho(\ga_n)^{-1} \Oc_{R_n}(\rho(\ga_n)) \subset \Nc_{\epsilon}(Z_2).
$$
Hence
$$
\mu_1\Big(\Nc_{\epsilon}(Z_1) \cup  f^{-1}(\Nc_{\epsilon}(Z_2))\Big) > 1 - 1/n
$$
for all large $n \ge 1$. Taking the limit $n \to \infty$, we have 
$$
\mu_1\Big(\Nc_{\epsilon}(Z_1) \cup  f^{-1}(\Nc_{\epsilon}(Z_2))\Big) = 1.
$$

Since $Z_1$ and $Z_2$ are closed,
$$
Z_1 \cup f^{-1}(Z_2) = \bigcap_{k \ge 1} \Nc_{1/k}(Z_1) \cup f^{-1}(\Nc_{1/k}(Z_2)).
$$
We therefore have $\mu_1(Z_1 \cup f^{-1}(Z_2)) = 1$. In other words,
\begin{equation} \label{eqn.mixedshadowproof1}
\mu_1 \left( (M_1 \smallsetminus Z_1) \cap f^{-1}(M_2 \smallsetminus Z_2) \right) = 0,
\end{equation} and hence
$$
\mu_1 \left( \bigcup_{\ga \in \Ga_1} (M_1 \smallsetminus \ga Z_1) \cap f^{-1}(M_2 \smallsetminus \rho(\ga)Z_2) \right) = 0
$$
by the $\Ga_1$-quasi-invariance of $\mu_1$. However then Lemma \ref{lem.keylemma} implies  $\mu_1(M_1) = 0$, contradiction.
\end{proof}

\begin{lemma} \label{lem.mixedshadow2}
Suppose that $f$ maps Borel subsets of $Y$ to Borel subsets of $M_2$ and $\mu_2(f(Y))>0$. For any sufficiently large $R > 0$, 
$$
\inf_{\ga \in \Ga_1} \mu_2 \left( \rho(\ga)^{-1}  f \left( \Oc_R^f(\ga) \right) \right) > 0.
$$
\end{lemma}

\begin{proof}
Suppose not. Then there exist sequences $R_n \to +\infty$ and $\{\ga_n \} \subset \Ga_1$ such that
$$
\mu_2 \left(\rho (\ga_n)^{-1}  f \left( \Oc_{R_n}^f(\ga_n) \right) \right)  < \frac{1}{n}.
$$
Then, since  $f$ is $\rho$-equivariant, 
\begin{equation}\label{eqn:mu2 meas converges to 0}
\mu_2 \left( f \left( \ga_{n}^{-1} \Oc_{R_n}(\ga_n) \cap Y \right) \cap \rho(\ga_n)^{-1} \Oc_{R_n}(\rho(\ga_n)) \right) \to 0.
\end{equation}

After passing to a subsequence,  we can assume that 
$$
[M_1 \smallsetminus  \ga_n^{-1}\Oc_{R_n}(\ga_n)]  \to Z_1
$$
for some (possibly empty) compact subset $Z_1 \subset M_1$ with respect to the Hausdorff distance and 
$$
[M_2 \smallsetminus \rho(\ga_n)^{-1} \Oc_{R_n}(\rho(\ga_n))] \to Z_2
$$
for some (possibly empty) compact subset $Z_2 \subset M_2$ with respect to the Hausdorff distance.

By Lemma \ref{lem.keylemma}, 
$$
M_1 \times M_2 = \bigcup_{\gamma \in \Gamma_1} (M_1 \smallsetminus \gamma Z_1) \times (M_2 \smallsetminus \rho(\gamma)Z_2) 
$$
and hence
$$
M_1 \times M_2 = \bigcup_{\epsilon > 0} \bigcup_{\gamma \in \Gamma_1} \left(M_1 \smallsetminus \gamma \overline{\Nc_\epsilon(Z_1)}\right) \times \left(M_2 \smallsetminus \rho(\gamma)\overline{\Nc_\epsilon(Z_2)}\right).
$$
By compactness, we can fix $\epsilon > 0$ and a finite set $F \subset \Gamma$ such that 
\begin{equation}\label{eqn:M1xM2 is some funny union}
M_1 \times M_2 =\bigcup_{\gamma \in F} \left(M_1 \smallsetminus \gamma \overline{\Nc_{2\epsilon}(Z_1)}\right) \times \left(M_2 \smallsetminus \rho(\gamma)\overline{\Nc_{2\epsilon}(Z_2)}\right).
\end{equation}

Now for  $n \ge 1$ sufficiently large, 
$$
M_1 \smallsetminus \ga_n^{-1} \Oc_{R_n}(\ga_n) \subset  \Nc_{\epsilon}(Z_1)  \quad \text{and} \quad M_2 \smallsetminus \rho(\ga_n)^{-1} \Oc_{R_n}(\rho(\ga_n)) \subset  \Nc_{\epsilon}(Z_2)
$$
and hence 
$$
M_1 \smallsetminus \Nc_{\epsilon}(Z_1) \subset \ga_n^{-1} \Oc_{R_n}(\ga_n)    \quad \text{and} \quad M_2 \smallsetminus \Nc_{\epsilon}(Z_2) \subset \rho(\ga_n)^{-1} \Oc_{R_n}(\rho(\ga_n)).
$$
So, by Equation~\eqref{eqn:mu2 meas converges to 0}, 
$$
\mu_2 \left( f \left( Y \smallsetminus  \Nc_{\epsilon}(Z_1) \right) \cap \left( M_2 \smallsetminus \Nc_{\epsilon}(Z_2) \right) \right) =0. 
$$
Then, since $\mu_2$ is $\rho(\Ga)$-quasi-invariant, 
$$
\mu_2 \left( \bigcup_{\gamma \in F} f \left( Y \smallsetminus \gamma  \Nc_{\epsilon}(Z_1) \right) \cap \left( M_2 \smallsetminus \rho(\gamma) \Nc_{\epsilon}(Z_2) \right) \right) =0. 
$$
Then Equation~\eqref{eqn:M1xM2 is some funny union} implies that $\mu_2(f(Y)) = 0 $, which is a contradiction. 
\qed
\medskip

With the lower bounds in Lemmas~\ref{lem.mixedshadow1} and~\ref{lem.mixedshadow2}, one can complete the proof of Theorem \ref{thm.mixedshadow} by arguing exactly same as in Proposition \ref{prop.shadowlemma}. 
\end{proof}

\section{The Main Theorem}

In this section we prove Theorem~\ref{thm:main rigidity theorem in intro}, which we restate here. 

\begin{theorem}\label{thm:main rigidity theorem}
 Suppose 
\begin{itemize} 
\item $(M_1, \Gamma_1, \sigma_1, \mu_1)$ is a well-behaved PS-system of dimension $\delta_1$ with respect to a hierarchy $\mathscr{H}_1 = \{\mathscr{H}_1(R) \subset \Ga_1 : R \ge 0\}$ and 
$$
\mu_1 (\La^{\rm con}(\mathscr{H}_1))= 1.
$$ 
\item $(M_2, \Gamma_2, \sigma_2, \mu_2)$ is a  PS-system of dimension $\delta_2$.
\item There exists an onto homomorphism $\rho : \Gamma_1 \rightarrow \Gamma_2$, a measurable $\Gamma_1$-invariant set $Y$ with full $\mu_1$-measure, and a measurable $\rho$-equivariant injective map $f : Y \rightarrow M_2$.
\end{itemize} 
If the measures $f_*\mu_1$ and $\mu_2$ are not singular, then 
$$
\sup_{\gamma \in \Gamma_1} \abs{\delta_1 \norm{\gamma}_{\sigma_1} - \delta_2 \norm{\rho(\gamma)}_{\sigma_2}} < +\infty.
$$
\end{theorem}

\begin{remark}
   By Theorem \ref{thm:conical limit set has full measure}, when we have the trivial hierarchy $\mathscr{H}_1(R) \equiv \Ga_1$, the condition $\mu_1(\La^{\rm con}(\mathscr{H}_1)) = 1$ in Theorem \ref{thm:main rigidity theorem} is equivalent to
   $$
   \sum_{\ga \in \Ga_1} e^{- \delta_1 \norm{\ga}_{\sigma_1}} = + \infty.
   $$
\end{remark}

\subsection{Proof of Theorem \ref{thm:main rigidity theorem}}
The rest of the section is devoted to the proof the theorem.  
For notational convenience, we  write $\| \cdot \|_i = \| \cdot \|_{\sigma_i}$ for $i = 1, 2$. 

Suppose that $f_*\mu_1$ and $\mu_2$ are not singular.
Since $f|_Y$ is injective  and $M_1, M_2$ are compact and metrizable, $f$ maps Borel subsets of $Y$ to Borel subsets of $M_2$ \cite[Coro. 15.2]{Kechris_classical}. Hence 
\begin{equation} \label{eqn.tildemu2}
\tilde\mu_2:=\mu_2(f(Y \cap \cdot))
\end{equation}
defines a finite Borel measure on $M_1$. 

\begin{lemma} \label{lem.wmaabsconteachother}
   The Borel measure $\tilde\mu_2$ is non-zero, and after possibly replacing $Y$ with a subset, we can assume that $\tilde \mu_2 \asymp \mu_1$ (i.e., $\tilde \mu_2 \ll \mu_1$ and $\tilde \mu_2 \gg \mu_1$).
 \end{lemma} 

\begin{proof} Decompose 
$$
\tilde\mu_2 = \tilde\mu_2' + \tilde \mu_2''
$$
where $\tilde \mu_2' \ll \mu_1$ and $ \tilde \mu_2''$ is singular to $\mu_1$.

Suppose for a contradiction that $\tilde \mu_2'$ is the zero measure. Then $\tilde\mu_2$ is singular to $\mu_1$. Then there exists a measurable subset $Y' \subset Y \subset M_1$ such that $\mu_1(Y') = 1$ and $\tilde\mu_2(Y') = 0$. Then 
$$
f_* \mu_1(f(Y')) \geq \mu_1(Y') = 1
$$
and
$$
\mu_2(f(Y')) = \tilde\mu_2(Y') = 0.
$$
Hence $\mu_2$ and $f_*\mu_1$ are singular, which is a contradiction. So $\tilde \mu_2'$ is not the zero measure. In particular, $\tilde \mu_2$ is non-zero.

Now fix a measurable subset $A \subset Y$ such that $\mu_1(A) = 1$ and $\tilde \mu_2''(A) = 0$. Since $\mu_1$ is $\Ga_1$-quasi-invariant, $A':=\bigcap_{\gamma \in \Gamma_1} \gamma A$ also has full $\mu_1$-measure and so by replacing $Y$ with $A'$ we can assume that $\tilde \mu_2 \ll \mu_1$.

Suppose for a contradiction that $\mu_1 \not\ll \tilde \mu_2$. Then there exists a measurable subset $B \subset Y$ where $\mu_1(B) > 0$ and $\tilde \mu_2(B) = 0$. Since the $\Ga_1$-action on $(M_1, \mu_1)$ is ergodic (Corollary~\ref{cor:ergodicity}), $\mu_1( \Gamma_1 \cdot B) = 1$. Since $\mu_2$ is $\Ga_2$-quasi-invariant and $\mu_2(f(Y \cap B)) = \tilde \mu_2(B)=0$, 
$$
\tilde \mu_2(  \Gamma_1 \cdot B) \leq \sum_{\gamma \in \Gamma_1} \mu_2( \rho(\gamma) f(Y \cap B)) =0.
$$
Hence $\mu_1$ and $\tilde \mu_2$ are singular, which contradicts the fact that $\tilde \mu_2 \ll \mu_1$. So $\mu_1 \ll \tilde \mu_2$ and thus $\mu_1 \asymp \tilde \mu_2$. 
\end{proof} 

By Lemma \ref{lem.wmaabsconteachother}, we can consider the following Radon--Nykodim derivative:
$$
h := \frac{d\tilde \mu_2}{d\mu_1} \in L^1(M_1, \mu_1).
$$

Since $\mu_1,\mu_2$ are PS-measures, $h$ satisfies the following. 

\begin{lemma}\label{lem:transformation rule for h} 
   There exists $C_1 \ge 0$ such that for any $\ga \in \Ga_1$ and $\mu_1$-a.e. $x \in M_1$,
$$
e^{-C_1 + \delta_1 \sigma_1(\gamma, x) -\delta_2 \sigma_2(\rho(\gamma), f(x))} \cdot h(x) \le h(\gamma x) \le e^{C_1 + \delta_1 \sigma_1(\gamma, x) -\delta_2 \sigma_2(\rho(\gamma), f(x))} \cdot h(x).
$$ 
\end{lemma} 

\begin{proof} 
   
   Since $\mu_2$ is a coarse $\sigma_2$-PS measure of dimension $\delta_2$, there exists $c_1 \ge 0$ such that
   $$
   e^{-c_1 -\delta_2 \sigma_2(\rho(\ga), y)} \le \frac{d \rho(\ga^{-1})_* \mu_2}{d \mu_2}(y) \le e^{c_1 -\delta_2 \sigma_2(\rho(\ga), y)} 
   $$
   for all $\ga \in \Ga_1$ and $\mu_2$-a.e. $y \in M_2$. Since $Y$ is $\Gamma$-invariant and $f$ is $\rho$-equivariant, we have for a measurable $A \subset M_1$ that 
   $$
   \gamma^{-1}_* \tilde\mu_2(A) = \mu_2( f(Y \cap  \gamma A )) = \mu_2( \rho(\gamma) f(Y \cap A)) = \rho(\gamma)^{-1}_*\mu_2(f (Y \cap A)). 
   $$
   Hence
   \begin{equation} \label{eqn.transformation rule for tilde mu2}
   e^{-c_1 -\delta_2 \sigma_2(\rho(\gamma), f(x))} \le \frac{d\gamma^{-1}_* \tilde\mu_2}{d\tilde \mu_2}(x) \le  e^{c_1 -\delta_2 \sigma_2(\rho(\gamma), f(x))} 
   \end{equation}
      for all $\ga \in \Ga_1$ and $\tilde\mu_2$-a.e. $x \in M_1$.  Since $\mu_1 \asymp  \tilde \mu_2$, this equation holds for $\mu_1$-a.e. $x \in M_1$. 
      
      Since $\mu_1$ is a coarse $\sigma_1$-PS measure of dimension $\delta_1$, there exists $c_2 \ge 0$ such that
   \begin{equation} \label{eqn.transformation rule for mu1}
   e^{-c_2 -\delta_1 \sigma_1(\ga, x)} \le \frac{d \ga^{-1}_* \mu_1}{d \mu_1}(x) \le e^{c_2 -\delta_1 \sigma_1(\ga, x)}
\end{equation} 
for all $\ga \in \Ga_1$ and $\mu_1$-a.e. $x \in M_1$.

   Finally, for any $\ga \in \Ga_1$,
   $$
\frac{d\gamma^{-1}_* \tilde\mu_2}{d\tilde \mu_2}h d \mu_1= d \ga^{-1}_* \tilde \mu_2= d \ga^{-1}_* (h\mu_1)= (h \circ \ga) d \ga^{-1}_* \mu_1. 
   $$
Combining with Equations \eqref{eqn.transformation rule for tilde mu2} and \eqref{eqn.transformation rule for mu1}, we get the desired inequalities with $C_1 := c_1 + c_2$.
\end{proof} 

Since $\mu_1 \asymp  \tilde \mu_2$ and $\Ga_1$ is countable, using Property \ref{item:coycles are bounded} we can replace $Y$ with a $\Gamma$-invariant full $\mu_1$-measure subset such that 
for all $\ga \in \Ga_1$,
\begin{equation}\label{eqn:Y satisfies PS1} 
\sup_{x \in Y} | \sigma_1(\ga, x) | < + \infty \quad \text{and} \quad \sup_{x \in Y} | \sigma_2(\rho(\ga), f(x)) | < + \infty.
\end{equation}

Since
$$
1 = \mu_1(\La^{\rm con}(\mathscr{H}_1)) =\mu_1 \left( \Ga \cdot \bigcup_{R > 0} \bigcap_{n \ge 1} \La_R(\mathscr{H}_1(n)) \right) > 0
$$
and $\mu_1$ is $\Gamma_1$-quasi-invariant, we can fix $R > 0$ such that $\bigcap_{n \ge 1} \La_R(\mathscr{H}_1(n))$ has positive $\mu_1$-measure. Since $\mu_1 \asymp  \tilde \mu_2$, $h$ is positive and finite $\mu_1$-a.e. and thus we can fix $n_0 \ge 1$ sufficiently large so that the set 
$$
E:= \{ x \in Y : n_0^{-1} \le h(x) \le n_0 \} \cap \bigcap_{n \ge 1} \La_R(\mathscr{H}_1(n))
$$
has positive $\mu_1$-measure.

Fix a sequence $R_n \rightarrow +\infty$ with $R_n \geq R$ for all $n$. After possibly increasing $R > 0$, we can assume that 
\begin{itemize}
\item $R$ satisfies Theorem \ref{thm.mixedshadow},
\item there is a subset $M_1' \subset M_1$ of full $\mu_1$-measure that satisfies Theorem \ref{thm:Leb diff mixed} for $h$ and $R$, and satisfies Lemma \ref{lem:Leb Density} for $E$ and all $R_n$. 
\end{itemize}

Fix 
$$
x_0 \in E \cap M_1'.
$$
Since $x_0 \in E \subset \bigcap_{n \ge 1} \La_R(\mathscr{H}_1(n))$, there exists an escaping sequence $\{ \ga_n \in \mathscr{H}_1(n) \}$ such that 
$$
x_0 \in \bigcap_{n \ge 1} \Oc_{R}(\ga_n).
$$

Since  $\{\gamma_n\}_{n > R} \subset \mathscr{H}_1(R)$ and $x_0 \in M_1'$, we have 
$$
h(x_0) = \lim_{n \rightarrow \infty} \frac{1}{\mu_1(\Oc^f_R(\gamma_n)) } \int_{\Oc^f_R(\gamma_n)} h d\mu_1 =\lim_{n \rightarrow \infty} \frac{\tilde \mu_2(\Oc^f_R(\gamma_n))}{\mu_1(\Oc^f_R(\gamma_n)) }.
$$
Since $x_0 \in E$, we have $h(x_0) \in [n_0^{-1}, n_0]$. Further, since $R$ satisfies Theorem \ref{thm.mixedshadow}, there exists $C_2 =C_2(R) > 1$ such that 
\begin{equation} \label{eqn:mass of shadows in long proof} 
\begin{aligned}
\frac{1}{C_2} e^{-\delta_1 \norm{\gamma}_1} & \le \mu_1(\Oc^f_R(\gamma))  \le C_2 e^{-\delta_1 \norm{\gamma}_1} \quad \text{and} \\
\frac{1}{ C_2} e^{-\delta_2 \norm{\rho(\gamma)}_2} & \le  \tilde \mu_2(\Oc^f_R(\gamma))  \le C_2 e^{-\delta_2 \norm{\rho(\gamma)}_2} 
\end{aligned}
\end{equation} 
for all $\gamma \in \Gamma_1$. Thus
\begin{equation} \label{eqn.normalongseq}
C_3:=\sup_{n \geq 1} \abs{ \delta_1 \norm{\gamma_n}_{1}-\delta_2 \norm{\rho(\gamma_n)}_{2}}
\end{equation}
is finite.

Using Lemma \ref{lem.keylemma}  we will prove the following covering lemma.

\begin{proposition}\label{prop:technical result} There exist $R'> 0$, $\alpha_1,\dots, \alpha_m \in \Gamma_1$, and $M_1'' \subset M_1$ with full $\mu_1$-measure with the following property: for any $x \in M_1''$ there exist $1 \leq i \leq m$ and $n \in \Nb$ such that 
$$
x \in \alpha_{i} \gamma_{n}^{-1}E
$$
and 
$$
(x,f(x)) \in  \left( \alpha_{i}\gamma_{n}^{-1} \Oc_{R'}(\gamma_{n})\right) \times \left(\rho(\alpha_{i})\rho(\gamma_n)^{-1} \Oc_{R'}(\rho(\gamma_n))  \right).
$$
\end{proposition} 

Delaying the proof of the proposition, we complete the proof of Theorem \ref{thm:main rigidity theorem}.

\begin{lemma} \label{lem.hboundedae}
   There exists $D > 1$ such that 
   $$D^{-1} \leq h(x) \leq D 
   $$
   for $\mu_1$-a.e. $x \in M_1$. 
\end{lemma} 

\begin{proof}
   Let $R' > 0$, $\alpha_1,\dots, \alpha_m \in \Gamma_1$, and $M_1'' \subset M_1$ be as in Proposition \ref{prop:technical result}.
   
 We start by fixing some constants. Fix $\kappa > 0$ such that $\sigma_1, \sigma_2$ are both $\kappa$-coarse-cocycles. Since $\Gamma$ is countable and $\mu_1 \asymp  \tilde \mu_2$, using Property~\ref{item:almost constant on shadows} we can fix $C_4 > 0$ and replace $M_1''$ with a full $\mu_1$-measure subset such that: if $x \in M_1''$ and $\ga \in \Ga_1$, then 
$$
\abs{\sigma_1(\gamma,x) - \norm{\gamma}_1} \leq C_4
$$
whenever $x \in \gamma^{-1}\Oc_{R'}(\gamma)$ and 
$$
\abs{\sigma_2(\rho(\gamma),f(x)) - \norm{\rho(\gamma)}_{2}} \leq C_4
$$
whenever $f(x) \in \rho(\gamma)^{-1}\Oc_{R'}(\rho(\gamma))$.  Replacing $M_1''$ by $M_1''\cap Y$, we can also assume that $M_1'' \subset Y$ and hence 
$$
C_5 : = \max_{1 \leq i \leq m} \max\left\{ \sup_{y \in M_1''} \abs{\sigma_1(\alpha_i^{-1}, y)}, \sup_{y \in M_1''} \abs{\sigma_2(\rho(\alpha_i)^{-1}, f(y))} \right\} < + \infty
$$
is finite, see Equation~\eqref{eqn:Y satisfies PS1}. Again replacing $M_1''$ with a full $\mu_1$-measure subset we can also assume that the estimate in Lemma \ref{lem:transformation rule for h} holds for all $x \in M_1''$ and all $\gamma \in \Gamma_1$.  Finally, since $\mu_1$ is $\Gamma_1$-quasi-invariant and $\Gamma_1$ is countable, we can replace $M_1''$ by $\bigcap_{\gamma \in \Gamma} \gamma M_1''$ and assume that $M_1''$ is $\Gamma_1$-invariant.

Fix $x \in M_1''$. By Proposition~\ref{prop:technical result}, there exist $1 \le i \le m$ and $n \in \mathbb{N}$ such that
$$
x \in \alpha_i \ga_n^{-1} E \cap \alpha_i \ga_n^{-1} \Oc_{R'}(\ga_n) \quad \text{and} \quad f(x) \in \rho(\alpha_{i})\rho(\gamma_n)^{-1} \Oc_{R'}(\rho(\gamma_n)).
$$
By Lemma \ref{lem:transformation rule for h}, 
$$\begin{aligned}
& e^{-C_1 - \delta_1\sigma_1(\gamma_n \alpha_i^{-1}, x)  + \delta_2 \sigma_2(\rho(\gamma_n)\rho(\alpha_i)^{-1}, f(x))} h(\gamma_n \alpha_i^{-1}x) \\
& \le h(x) \le e^{C_1 - \delta_1\sigma_1(\gamma_n \alpha_i^{-1}, x) + \delta_2 \sigma_2(\rho(\gamma_n)\rho(\alpha_i)^{-1}, f(x))} h(\gamma_n \alpha_i^{-1}x).
\end{aligned}
$$
Further, since $\alpha_i^{-1}x \in \gamma_n^{-1} \Oc_{R'}(\gamma_n) \cap M_1''$, we have
\begin{align*}
\abs{ \sigma_1(\gamma_n \alpha_i^{-1},x) - \norm{\gamma_n}_1} &  \leq \kappa + \abs{ \sigma_1(\gamma_n, \alpha_i^{-1}x) +\sigma_1(\alpha_i^{-1},x) - \norm{\gamma_n}_1}\\
&  \leq \kappa+ C_4+C_5.
\end{align*} 
Likewise, 
\begin{align*}
\abs{ \sigma_2(\rho(\gamma_n)\rho( \alpha_i)^{-1},f(x)) - \norm{\rho(\gamma_n)}_2}   \leq \kappa+ C_4+C_5.
\end{align*} 
Since $\ga_n \alpha_i^{-1} x \in E$, 
$$
n_0^{-1} \le h(\ga_n \alpha_i^{-1} x) \le n_0.
$$
Finally notice that 
$$
|\delta_1 \| \ga_n \|_1 - \delta_2 \| \rho(\ga_n) \|_2 | \le C_3 < +\infty
$$
by Equation \eqref{eqn.normalongseq}. Thus
$$
D^{-1} \le h(x) \le D,
$$
where $D := e^{C_1 +C_3+ (\delta_1 + \delta_2)(\kappa + C_4 + C_5)} n_0$.
\end{proof} 

Recalling that $h = \frac{d \tilde \mu_2}{d\mu_1}$, it follows from Lemma \ref{lem.hboundedae} that
$$
D^{-1}\mu_1(\Oc_R^f(\gamma)) \leq \tilde\mu_2(\Oc_R^f(\gamma)) \leq D \mu_1(\Oc_R^f(\gamma))
$$
for all $\gamma \in \Gamma_1$. Therefore, by Equation~\eqref{eqn:mass of shadows in long proof}, we have the desired estimate:
$$
\sup_{\gamma \in \Gamma} \abs{ \delta_1 \norm{\gamma}_{\sigma_1} - \delta_2 \norm{\rho(\gamma)}_{\sigma_2}} < +\infty.
$$
Now the proof of Theorem \ref{thm:main rigidity theorem} is complete once we show Proposition \ref{prop:technical result}.

\subsection{Proof of Proposition~\ref{prop:technical result}}

Fix metrics on $M_1$ and $M_2$ inducing their topologies. 
For each $j \geq 1$ fix a subsequence $\{ \tilde \ga_{j,n}\} \subset \{ \ga_n\}$ so that 
$$
[M_1 \smallsetminus \tilde\ga_{j,n}^{-1} \Oc_{R_j}(\tilde\ga_{j,n})] \rightarrow Z_j \quad \text{and} \quad [M_2 \smallsetminus \rho(\tilde\ga_{j,n})^{-1} \Oc_{R_j}(\rho(\tilde\ga_{j,n}))] \rightarrow Z_j'
$$
for some (possibly empty) compact subsets $Z_j \subset M_1$ and $Z_j' \subset M_2$ with respect to the Hausdorff distance. Then passing to a subsequence of $\{R_j\}$, we can assume that 
$$
Z_j \to Z \quad \text{and} \quad Z_j' \to Z'
$$
for some (possibly empty) compact subsets $Z \subset M_1$ and $Z' \subset M_2$ with respect to the Hausdorff distance.

By a diagonal argument, we can extract a subsequence $\{ \ga_{n_j} \} \subset \{\ga_n\}$ so that
\begin{align*}
[M_1 \smallsetminus \gamma_{n_j}^{-1} \Oc_{R_{j}}(\gamma_{n_j})] \rightarrow Z    \quad \text{and} \quad [M_2 \smallsetminus \rho(\gamma_{n_j}^{-1}) \Oc_{R_{j}}(\rho(\gamma_{n_j}))] \rightarrow Z' 
\end{align*}
with respect to the Hausdorff distance. Since $(M_1, \Ga_1, \sigma_1, \mu_1)$ and $(M_2, \Ga_2, \sigma_2, \mu_2)$ are PS-systems and $(M_1, \Ga_1, \sigma_1, \mu_1)$ is well-behaved (with respect to the hierarchy $\mathscr{H}_1 = \{ \mathscr{H}_1(R) \subset \Ga_1 : R \ge 0\}$), it then follows from  Lemma \ref{lem.keylemma} that
$$
M_1 \times M_2 = \bigcup_{\gamma \in \Gamma_1} (M_1 \smallsetminus \gamma Z) \times (M_2 \smallsetminus \rho(\gamma)Z').
$$
This implies that
$$
M_1 \times M_2 = \bigcup_{\epsilon > 0}\bigcup_{\gamma \in \Gamma_1} (M_1 \smallsetminus \gamma \overline{\Nc_{\epsilon}(Z)}) \times (M_2 \smallsetminus \rho(\gamma)\overline{\Nc_{\epsilon}(Z')}).
$$
By the compactness, there exist $\epsilon > 0$ and $\alpha_1, \ldots, \alpha_m \in \Ga_1$ such that 
$$
M_1 \times M_2 = \bigcup_{i = 1}^m (M_1 \smallsetminus \alpha_i \overline{\Nc_{\epsilon}(Z)}) \times (M_2 \smallsetminus \rho(\alpha_i)\overline{\Nc_{\epsilon}(Z')}).
$$
We then fix $j_0 \ge 1$ such that 
$$
Z_{j_0} \subset \Nc_{\epsilon/2}(Z) \quad \text{and} \quad Z_{j_0}' \subset \Nc_{\epsilon/2}(Z').
$$
Let $\{ \tilde \ga_n \} =\{\tilde \ga_{j_0,n}\}$. Then  there exists $N \ge 1$ such that for any $n \ge N$,
$$
M_1 \smallsetminus \tilde \ga_n^{-1} \Oc_{R_{j_0}}(\tilde \ga_n) \subset \Nc_{\epsilon/2}(Z_{j_0}) \quad \text{and} \quad M_2 \smallsetminus \rho(\tilde \ga_n)^{-1} \Oc_{R_{j_0}}(\rho(\tilde \ga_n)) \subset \Nc_{\epsilon/2}(Z_{j_0}').
$$
Therefore,
\begin{equation} \label{eqn.diagonal argument consequence}
M_1 \times M_2 = \bigcup_{i=1}^m  \left(\alpha_i \tilde \gamma_n^{-1} \Oc_{R_{j_0}}(\tilde \gamma_n)\right) \times \left( \rho(\alpha_i) \rho(\tilde \gamma_n)^{-1} \Oc_{R_{j_0}}(\rho(\tilde \gamma_n))\right) 
\end{equation}
for all  $n \ge N$.

Recall that  $M_1'$ satisfies Lemma \ref{lem:Leb Density} for $E$ and all $R_n$. Also, since $n \mapsto \mathscr{H}(n)$ is a non-increasing sequence of sets  and $\tilde \gamma_n \in \mathscr{H}(n)$ for all $n$, we have $\{\tilde \gamma_n\}_{n > R_{j_0}} \subset \mathscr{H}(R_{j_0})$. So by Lemma \ref{lem:Leb Density},
$$
\lim_{n \to \infty} \mu_1( \tilde \ga_n^{-1} \Oc_{R_{j_0}}(\tilde \ga_n) \smallsetminus \tilde \ga_n^{-1} E) = 0.
$$
Hence, since $\mu_1$ is $\Ga_1$-quasi-invariant, 
$$
\lim_{n \to \infty} \mu_1( \alpha_i \tilde \ga_n^{-1} \Oc_{R_{j_0}}(\tilde \ga_n) \smallsetminus \alpha_i \tilde \ga_n^{-1} E) = 0 
$$
for all $i = 1, \ldots, m$. We set
$$
M_1''  := M_1 \smallsetminus \bigcap_{n \geq N} \bigcup_{i=1}^m \left(  \alpha_i \tilde \gamma_n^{-1}\Oc_{R_{j_0}}(\tilde \gamma_n) \smallsetminus \alpha_i \tilde \gamma_n^{-1} E \right),
$$
which is of full $\mu_1$-measure.

For $x \in M_1''$, there exists $n \ge N$ such that
$$
x \notin \bigcup_{i = 1}^m \left(  \alpha_i \tilde \gamma_n^{-1}\Oc_{R_{j_0}}(\tilde \gamma_n) \smallsetminus \alpha_i \tilde \gamma_n^{-1} E \right).
$$
On the other hand, by Equation \eqref{eqn.diagonal argument consequence}, there exists $1 \le i \le m$ such that
$$
(x, f(x)) \in \left(\alpha_i \tilde \gamma_n^{-1} \Oc_{R_{j_0}}(\tilde \gamma_n)\right) \times \left( \rho(\alpha_i) \rho(\tilde \gamma_n)^{-1} \Oc_{R_{j_0}}(\rho(\tilde \gamma_n))\right),
$$
and therefore we must have 
$$
x \in \alpha_i \tilde \ga_n^{-1}E
$$
as well. This completes the proof of Proposition \ref{prop:technical result} with $R' := R_{j_0}$, and hence the proof of Theorem \ref{thm:main rigidity theorem}.
\qed

\part{Examples and Applications}

\section{Convergence groups and expanding coarse-cocycles}  \label{sec.convergenceaction}

In~\cite{BCZZ2024a}, Blayac--Canary--Zhu--Zimmer developed Patterson--Sullivan theory for coarse-cocycles of convergence groups. In this section we show that this theory is a special case of the definitions developed in the current paper. 

Let $M$ be a compact metrizable space and let $\Ga < \mathsf{Homeo}(M)$ be a non-elementary convergence group. In \cite[Prop. 2.3]{BCZZ2024a}  it was observed that the set $\Gamma \sqcup M$ has a unique topology such that 
\begin{itemize}  
\item $\Gamma \sqcup M$ is a compact metrizable space.
\item The inclusions $\Gamma \hookrightarrow \Gamma \sqcup M$ and $M \hookrightarrow \Gamma \sqcup M$ are embeddings (where in the first embedding $\Gamma$ has the discrete topology). 
\item the $\Ga$-action on $\Ga \sqcup M$, induced by the left-multiplication on $\Ga$ and the given $\Ga$-action on $M$, is a convergence action.
\end{itemize} 
Moreover, 
\begin{itemize}  
\item $\gamma_n \rightarrow a \in M$ and $\gamma_n^{-1} \rightarrow b \in M$ if and only if $\gamma_n|_{M \smallsetminus \{b\}} \rightarrow a$ locally uniformly. 
\end{itemize} 
For the rest of the section fix a metric $\dist$ on $\Gamma \sqcup M$ which generates this topology. 

In this setting, shadows can be defined as follows: for $\ga \in \Ga$ and $R > 0$ let
\begin{equation} \label{eqn.shadowinconvergence}
\Oc_R(\ga) := \ga\left(M \smallsetminus \overline{B_{1/R}(\ga^{-1})}\right)
\end{equation}
where $B_{1/R}(\ga^{-1})$ denotes the open ball of radius $1/R$ centered at $\ga^{-1}$ with respect to $\dist$.

\begin{remark} In~\cite{BCZZ2024a}, shadows are defined to be the closed sets 
\begin{equation*}
\ga\left(M \smallsetminus B_{1/R}(\ga^{-1})\right).
\end{equation*}
For the results cited below the difference between the two definitions is immaterial. 
\end{remark}

\begin{observation}\cite[proof of Lem. 5.4]{BCZZ2024a}\label{obs:conical limit points equivalent definition} With shadows as in Equation~\eqref{eqn.shadowinconvergence}, the set $\Lambda^{\rm con}(\Gamma)$ defined in Section~\ref{sec:conical limit set}  coincides with the set of conical limit points in the usual convergence group sense. Moreover, if $\dist(a,b) > 1/R$, $\gamma_n^{-1}x \rightarrow a$, and $\gamma_n^{-1}y \rightarrow b$ for all $y \in M \smallsetminus \{x\}$, then 
$$
x \in \bigcap_{n \geq 1} \Oc_{R}(\gamma_n).
$$
 \end{observation}

 In \cite[Def. 1.2, Prop. 3.2 and 3.3]{BCZZ2024a} the following special class of coarse-cocycles where introduced. 

\begin{definition}   \label{defn.expanding coarse cocycle}
A coarse-cocycle $\sigma : \Ga \times M \to \mathbb{R}$ is called \emph{expanding} if:
\begin{enumerate}
   \item There exists $\kappa > 0$ such that for any $\ga \in \Ga$, the function $\sigma(\ga, \cdot)$ is \emph{$\kappa$-coarsely-continuous}: for $x_0 \in M$, $$\limsup_{x \to x_0} | \sigma(\ga, x) - \sigma(\ga, x_0) | \le \kappa.$$
   \item For every $\ga \in \Ga$, there is a number $\| \ga \|_{\sigma} \in \mathbb{R}$, called the \emph{$\sigma$-magnitude of $\ga$}, with the following properties: 
   \begin{enumerate} 
   \item $\lim_{n \rightarrow \infty} \norm{\gamma_n}_\sigma = +\infty$ for any escaping sequence $\{\gamma_n\} \subset \Gamma$. 
   \item For any $\epsilon > 0$, there exists $C > 0$ such that 
   $$
   \| \ga \|_{\sigma} - C \le \sigma(\ga, x) \le \| \ga \|_{\sigma} + C
   $$
  whenever $x \in M \smallsetminus B_\epsilon(\gamma^{-1})$. 
  \end{enumerate} 
\end{enumerate}

\end{definition}  

 Part of \cite{BCZZ2024a} was devoted to developing a theory of PS-measures for expanding coarse-cocycles and using these results we show that this theory is a special case of our well-behaved PS-systems. 

\begin{theorem} \label{thm.convergencepssystem}
   Let  $\sigma : \Ga \times M \to \mathbb{R}$ be an expanding coarse-cocycle and $\mu$ a coarse $\sigma$-PS measure, then $(M, \Ga, \sigma, \mu)$ is a well-behaved PS-system with respect to the trivial hierarchy $\mathscr{H}(R) \equiv \Ga$, with shadows as in Equation~\eqref{eqn.shadowinconvergence}.
\end{theorem}

\begin{proof} Since each $\sigma(\gamma, \cdot)$ is coarsely-continuous, Property~\ref{item:coycles are bounded} is satisfied. Property~\ref{item:almost constant on shadows}  follows from the defining property of the $\sigma$-magnitude and the definition of the shadows. Property \ref{item:shadow inclusion} follows from the definition of the shadows.

Property~\ref{item:properness} is a consequence of \cite[Prop. 3.3 part (2)]{BCZZ2024a}, Property~\ref{item:intersecting shadows} is a consequence of \cite[Prop. 5.1 parts (3) and (4)]{BCZZ2024a}, and Property~\ref{item:diam goes to zero ae} is a consequence of  \cite[Prop. 5.1 part (2)]{BCZZ2024a}.

To verify Property \ref{item:empty Z intersection} and Property \ref{item:baire}, assume  $\{ \ga_n \} \subset \Ga$, $R_n \to + \infty$, and $[M \smallsetminus \ga_n^{-1} \Oc_{R_n}(\ga_n)] \to Z$ with respect to the Hausdorff distance. Then $Z$ must be singleton or empty. Then, since $\Ga$ is a non-elementary convergence group,  Property \ref{item:empty Z intersection} and Property \ref{item:baire} are true.
\end{proof}  

\subsection{Examples} We will describe one class of examples of expanding coarse-cocycle, for more see~\cite[Sect. 1.2]{BCZZ2024a}. For the rest of this subsection suppose $(X, \dist_X)$ is a proper geodesic Gromov hyperbolic metric space and $\Gamma < \mathsf{Isom}(X)$ is discrete. 

Following \cite[Def. 1.9]{BCZZ2024a} (which is similar to ~\cite[Def.\ 2.2]{CantrellTanaka_measures}), a function $\psi \colon X \times X \to \Rb$ is a \emph{coarsely additive potential} if 
\begin{enumerate}
\item\label{item:potential proper} $\lim_{r \rightarrow \infty} \inf_{\dist_X(p,q) \geq r} \psi(p,q) = +\infty$, 
\item\label{item:potential bounded} for any  $r>0$,  
$$\sup_{\dist_X(p,q) \leq r}\abs{ \psi(p,q)} < +\infty,
$$ 
\item\label{item:potential coarsely additive} for every $r > 0$ there exists $\kappa=\kappa(r) > 0$ such that: if $u$ is contained in the $r$-neighborhood of a geodesic in $(X,\dist_X)$ joining $p$ to $q$, then 
$$
\abs{\psi(p,q) - \big(\psi(p,u) + \psi(u,q)\big) } \leq \kappa.
$$
\end{enumerate}

\begin{theorem}\cite[Thm. 1.11 and 1.13]{BCZZ2024a}\label{thm:theory of potentials} 
\ \begin{enumerate}
\item  If $\psi$ is a $\Gamma$-invariant coarsely additive potential, then
\begin{align*}
\sigma_{\psi}(\gamma,x) : = \limsup_{p \rightarrow x} \, \psi(\gamma^{-1}o, p) - \psi(o,p)
\end{align*}
is an expanding coarse-cocycle on $\partial_\infty X$ and one can choose 
$$
\norm{\gamma}_{\sigma_{\psi}} = \psi(o,\gamma o).
$$
\item If $\Gamma$ acts cocompactly on $X$ and $\sigma : \Gamma \times \partial_\infty X \rightarrow \Rb$ is an expanding coarse-cocyle, then there exists a $\Gamma$-invariant coarsely additive potential $\psi$ such that 
$$
\sup_{\gamma \in \Gamma, x \in \partial_\infty X} \abs{\sigma(\gamma,x) - \sigma_\psi(\gamma,x)} < +\infty. 
$$
\end{enumerate} 
\end{theorem} 

\begin{example} \label{ex.distance Gromov potential} The distance function $\dist_X$ is a $\Isom(X)$-invariant coarsely additive potential and the associated expanding coarse-cocycle is just the coarse Busemann cocycle. \end{example} 

\begin{example}[{see~\cite[Sect. 1.2.5]{BCZZ2024a}}]\label{ex.RW GH}  Suppose $\Gamma$ is word hyperbolic, $X$ is a Cayley graph of $\Gamma$, and $\mathsf{m}$ is a probability measure on $\Gamma$ with finite superexponential moment and whose support generates $\Gamma$ as a semigroup. Then the Green metric $\dist_G$ is a $\Gamma$-invariant coarsely additive potential and the unique $\mathsf{m}$-stationary measure on $\partial_\infty \Gamma$ is a $\sigma_{\dist_G}$-PS measure of dimension 1. Note: in ~\cite[Sect. 1.2.5]{BCZZ2024a} it is assumed that $\mathsf{m}$ has finite support, but using~\cite{Gou2015} the same discussion is valid when $\mathsf{m}$ has finite superexponential moment.
\end{example} 

In Section~\ref{sec.randomwalks} we consider stationary measures on the Bowditch boundary of a relatively hyperbolic group.

\subsection{Measurable isomorphisms} As an application of Theorem~\ref{thm:main rigidity theorem in intro}, we show that for word hyperbolic groups a measurable isomorphism between boundaries endowed with PS-measures is always induced by a homeomorphism.

\begin{theorem}\label{thm:GH case meas=>homeo} For $i=1,2$ suppose $\Gamma_i$ is non-elementary word hyperbolic, $\sigma_i : \Gamma_i \times \partial_\infty \Gamma_i \rightarrow \Rb$ is an expanding coarse-cocycle, and $\mu$ is a coarse $\sigma_i$-PS measure for $\Gamma_i$ of dimension $\delta_i$ on $\partial_{\infty} \Ga_i$. Assume there exist
\begin{itemize}
\item a homomorphism $\rho : \Gamma_1 \rightarrow \Gamma_2$ with non-elementary image and 
\item a $\mu_1$-almost everywhere defined measurable $\rho$-equivariant injective map $f : \partial_\infty \Gamma_1 \rightarrow \partial_\infty \Gamma_2$.
\end{itemize} 
If $f_* \mu_1$ and $\mu_2$ are not singular, then $\ker \rho$ is finite, $\rho(\Gamma_1) < \Gamma_2$ has finite index, 
$$
\sup_{\gamma \in \Gamma_1} \abs{ \delta_1\norm{\gamma}_{\sigma_1} - \delta_2 \norm{\rho(\gamma)}_{\sigma_2}} < +\infty, 
$$
and there exists a $\rho$-equivariant homeomorphism $\tilde f : \partial_\infty \Gamma_1 \rightarrow \partial_\infty \Gamma_2$ such that 
\begin{enumerate}
\item $\tilde f = f$ $\mu_1$-a.e., 
\item $\sup_{(\gamma,x) \in \Ga_1 \times \partial_{\infty} \Ga_1} \abs{\delta_1\sigma_1(\gamma,x) -\delta_2\sigma_2(\rho(\gamma), \tilde f(x))} < + \infty$, 
\item $\tilde f_* \mu_1$, $\mu_2$ are in the same measure class and the Radon--Nikodym derivatives are a.e. bounded from above and below by a positive number.
\end{enumerate} 
\end{theorem}

\subsection{Proof of Theorem~\ref{thm:GH case meas=>homeo} } For notational convenience, we let $\norm{\cdot}_i : = \norm{\cdot}_{\sigma_i}$. 

By Theorem~\ref{thm:theory of potentials} we can assume that each $\sigma_i$ corresponds to a coarsely additive potential on a Cayley graph. Then the third defining property for coarsely additive potentials implies that there exist $c > 1$ such that 
\begin{equation}\label{eqn:QI to word metric}
c^{-1} \abs{\gamma}_i - c \leq \norm{\gamma}_i \leq c \abs{\gamma}_i + c
\end{equation} 
for all $\gamma \in \Gamma_i$, where $\abs{\cdot}_i$ is the distance from $\id \in \Ga_i$ with respect to a  word metric on $\Gamma_i$ with respect to a finite generating set.

By Theorem~\ref{thm:main rigidity theorem in intro},
\begin{equation}\label{eqn: equivalence of norms} 
\sup_{\gamma \in \Gamma_1} \abs{\delta_1\norm{\gamma}_1 - \delta_2\norm{\rho(\gamma)}_2} < +\infty. 
\end{equation} 
Then Property~\ref{item:properness} implies that $\ker \rho$ is finite and Equation~\eqref{eqn:QI to word metric} implies that $\rho$ induces a quasi-isometric embedding  $\Gamma_1 \rightarrow \Gamma_2$. So there exists a $\rho$-equivariant embedding $\tilde f : \partial_\infty \Gamma_1 \rightarrow \partial_\infty \Gamma_2$. 

For a subgroup $\Hsf < \Ga_2$, let $\delta_{\sigma_2}(\Hsf)$ be the critical exponent  of the Poincar\'e series $s \mapsto \sum_{g \in \Ga_2} e^{-s \norm{g}_{\sigma_2}}$.
Since $\mu_2$ is a coarse $\sigma_2$-PS measure for $\Gamma_2$ of dimension $\delta_2$, ~\cite[Prop. 6.2]{BCZZ2024a} implies that $\delta_2 \geq \delta_{\sigma_2}(\Gamma_2)$. Moreover, since every point in $\partial_\infty \Gamma_i$ is conical, ~\cite[Prop. 6.3]{BCZZ2024a} implies that for $i = 1, 2$,
$$
\sum_{\gamma \in \Gamma_i} e^{-\delta_i \norm{\gamma}_i} = +\infty. 
$$
This, together with Equation~\eqref{eqn: equivalence of norms}, implies that
$$
\delta_{\sigma_2}(\Gamma_2)= \delta_{\sigma_2}(\rho(\Gamma_1)) = \delta_2. 
$$
Then ~\cite[Thm. 4.3]{BCZZ2024a} implies that $\tilde f(\partial_\infty \Gamma_1) = \partial_\infty \Gamma_2$. Since $\rho(\Gamma_1)$ is quasi-convex in $\Gamma_2$, this implies that $\rho(\Gamma_1) < \Gamma_2$ has finite index.

Now by replacing $\Gamma_1$ with $\Gamma_1/\ker \rho$ and $\Gamma_2$ with $\rho(\Gamma_2)$, it suffices to consider the case where $\Gamma:=\Gamma_1 = \Gamma_2$, $\rho: \Gamma \rightarrow \Gamma$ is the identity representation, and $f : \partial_\infty \Gamma \rightarrow  \partial_\infty \Gamma$ commutes with the $\Gamma$ action, then show that 
\begin{enumerate}
\item $f=\id_{\partial_\infty \Gamma}$ $\mu_1$-a.e., 
\item $\displaystyle \sup_{(\gamma,x) \in \Ga \times \partial_{\infty} \Ga} \abs{\delta_1\sigma_1(\gamma,x) - \delta_2\sigma_2(\gamma, x)} < + \infty$,
\item $\mu_1$, $\mu_2$ are in the same measure class and the Radon--Nikodym derivatives are a.e. bounded from above and below by a positive number.
\end{enumerate} 
Assertions (2) and (3) are an immediate consequence of \cite[Prop. 13.1 and 13.2]{BCZZ2024a}.

We now show (1). Fix $R_j \rightarrow +\infty$. After possibly passing to a tail of $\{R_j\}$, by Corollary~\ref{cor:approx continuity} and the fact that $\mathscr{H}_1(R) \equiv \Gamma$, there exists a $\mu_1$-full measure set $M'$ such that whenever $x \in M' \cap \bigcap_{n \ge 1} \ga \Oc_{R_j}(\ga_n)$ for some $j \ge 1$, $\ga \in \Ga$, and an escaping sequence $\{ \ga_n \} \subset \Ga$, we have
$$
0 = \lim_{n \rightarrow \infty} \frac{1}{\mu(\gamma\Oc_{R_j}(\gamma_n))}\mu\left( \left\{ y \in \gamma\Oc_{R_j}(\gamma_n) : \dist(f(x), f(y)) > \epsilon\right\} \right)
$$
for all $\epsilon > 0$. 

Fix $x \in M'$. Since $\Gamma$ acts on $\partial_\infty \Gamma$ as a uniform convergence group,  $x$ is a conical limit point. So there exist $\{\gamma_n\}$ and distinct $a,b \in \partial_\infty \Gamma$ such that $\gamma_n^{-1}x \rightarrow a$ and $\gamma_n^{-1}y \rightarrow b$ for all $y \in \partial_\infty \Gamma \smallsetminus \{x\}$. Then $\gamma_n \rightarrow x$ and $\gamma_n^{-1} \rightarrow b$ in $\Gamma \sqcup \partial_\infty \Gamma$. So $\gamma_n|_{\partial_\infty \Gamma \smallsetminus \{b\}} \rightarrow x$ locally uniformly. Further, by Observation~\ref{obs:conical limit points equivalent definition}, 
 $$
 x \in \bigcap_{n \ge 1} \Oc_{R'}(\ga_n)
 $$
 where $R' : = \frac{2}{\dist(a,b)}$. 

\begin{lemma}\label{lem:adapting Tukias argument}  After replacing $\{\gamma_n\}$ with a subsequence we can find a $\mu_1$-full measure set $E$ where $\gamma_n f(y) \rightarrow f(x)$ for all $y \in E$. \end{lemma} 

Assuming the lemma for a moment we finish the proof. By ~\cite[Prop. 6.3 and 7.1]{BCZZ2024a}, $\mu_1$ has no atoms and by assumption $f$ is injective on a full measure set. Thus $f(E)$ has at least two points. Then, since $\gamma_n|_{\partial_\infty \Gamma \smallsetminus \{b\}} \rightarrow x$ locally uniformly, we must have $f(x) = x$. Since $x \in M'$ was arbitrary, we see that $f = \id_{\partial_\infty \Gamma}$ $\mu_1$-a.e.

\begin{proof}[Proof of Lemma~\ref{lem:adapting Tukias argument}]
For $R_j \geq R'$, notice that 
\begin{align*}
0 & = \lim_{n \rightarrow \infty} \frac{1}{\mu(\Oc_{R_j}(\gamma_n))}\mu\left( \left\{ y \in \Oc_{R_j}(\gamma_n) : \dist(f(x), f(y)) > \epsilon\right\} \right) \\
& = \lim_{n \rightarrow \infty} \frac{1}{(\gamma_n^{-1})_*\mu(\gamma_n^{-1}\Oc_{R_j}(\gamma_n))}(\gamma_n^{-1})_*\mu\left( \left\{ y \in \gamma_n^{-1}\Oc_{R_j}(\gamma_n) : \dist(f(x), \gamma_n f(y)) > \epsilon\right\} \right)
\end{align*}
for all $\epsilon > 0$.  

By Property~\ref{item:almost constant on shadows}, there exists $C_j = C_j(R_j) > 1$ such that
   $$
   \frac{1}{C_j} e^{-\delta\norm{\gamma_n}} \le \frac{d{\ga_n^{-1}}_*\mu}{d\mu} \le C_j e^{-\delta\norm{\gamma_n}} \quad \mu\text{-a.e.}
   $$
   on $\ga_n^{-1} \Oc_{R_j}(\ga_n)$.    Hence 
\begin{align*}
0  = \lim_{n \rightarrow \infty} \frac{1}{\mu(\gamma_n^{-1}\Oc_{R_j}(\gamma_n))}\mu\left(\left\{ y \in  \gamma_n^{-1}\Oc_{R_j}(\gamma_n) : \dist(f(x), \gamma_n f(y)) > \epsilon\right\} \right)
\end{align*}
for all $R_j \geq R'$ and $\epsilon > 0$. Since 
$$
\mu(\gamma^{-1}\Oc_{R_j}(\gamma)) \leq 1,
$$
we have 
\begin{align*}
0  = \lim_{n \rightarrow \infty} \mu\left(\left\{ y \in  \gamma_n^{-1}\Oc_{R_j}(\gamma_n) : \dist(f(x), \gamma_n f(y)) > \epsilon\right\} \right)
\end{align*}
for all $R_j \geq R'$ and $\epsilon > 0$.

After passing to a subsequence of $\{\gamma_n\}$, we can fix $\epsilon_n \searrow 0$ such that 
\begin{align}\label{eqn:sum of contracting sets}
\sum_{n =1}^\infty \mu\left(\left\{ y \in  \gamma_n^{-1}\Oc_{R_n}(\gamma_n) : \dist(f(x), \gamma_n f(y)) > \epsilon_n\right\} \right) < +\infty. 
\end{align}
Recall that $\gamma_n^{-1} \rightarrow b$ in $\Gamma \sqcup M$. Then let
$$
E_n : = \left\{ y \in  \gamma_n^{-1}\Oc_{R_n}(\gamma_n) : \dist(f(x), \gamma_n f(y)) > \epsilon_n\right\} 
$$
and 
$$
E : =(\partial_\infty \Gamma-\{b\}) \smallsetminus \bigcap_{N =1}^\infty \bigcup_{n \geq N} E_n.
$$
By ~\cite[Prop. 6.3 and 7.1]{BCZZ2024a}, $\mu_1$ has no atoms and hence Equation~\eqref{eqn:sum of contracting sets} implies that $E$ has full $\mu_1$-measure. Further, if $y \in E \subset \partial_\infty \Gamma \smallsetminus \{b\}$, then 
$$
y \in \gamma_n^{-1}\Oc_{R_n}(\gamma_n) =\partial_\infty \Gamma \smallsetminus \overline{B_{1/R_n}(\gamma_n^{-1})} 
$$ 
for $n$ sufficiently large and there exists $N \geq 1$ such that $y \notin  \bigcup_{n \geq N} E_n$. Thus $\gamma_n f(y) \rightarrow f(x)$. 
\end{proof}

\section{Discrete subgroups of Lie groups} \label{sec.Liegps}

Let $\Gsf$ be a connected semisimple Lie group without compact factors and with finite center.
We fix a Cartan decomposition $\mathfrak{g} = \mathfrak{k} + \mathfrak{p}$ 
of the Lie algebra of $\Gsf$, a Cartan subspace $\mfa \subset \mfp$, and a positive closed Weyl chamber $\mathfrak{a}^+\subset\mathfrak a$. Then  let $$\kappa:\Gsf\to\mathfrak{a}^+$$ denote the associated Cartan projection. Denoting by $\Asf = \exp \fa$ and $\Asf^+ = \exp \fa^+$, we have $\Gsf = \Ksf \Asf^+ \Ksf$ for a maximal compact subgroup $\Ksf < \Gsf$.
 The Jordan projection $\lambda : \Gsf \to \fa^+$ is given by 
$$
\lambda(g) = \lim_{n \rightarrow \infty} \frac{\kappa(g^n)}{n}.
$$
We also let $\opp : \fa \to \fa$ denote the opposition involution, which is defined as $\opp(\cdot) = - \operatorname{Ad}_{w_0}(\cdot)$ where $w_0$ is the longest Weyl element. We then have $\kappa(g^{-1}) = \opp(\kappa(g))$ for all $g \in \Gsf$.

Let $X := \Gsf/\Ksf$ and fix a basepoint $o = [e] \in \Gsf / \Ksf$. Fix a $\Ksf$-invariant norm $\| \cdot \|$ on $\fa$ induced from the Killing form, and let $\dist_X$ denote the $\Gsf$-invariant symmetric Riemannian metric on $X$ defined by $\dist_X(g o, h o) = \| \kappa(g^{-1} h) \|$ for $g, h \in \Gsf$.

Let $\Msf < \Ksf$ be the centralizer of $\Asf$, and 
$\Delta$ the set of all simple roots associated to $\fa^+$. For a non-empty subset $\theta \subset \Delta$, 
let $\Psf_{\theta}$ be the standard parabolic subgroup corresponding to $\theta$. That is, $\Psf_{\theta}$ is generated by $\Msf \Asf$ and all root subgroups $\Usf_{\alpha}$, where $\alpha$ ranges over all positive roots and any
negative root which is a $\mathbb{Z}$-linear combination of $\Delta \smallsetminus \theta$.
We denote by $\mathsf{N}_{\theta}$ the unipotent radical of $\Psf_{\theta}$. We simply write $\Psf = \Psf_{\Delta}$ and $\Nsf = \Nsf_{\Delta}$.

Let $\fa_{\theta} := \bigcap_{\alpha \in \Delta \smallsetminus \theta} \ker \alpha$ and let $\fa_{\theta}^*$ denote the space of $\R$-linear forms on $\fa_{\theta}$.
Let $p_{\theta} : \fa \to \fa_{\theta}$ be the unique projection which is invariant under all Weyl elements fixing $\fa_{\theta}$ pointwise. We can identify  $\fa_{\theta}^*$ with the subspace of $p_{\theta}$-invariant linear forms on $\fa$.

The Furstenberg boundary and general $\theta$-boundary  are defined as $$\F := \Gsf / \Psf = \Ksf / \Msf \quad \text{and} \quad \F_{\theta} := \Gsf / \Psf_{\theta}$$
respectively. We denote by $\pi_{\theta} : \F \to \F_{\theta}$ the quotient map.

Let $\Psf_{\theta}^{\rm opp} := w_0 \Psf_{\opp(\theta)} w_0^{-1}$ which is a parabolic subgroup opposite to $\Psf_{\theta}$, and denote by $\mathsf{N}^{\rm opp}_{\theta}$ the unipotent radical of $\Psf_{\theta}^{\rm opp}$. Two points $x \in \F_{\theta}$ and $y \in \F_{\opp(\theta)}$ are called \emph{transverse} if there exists $g \in \Gsf$ such that
$$
x = g \Psf_{\theta} \quad \text{and} \quad y = g w_0 \Psf_{\opp(\theta)}.
$$
One can see that $x \in \F_{\theta}$ is transverse to $w_0 \Psf_{\opp(\theta)}$ if and only if $x \in \Nsf_{\theta}^{\rm opp} \Psf_{\theta}$.

\subsection{Iwasawa cocycles and Patterson--Sullivan measures}
The \emph{Iwasawa cocycle} $B : \Gsf \times \F \to \fa$ is defined as follows: for $g \in \Gsf$ and $x \in \F$, fix $k \in \Ksf$  such that $k \Msf = x$ and let $B(g, x) \in \fa$ be the unique element such that
$$
gk \in \Ksf \left( \exp B(g, x) \right) \Nsf.
$$
For general $\theta \subset \Delta$, the \emph{partial Iwasawa cocycle} $B_{\theta} : \Gsf \times \F_{\theta} \to \fa_{\theta}$ is defined as
$$
B_{\theta}(g, x) = p_{\theta} \left( B(g, \tilde x) \right)
$$
for some (any) $\tilde x \in \pi_\theta^{-1}(x) \in \F$. This does not depend on the choice of $\tilde x$ \cite[Lem. 6.1]{Quint_PS}. Then $B_{\theta}$ satisfies the cocycle relation: for any $x \in \F_{\theta}$ and $g_1, g_2 \in \Gsf$,
$$
B_{\theta}(g_1 g_2, x) = B_{\theta}(g_1, g_2 x) + B_{\theta}(g_2, x).
$$

Let $\Hsf < \Gsf$ be a subgroup. Recall from the introduction  that for $\delta \ge 0$ and $\phi \in \fa_{\theta}^*$, a Borel probability measure $\mu$ on $\F_{\theta}$ is called a \emph{coarse $\phi$-Patterson--Sullivan measure (coarse $\phi$-PS measure) for $\Hsf$ of dimension $\delta$ } if there exists $C\geq1$ such that for any $\ga \in \Hsf $ the measures $\mu, \gamma_*\mu$ are absolutely continuous and 
   $$
   C^{-1}e^{- \delta \phi(B_{\theta}(g^{-1}, x))} \le \frac{d \ga_* \mu}{d \mu}(x) \le Ce^{- \delta \phi(B_{\theta}(g^{-1}, x))} \quad \text{for } \mu\text{-a.e. } x \in \Fc_{\theta}. 
   $$
   If $C=1$, then $\mu$ is a \emph{$\phi$-Patterson--Sullivan measure  ($\phi$-PS measure) for $\Hsf$ of dimension $\delta$}.

\subsection{Limit sets} \label{subsec.convergence theta bdr}
We say that a sequence \emph{$\{ g_n \} \subset \Gsf$ converges to $x \in \F_{\theta}$} if
\begin{itemize}
   \item $\alpha(\kappa(g_n)) \to + \infty$ for all $\alpha \in \theta$ and
   \item a Cartan decomposition $g_n = k_n (\exp \kappa(g_n)) \ell_n \in \Ksf \Asf^+ \Ksf$ satisfies 
   $$
   k_n \Psf_{\theta} \to x \quad \text{in } \F_{\theta}.
   $$
\end{itemize}
We say that the sequence \emph{$g_n o \in X$ converges to $x$ if $g_n \to x$}. 
This notion of convergence leads us to define the limit set of a discrete subgroup.
\begin{definition}
   Let $\Ga < \Gsf$ be a discrete subgroup. The \emph{limit set of $\Ga$ in $\F_{\theta}$} is  defined as
   $$
   \La_{\theta}(\Ga) := \{ x \in \F_{\theta} : \ga_n \to x \text{ for some sequence } \{\ga_n \} \subset \Ga \}.
   $$
\end{definition}
When $\Ga < \Gsf$ is Zariski dense, then $\La_{\theta}(\Ga)$ is the unique $\Ga$-minimal set in $\F_{\theta}$ as shown by Benoist \cite{Benoist_properties}.
Note that if $\{g_n \} \subset \Gsf$ is a sequence converging to a point in $\F_{\theta}$, then $\{g_n^{-1}\} \subset \Gsf$ has a subsequence converging to a point in $\F_{\opp(\theta)}$.
The following well-known lemma asserts that such a sequence $\{ g_n \} \subset \Gsf$ exhibits a source-sink dynamics, giving the motivation for the definitions above.

\begin{lemma}
   Let $\{g_n\} \subset \Gsf$ be a sequence such that $g_n \to x \in \F_{\theta}$ and $g_n^{-1} \to y \in \F_{\opp(\theta)}$ as $n \to \infty$. Then for any $z \in \F_{\theta}$ transverse to $y \in \F_{\opp(\theta)}$, we have
   $$
   g_n z \to x \quad \text{as } n \to \infty.
   $$
\end{lemma}

For a proof see \cite[Lem. 2.9]{LO_Invariant} (for $\theta = \Delta$), \cite[Lem. 2.4]{KOW_PD}, \cite[Prop. 2.3]{CZZ_transverse}, or \cite[Sect. 4]{KLP_Anosov}.

\subsection{Transverse subgroups}\label{sec:defn of transverse groups}

The class of transverse subgroups of $\Gsf$ provides well-behaved PS-systems.

\begin{definition}
   A discrete subgroup $\Ga < \Gsf$ is \emph{$\Psf_\theta$-transverse} if
   \begin{itemize}
      \item  $\alpha(\kappa(g_n)) \to + \infty$ for all $\alpha \in \theta$ and
      \item any distinct $x, y \in \La_{\theta \cup \opp(\theta)} (\Ga)$ are transverse.
   \end{itemize}
   A $\Psf_\theta$-transverse subgroup $\Ga < \Gsf$ is called \emph{non-elementary} if $\# \La_{\theta \cup \opp(\theta)}(\Ga) > 2$.
\end{definition}

\begin{remark} In the literature, transverse groups are sometimes called antipodal groups (e.g. ~\cite{KLP_Anosov}). \end{remark}

It is easy to see that for a $\Psf_\theta$-transverse $\Ga < \Gsf$, the canonical projection $\La_{\theta \cup \opp(\theta)}(\Ga) \to \La_{\theta}(\Ga)$ is a $\Ga$-equivariant homeomorphism (cf. \cite[Lem. 9.5]{KOW_PD}).
An important feature of a $\Psf_\theta$-transverse subgroup $\Ga < \Gsf$ is that the $\Ga$-action on $\La_{\theta}(\Ga)$ is a convergence action (\cite[Thm. 4.16]{KLP_Anosov}, \cite[Prop. 2.8]{CZZ_transverse}) and that there is a natural class of expanding cocycles. 

\begin{proposition} \cite[Prop. 10.3]{BCZZ_counting} \label{prop.transverseexpanding}
   Let $\Ga < \Gsf$ be a non-elementary $\Psf_\theta$-transverse subgroup and $\phi \in \fa_{\theta}^*$. If $\phi(\kappa(\ga_n)) \to + \infty$ for any sequence $\{ \ga_n \} \subset \Ga$ of distinct elements, then  $\sigma_{\phi}:=\phi \circ B_\theta|_{\Ga \times \La_{\theta}(\Ga)}$ 
   is an expanding coarse-cocycle with magnitude $\ga \mapsto \phi(\kappa(\ga))$.
   
   Hence, if $\mu$ is a coarse $\phi$-PS measure for $\Ga$ supported on $\La_{\theta}(\Ga)$ of dimension $\delta$, then $(\La_{\theta}(\Ga), \Ga, \sigma_{\phi}, \mu)$ is a well-behaved PS-system of dimension $\delta$ with resepct to the trivial hierarchy $\mathscr{H}(R) \equiv  \Ga $, with shadows as in Equation \eqref{eqn.shadowinconvergence}.
\end{proposition}

Given a subgroup $\Gamma < \Gsf$ and a functional $\phi \in \fa_\theta^*$, let $\delta_{\phi}(\Gamma) \in [0,+\infty]$ denote the critical exponent of the Poincar\'e series 
$$
 \sum_{\ga \in \Ga} e^{-s \phi(\kappa(\ga))},
 $$
i.e. the series diverges for $s \in [0,\delta_\phi(\Ga))$ and converges for $s \in ( \delta_\phi(\Ga), +\infty)$.  For transverse groups, we have the following existence/uniqueness results.

\begin{theorem}\label{thm:transverse group PS existence/uniqueness} 
Suppose $\Ga < \Gsf$ is a non-elementary $\Psf_\theta$-transverse subgroup and $\phi \in \fa_{\theta}^*$ satisfies $\delta_{\phi}(\Gamma) < +\infty$.
\begin{enumerate}
\item  \cite{CZZ_transverse} There exists a  $\phi$-PS measure for $\Gamma$ of dimension $\delta_\phi(\Gamma)$ supported on $\La_{\theta}(\Ga)$.
\item  \cite{CZZ_transverse} If $\sum_{\ga \in \Ga} e^{-\delta_{\phi}(\Gamma) \phi(\kappa(\ga))}=+\infty$, then there is a unique $\phi$-PS measure for $\Gamma$ of dimension $\delta_\phi(\Gamma)$ supported on $\La_{\theta}(\Ga)$.
\item   \cite{KOW_PD}  If $\Ga$ is Zariski dense and  $\sum_{\ga \in \Ga} e^{-\delta_{\phi}(\Gamma) \phi(\kappa(\ga))}=+\infty$, then any $\phi$-PS measure for $\Gamma$ of dimension $\delta_\phi(\Gamma)$ is supported on $\La_{\theta}(\Ga)$.
\end{enumerate} 
\end{theorem}

\subsection{(relatively) Anosov groups} A non-elementary $\Psf_\theta$-transverse group $\Gamma$ is \emph{$\Psf_\theta$-Anosov} if it is word hyperbolic (as an abstract group) and there is an equivariant homeomorphism between the Gromov boundary $\partial_\infty \Gamma$ and the limit set $\Lambda_\theta(\Gamma)$. More generally, a non-elementary  $\Psf_\theta$-transverse group $\Gamma$ is \emph{relatively $\Psf_\theta$-Anosov with respect to a collection $\Pc$ of subgroups} if it is relatively hyperbolic with respect to $\Pc$ (as an abstract group) and there is an equivariant homeomorphism between the Bowditch boundary $\partial(\Gamma,\Pc)$ and the limit set $\Lambda_\theta(\Gamma)$.

For relatively Anosov groups, the Poincar\'e series diverges at its critical exponent.

\begin{theorem}\cite{relativelyAnosovPS} \label{thm:Poincare series diverges} If $\Gamma < \Gsf$ is relatively $\Psf_\theta$-Anosov, $\phi \in \fa_{\theta}^*$, and $\delta_{\phi}(\Gamma) < +\infty$, then $\sum_{\ga \in \Ga} e^{-\delta_{\phi}(\Gamma) \phi(\kappa(\ga))}=+\infty$.

\end{theorem}

\subsection{Irreducible subgroups}

We now consider a more general class of subgroups.

\begin{definition}
  A subgroup $\Gamma < \Gsf$ is called \emph{$\Psf_{\theta}$-irreducible} if for any $x \in \F_{\theta}$ and $y \in \F_{\opp(\theta)}$, there exists $\ga \in \Ga$ such that $\ga x$ is transverse to $y$. We say that $\Ga$ is \emph{strongly $\Psf_{\theta}$-irreducible} if any finite index subgroup of $\Ga$ is $\Psf_{\theta}$-irreducible.
\end{definition}

It is easy to see that any Zariski dense subgroup of $\Gsf$ is strongly $\Psf_{\theta}$-irreducible.
We will show that irreducible subgroups form PS-systems, with higher rank shadows defined as follows. First, for $p \in X$ and $R > 0$, let $B_X(p, R) $ denote the metric ball $ \{ x \in X : \dist_X(x, p) < R\}$. Then, for $q \in X$, the \emph{$\theta$-shadow} $O_R^{\theta}(q, p) \subset \F_{\theta}$ of $B_X(p, R)$ viewed from $q$ is defined as
$$
    O_R^{\theta}(q, p)   := \{ g\Psf_{\theta} \in \F_{\theta} : g \in \Gsf, \ go = q, \ gA^+o \cap B_X(p, R) \neq \emptyset \}.
$$
Note that for any $g \in \Gsf$, $q, p \in X$, and $R > 0$,
$$
g O_R^{\theta}(q, p) = O_R^{\theta}(g q, g p).
$$
We will use the following  observations. 

\begin{lemma}[{\cite[Lem. 5.7]{LO_Invariant}, \cite[Lem. 5.7]{KOW_PD}}] \label{lem.busemanncartan} For any $R > 0$ there exists $C > 0$ such that: if $g \in \Gsf$ and $x \in O_R^{\theta}(g^{-1}o, o)$, then 
$$
\norm{ p_{\theta}(\kappa(g)) -B_\theta(g,x)} \leq C.
$$
\end{lemma}

\begin{lemma}\label{lem.precompactshadow}  For any relatively compact subset $V \subset \mathsf{N}^{\rm opp}_\theta$ there exists $R_0 > 0$ such that: if $g \in \Gsf$ has a Cartan decomposition $g = k a \ell  \in \Ksf \Asf^+ \Ksf$, then
$$
\ell^{-1} V \Psf_{\theta} \subset O_{R_0}^{\theta}(g^{-1}o, o).
$$
\end{lemma} 

\begin{proof} Notice that the desired inclusion is equivalent to $V \Psf_{\theta} \subset O_{R_0}^{\theta}(a^{-1} o, o)$. 

Fix $h \in V$ and let 
$$
a h = k' b n \in \mathsf{K}\mathsf{A}\mathsf{N}.
$$ 
denote the Iwasawa decomposition of $ah$.  Notice that 
 $$
a h a^{-1} = k'(b a^{-1})(a n a^{-1}) \in \mathsf{K}\mathsf{A}\mathsf{N},
$$
is the Iwasawa decomposition of $a h a^{-1}$. Since $V \subset \mathsf{N}_\theta^{\rm opp}$ is relatively compact and $a \in \mathsf{A}^+$, there exists a relatively compact subset $V'\subset \mathsf{N}_\theta^{\rm opp}$, which only depends on $V$, such that $aha^{-1} \in V'$. Then, since the Iwasawa decomposition induces a diffeomorphism $\mathsf{K} \times \mathsf{A} \times \mathsf{N} \rightarrow \mathsf{G}$,  there exists a relatively compact subset $W \subset \mathsf{G}$, which only depends on $V$, such that 
$$
b a^{-1},  a n a^{-1} \in W.
$$
Since $n \in \mathsf{N}$ and $a \in \mathsf{A}^+$, there exists a relatively compact subset $W' \subset \mathsf{G}$, which only depends on $V$, such that $n \in W'$. 

Then 
$$
h n^{-1} b^{-1} a \in V \cdot W'^{-1} \cdot W^{-1} 
$$
is uniformly bounded. Thus there exists $R_0 > 0$, which only depends on $V$, such that 
$$
h n^{-1} b^{-1} a o \in B_X(o, R_0).
$$
Therefore, $h\Psf_\theta = h(n^{-1}b^{-1}) \Psf_{\theta} \in O_{R_0}^{\theta}( h n^{-1} b^{-1}o, o)$. Since $h n^{-1} b^{-1} = a^{-1} k'$, we have $h  \Psf_{\theta} \in O_{R_0}^{\theta}(a^{-1} o, o)$. This finishes the proof.
\end{proof}

We now verify that irreducible subgroups give PS-systems. We emphasize that $\Ga$ is not assumed to be discrete in the following.

\begin{theorem} \label{thm.irreduciblePS}
   Let $\Ga < \Gsf$ be a $\Psf_{\theta}$-irreducible subgroup. If $\phi \in \fa_{\theta}^*$ and $\mu$  is a coarse $\phi$-PS measure on $\F_{\theta}$, then $(\F_{\theta}, \Ga, \sigma_{\phi}, \mu)$ is a PS-system with magnitude $\ga \mapsto \phi(\kappa(\ga))$ and shadows $\{\Oc_R(\ga) := O_R^{\theta}(o, \ga o) : \ga \in \Ga, R > 0 \}$.
\end{theorem}

\begin{proof}
Since $B_{\theta}$ is continuous and $\F_{\theta}$ is compact, Property \ref{item:coycles are bounded} holds. Property \ref{item:almost constant on shadows} follows from Lemma \ref{lem.busemanncartan}. We now show Property \ref{item:empty Z intersection}.

Suppose $\{\gamma_n\} \subset \Gamma$, $R_n \rightarrow +\infty$, and $[M \smallsetminus \gamma_n^{-1}\Oc_{R_n}(\gamma_n)] \rightarrow Z$ with respect to the  Hausdorff distance. Since
$$
\gamma_n^{-1}\Oc_{R_n}(\gamma_n) = O_R^{\theta}(\ga_n^{-1}o, o),
$$
Lemma \ref{lem.precompactshadow} implies that  $Z \subset \F_{\theta} \smallsetminus k \Nsf_{\theta}^{\rm opp} \Psf_{\theta}$ for some $k \in \Ksf$. Since $k \Nsf_{\theta}^{\rm opp} \Psf_{\theta}$ consists of points transverse to $k w_0 \Psf_{\opp(\theta)}$, Property \ref{item:empty Z intersection} follows from the definition of $\Psf_{\theta}$-irreducibility.
\end{proof}

    \subsection{Zariski dense discrete subgroups}

   In this section, we show that Zariski dense discrete subgroups give rise to well-behaved PS-systems with respect to some natural subsets.

   Let $\Ga < \Gsf$ be a Zariski dense discrete subgroup.  For $R > 0$ and $\ga \in \Ga$, we consider the shadow
   \begin{equation} \label{eqn.shadow ZD}
   \Oc_R(\ga) := O_R^{\Delta}(o, \ga o) \subset \F.
\end{equation}
   For  $u \in \interior \fa^+$ and $r > 0$, we collect elements of $\Ga$ along the direction $u$:
   $$
   \Ga_{u, r} := \{ \ga \in \Ga : \| \kappa(\ga) - t u \| < r \text{ for some } t > 0 \}.
   $$

   \begin{theorem} \label{thm.directionalps}
      Let $\Ga < \Gsf$ be a Zariski dense discrete subgroup and $u \in \interior \fa^+$. Let $\phi \in \fa^*$ be such that $\phi(u) > 0$ and let $\mu$ be a $\phi$-PS measure for $\Ga$ on $\F$. Then for any $r > 0$, the PS-system 
      $(\F, \Ga, \sigma_{\phi}, \mu)$ is well-behaved with respect to the constant hierarchy $\mathscr{H}(R) \equiv \Ga_{u, r}$, with magnitude $\ga \mapsto \phi(\kappa(\ga))$ and shadows as in Equation \eqref{eqn.shadow ZD}.
   \end{theorem}
    
   \begin{proof}
   By Theorem \ref{thm.irreduciblePS}, $(\F, \Ga, \sigma_{\phi}, \mu)$ is a PS-system. To see Property \ref{item:baire}, let $\{\gamma_n\} \subset \Gamma$ and $R_n \rightarrow +\infty$ be sequences so that $[M \smallsetminus \gamma_n^{-1}\Oc_{R_n}(\gamma_n)] \rightarrow Z$ with respect to the  Hausdorff distance. Since
$$
\gamma_n^{-1}\Oc_{R_n}(\gamma_n) = O_{R_n}^{\Delta}(\ga_n^{-1}o, o),
$$
Lemma \ref{lem.precompactshadow} implies that  $Z \subset \Fc \smallsetminus k \Nsf^{\rm opp} \Psf$ for some $k \in \Ksf$.
   
Thus $Z$ is contained in a proper subvariety of $\F$. Hence, Property \ref{item:baire} follows from the Zariski density of $\Ga$.
   Property \ref{item:properness} and Property \ref{item:shadow inclusion} are straightforward. By \cite[Lem. 3.6 and its proof]{BLLO}, Property \ref{item:intersecting shadows} holds. Property \ref{item:diam goes to zero ae} is a consequence of $u \in \interior \fa^+$.
   \end{proof}

   \begin{remark}
      The set $\La^{\rm con}(\Ga_{u, r}) = \La^{\rm con}(\mathscr{H})$ above is related to the notion of ``$u$-directional limit set'' discussed in \cite{Link_ergodicity, BLLO, Sambarino_report, KOW_SF}. When $\Ga$ is an irreducible lattice and $\mu$ is a $\Ksf$-invariant measure on $\F$, it follows from the work of Link \cite{Link_ergodicity} that $\mu(\La^{\rm con}(\Ga_{u, r})) = 1$ for all large $r > 0$. For general $\Ga$ and $\mu$, it was shown by
      Burger--Landesberg--Lee--Oh \cite{BLLO}  that $\mu(\La^{\rm con}(\Ga_{u, r})) = 1$ holds for large $r > 0$ if and only if the right-multiplication of $\exp (u \R)$ on $\Ga \ba \Gsf / \Msf$ is ergodic with respect to a Bowen--Margulis--Sullivan measure associated to $\mu$ (see also \cite{KOW_SF}). It was also shown in \cite{BLLO} that if $\Ga < \Gsf$ is $\Psf_{\Delta}$-Anosov and $\rank G \le 3$, $\mu( \La^{\rm con}(\Ga_{u, r})) = 1$ for some $u \in \interior \fa^+$ and all large $r > 0$.
   \end{remark}

   \subsection{Tukia's theorem in higher rank}
   Let $\Gsf_1, \Gsf_2$ be connected semisimple Lie groups without compact factors and with finite centers. For $i = 1, 2$, let $\theta_i$ be a non-empty subset of simple roots for $\Gsf_i$.
   Combining Proposition \ref{prop.transverseexpanding} and Theorem \ref{thm:main rigidity theorem}, we obtain the following.

   \begin{corollary} \label{cor.Tukia in higher rank} 
      For $i = 1, 2$, let $\Ga_i < \Gsf_i$, $\phi_i \in \fa_{\theta_i}^*$, and $\mu_i$ a coarse $\phi_i$-PS measure for $\Ga_i$ of dimension $\delta_i$ on $\F_{\theta_i}$. Suppose
      \begin{itemize}
         \item $\Ga_1$ is non-elementary $\Psf_{\theta_1}$-transverse and $\sum_{\ga \in \Ga_1} e^{-\delta_1 \phi_1(\kappa(\ga))} = + \infty$.
         \item $\Ga_2$ is $\Psf_{\theta_2}$-irreducible.
         \item There exists an onto homomorphism $\rho : \Ga_1 \to \Ga_2$ and a $\mu_1$-almost everywhere defined measurable $\rho$-equivariant injective map $f : \Fc_{\theta_1} \rightarrow \Fc_{\theta_2}$.
      \end{itemize} 
      If $f_*\mu_1$ and $\mu_2$ are not singular, then
      $$
      \sup_{\ga \in \Ga_1} | \delta_1 \phi_1(\kappa(\ga)) - \delta_2 \phi_2(\kappa(\rho(\ga))) | < + \infty.
      $$
   \end{corollary}   

   \begin{remark} \label{rmk.directional}
      
      Note that $\Ga_2$ is not assumed to be discrete. When $\theta_1$ is the set of all simple roots for $\Gsf_1$, by Theorem \ref{thm.directionalps}, we can replace the first condition in Corollary \ref{cor.Tukia in higher rank} with $\Ga_1$ being Zariski dense discrete and $\mu_1(\La^{\rm con}(\Ga_{1, u, r})) = 1$ for some $u \in \interior \fa_1^+$ with $\phi_1(u) > 0$ and $r > 0$.
      \end{remark}

    To complete the proof of    Theorem \ref{thm:Tukia in higher rank} from the introduction, we use the following result of Dal'Bo--Kim. 
    
       \begin{theorem} \cite{DK_criterion} \label{thm.dalbokim}
      For $i=1,2$, suppose that $\Gsf_i$ is simple and has a trivial center and let $\Ga_i < \Gsf_i$ be a Zariski dense subgroup and $\phi_i \in \fa_i^* \smallsetminus \{0\}$. If $\rho : \Ga_1 \to \Ga_2$ is an onto homomorphism and 
      $$
      \sup_{\ga \in \Ga_1} \left| \phi_1(\lambda(\ga)) - \phi_2 (\lambda(\rho(\ga)))  \right| < + \infty,
      $$
      then $\rho$ extends to a Lie group isomorphism $\Gsf_1 \to \Gsf_2$.
   \end{theorem}

\section{Group actions with contracting isometries} \label{sec.contracting}

In this section we use the theory of contracting isometries on general metric spaces developed by Coulon \cite{Coulon_PS} and Yang \cite{Yang_conformal}, to verify that Busemann PS-measures on the Gardiner--Masur boundary of Teichm\"uller space are part of PS-systems. Let $\Sigma$, $(\Tc, \dist_{\Tc})$, and $\Mod(\Sigma)$ be as in Section~\ref{sec:RW on Tecihmuller space in intro}.

\begin{theorem}[Teichm\"uller space] \label{thm.Teichmueller is wellbehavedPS} Suppose $\Ga < \Mod(\Sigma)$ is non-elementary and $\mu$ is a Busemann PS-measure for $\Gamma$ of dimension $\delta$ on $\partial_{GM} \Tc$. 
Then   $\mu$ is  part of a well-behaved PS-system with respect to some hierarchy $\mathscr{H} = \{\mathscr{H}(R) \subset \Ga  : R \ge 0\}$ and with magnitude function $\gamma \mapsto \dist_{\Tc}(o,\gamma o)$ for a fixed $o \in \Tc$. Moreover, if 
   $\sum_{\ga \in \Ga} e^{-\delta \dist_{\Tc}(o, \ga o)} = +\infty,$
   then 
   $$
   \mu(\Lambda^{\rm con}(\mathscr{H})) = 1. 
   $$
\end{theorem}

In fact, we show a more general result about isometric actions on general metric spaces which have a contracting isometry (see Theorems \ref{thm.contractingPS} and \ref{thm.contractingwellbehaved} below). 

\begin{remark} \label{rmk.cat0 in section 10}

   Let $X$ be a proper geodesic $\operatorname{CAT}(0)$ space. The same statement as in Theorem \ref{thm.Teichmueller is wellbehavedPS} holds for a non-elementary discrete subgroup of $\Isom(X)$ with a rank one isometry and a Busemann PS-measure on the visual boundary (which coincides in this case with the horofunction boundary) (see Examples \ref{ex.contracting isom}, \ref{ex.horofunction}, and \ref{ex.saturation}).

   \end{remark}

\subsection{Contracting isometries}

Let $(X, \dist)$ be a proper geodesic metric space. For a closed subset $Y \subset X$ and $x \in X$, a point $y \in Y$ is called a nearest-point projection of $x$ on $Y$ if $\dist(x, y) = \dist(x, Y)$. This defines a set-valued map $\pi_Y$ as follows: for a subset $Z \subset X$,
$$
\pi_Y(Z) = \{ y \in Y : y \text{ is a nearest-point projection of some } z \in Z \}.
$$

\begin{definition}
   For $\alpha \ge 0$, a closed subset $Y \subset X$ is called \emph{$\alpha$-contracting} if for any geodesic $L \subset X$ with $\dist(L, Y) \ge \alpha$,
   $$
   \diam \pi_Y(L) \le \alpha.
   $$
   We call $Y$ \emph{contracting} if $Y$ is $\alpha$-contracting for some $\alpha \ge 0$.
\end{definition}

\begin{definition}
   An isometry $g \in \Isom(X)$ is called \emph{($\alpha$-)contracting} if an orbit map $\mathbb{Z} \to X$, $n \mapsto g^n x$, is a quasi-isometric embedding and the image is ($\alpha$-)contracting, for some (hence any) $x \in X$.
\end{definition}
Note that conjugates of contracting elements are contracting.
In this section, we consider the assumption: 
\begin{equation} \label{eqn.standingassupmtioncontracting}
\Ga < \Isom(X)  \text{ discrete with a contracting isometry.} \tag{CTG}
\end{equation}
Such $\Ga$ is acylindrically hyperbolic \cite{Sisto_acylindrical}. We also call $\Ga$ \emph{non-elementary} if $\Ga$ is not virtually cyclic.

\begin{example} \label{ex.contracting isom}
   The following are examples of metric spaces and contracting isometries:
   \begin{enumerate}
      \item When $X$ is Gromov hyperbolic space, any loxodromic isometry on $X$ is contracting \cite{Gromov_hyperbolic}.
      
      \item  Let $\Ga$ be a relatively hyperbolic group acting properly and cocompactly on a metric space $X$ by isometries (e.g. $X$ is a Cayley graph of $\Ga$). Then any infinite order element of $\Ga$ which is not conjugated into a peripheral subgroup is contracting \cite{GP_relatively, GP_quasiconvex}.
      
      \item If $X$ is $\operatorname{CAT}(0)$, any rank one isometry of $X$ is contracting \cite{BF_characterization}.
      
      \item 
      Let $\Sigma$ be a closed connected orientable surface of genus at least two.
      Consider the action of its mapping class group $\Mod(\Sigma)$ on its Teichm\"uller space $\Tc$ equipped with the Teichm\"uller metric. Then pseudo-Anosov mapping classes are contracting \cite{Minsky_contracting}.
   \end{enumerate}
\end{example}

\subsection{Horofunction compactification}
We recall the horofunction compactification of $X$. Fix a basepoint $o \in X$ and let
$$C_*(X) := \{ h : X \to \R : h(o) = 0\}$$
which is equipped with the topology of uniform convergence on compact subsets. 

We embed $X \hookrightarrow C_*(X)$ via the map
$$
x \mapsto \dist(x, \cdot) - \dist(x, o).
$$
Then by Arzel\`a--Ascoli theorem, its image has the compact closure. This gives the horofunction compactification.

\begin{definition}
   The \emph{horofunction compactification} $\overline{X}$ of $X$ is the closure of $X$ in $C_*(X)$. The \emph{horofunction boundary} of $X$ is $\partial_H X := \overline{X} \smallsetminus X$.
\end{definition}

Note that every $h \in \overline{X}$ is $1$-Lipschitz. Since uniform convergence on compact subsets is equivalent to pointwise convergence for $1$-Lipschitz functions, it follows from the separability of $X$ that $\overline{X}$ is metrizable.

\begin{example} \label{ex.horofunction}
   The following examples are horofunction boundaries. See \cite{Yang_conformal} for further discussion on each of them.
   \begin{enumerate}
      \item When $X$ is  $\operatorname{CAT}(0)$, it is well-known that the visual boundary is the same as the horofunction boundary \cite[II.8]{BH_book}.

      \item     As mentioned in the introduction, the horofunction boundary of a Teichm\"uller space $\Tc$ equipped with its Teichm\"uller metric is the same as Gardiner--Masur boundary $\partial_{GM} \Tc$ of $\Tc$ \cite{LS_horofunction}.

   \end{enumerate}
\end{example}

We employ a slightly different point of view on the horofunction compactification, which is more suitable to our purpose.
For $h \in C_*(X)$, the function $c_h : X \times X \to \R$ defined as
$$
c_h(x, y) := h(x) - h(y)
$$
is a cocycle, i.e. $c_h (x, z) = c_h (x, y) + c_h (y, z)$. Conversely, given a continuous cocycle $c : X \times X \to \R$, we have $c(\cdot, o) \in C_*(X)$. This gives another characterization of $C_*(X)$ as the space of all continuous cocycles.

In this perspective, each point $x \in X$ corresponds to the  Busemann cocycle $b_x : X \times X \to \R$ defined as
$$
b_x(y, z) = \dist(x, y) - \dist(x, z).
$$
In the rest of this section, we regard each point of $\overline{X}$ as a cocycle. It is easy to see that for $c \in \overline{X}$,
$$
|c(x, y)| \le \dist(x, y) \quad \text{for all } x, y \in X.
$$
For $g \in \Isom(X)$, its action on $X$ extends to a homeomorphism of $\overline{X}$, by
$$
(g c)(x, y) = c(g^{-1}x, g^{-1} y).
$$
In particular, $(gc)(gx, gy) = c(x, y)$.

\subsection{Shadows}

Given $x, y \in X$ and $c \in \overline{X}$, the \emph{Gromov product} is
$$
\langle x, c \rangle_y = \frac{1}{2}( \dist(y, x) + c(y, x)),
$$
which is equal to the usual Gromov product when $c \in X$.

\begin{definition}
   Let $x, y \in X$ and $R > 0$. The \emph{$R$-shadow} of $y$ seen from $x$ is 
   $$
   O_R(x, y) := \{ c \in \overline{X} : \langle x, c \rangle_y < R \}.
   $$
\end{definition}
Note that for $g \in \Isom(X)$,
$$
g O_R(x, y) = O_R(gx, gy).
$$
The following is direct from the definition:

\begin{observation} \label{obs.contractingcocycleshadow}
   Let $x, y \in X$ and $R > 0$. If $c \in O_R(x, y)$, then
   $$
   \dist(x, y) - 2R < c(x, y) \le \dist(x, y)
   $$
\end{observation}

\subsection{Patterson--Sullivan measures}

For $\Ga < \Isom(X)$, the \emph{Busemann cocycle} $\beta : \Ga \times \overline{X} \to \R$ is
$$
\beta(\ga, c) = c(\ga^{-1} o, o).
$$
Recall from Equation~\eqref{eqn.coarse PS meaasure intro} that a probability measure $\mu$ is a $\beta$-PS measure for $\Ga$ of dimension $\delta \ge 0$ on $\overline{X}$ if for every $\gamma \in \Gamma$, 
$$
\frac{d \ga_* \mu}{d \mu}(c) = e^{\delta c(o, \ga o)} \quad \text{for $\mu$-a.e. } c \in \overline{X}
$$
(in this setting we do not consider coarse PS-measures). We denote by $\delta_{\Ga} \ge 0$ the \emph{critical exponent} of the Poincar\'e series
$$
s \mapsto \sum_{\ga \in \Ga} e^{-s \dist(o, \ga o)}.
$$
Following Patterson \cite{Patterson_fuchsian} and Sullivan \cite{Sullivan_density}'s construction, Coulon and Yang showed the existence of PS-measures in the critical dimension.

\begin{proposition}[{\cite[Prop. 4.3, Cor. 4.25]{Coulon_PS}, \cite[Lem. 6.3, Prop. 6.8]{Yang_conformal}}] 
   Let $\Ga < \Isom(X)$ be a non-elementary subgroup satisfying \eqref{eqn.standingassupmtioncontracting}.
If $\delta_\Gamma < +\infty$, then there exists a $\beta$-PS measure of dimension $\delta_{\Ga}$, which is supported on $\partial_H X$.
Moreover, if a $\beta$-PS measure for $\Ga$ of dimension $\delta$ exists, then $\delta \ge \delta_{\Ga}$.
\end{proposition}

\subsection{Verification of PS-system}
In the rest of this section, let $\Ga < \Isom(X)$ be a non-elementary subgroup satisfying \eqref{eqn.standingassupmtioncontracting}. We verify that the $\Ga$-action on $\partial_H X$ gives a PS-system. 
For $\gamma \in \Gamma$, we define the $\beta$-magnitude by 
\begin{equation} \label{eqn.magnitude contracting}
   \| \ga \|_{\beta} := \dist(o, \ga o) 
\end{equation}
and the $R$-shadow of $\ga$ to be
\begin{equation} \label{eqn.contracting shadows}
\Oc_R(\ga) := \partial_H X \cap O_{R}(o, \ga o).
\end{equation}

\begin{theorem} \label{thm.contractingPS}
   Let $\Ga < \Isom(X)$ be  non-elementary and satisfying \eqref{eqn.standingassupmtioncontracting}. 
 If $\mu$ is a $\beta$-PS measure for $\Gamma$, then $(\overline{X}, \Ga, \beta, \mu)$ is a PS-system with magnitude and shadows as in Equations \eqref{eqn.magnitude contracting} and \eqref{eqn.contracting shadows}. Moreover, Properties \ref{item:baire}--\ref{item:shadow inclusion}  hold with any choice of the hierarchy $\mathscr{H}$.
\end{theorem}

We further show that the PS-system in Theorem \ref{thm.contractingPS} is well-behaved under some condition related to a saturation of $\partial_H X$; w call $c \in \partial_H X$ \emph{saturated} if for any $c' \in \overline{X} \smallsetminus \{c\}$, $\| c - c' \|_{\infty} = +\infty$.

To make an appropriate choice of the hierarchy $\{\mathscr{H}(R)\subset \Ga : R \ge 0\}$, we use the notion of contracting tails, following \cite{Coulon_PS}.

\begin{definition}
   Let $\alpha, L \ge 0$. For $x, y \in X$, we say that the pair $(x, y)$ has an \emph{$(\alpha, L)$-contracting tail} if there exists an $\alpha$-contracting geodesic $\tau$ ending at $y$ and a projection $p \in \tau$ of $x$ such that $\dist(p, y) \ge L$. 
\end{definition}

We then consider the following subset of $\Ga$:
$$
\Cc(\alpha, L) := \{ \ga \in \Ga : (o, \ga o) \text{ has an }(\alpha, L)\text{-contracting tail}\}.
$$
Note that for a fixed $\alpha \ge 0$, the set $\Cc(\alpha, L)$ is non-increasing in $L \ge 0$.

\begin{theorem} \label{thm.contractingwellbehaved}
   Let $\Ga < \Isom(X)$ be  non-elementary and satisfying \eqref{eqn.standingassupmtioncontracting}. 
   If $\mu$ is a $\beta$-PS measure for $\Gamma$ and $\mu$-a.e. point in $\La^{\rm con}(\Ga)$ is saturated, then the PS-system $(\partial_H X, \Ga, \beta, \mu)$ is well-behaved with respect to the hierarchy  $\{ \mathscr{H}(R) = \Cc(\alpha, R + 16 \alpha + 1): R \ge 0\}$ for some $\alpha \ge 0$, with magnitude and shadows as in Equations \eqref{eqn.magnitude contracting} and \eqref{eqn.contracting shadows}.

\end{theorem}

\begin{example} \label{ex.saturation}
   
   The following are examples that almost every point is saturated:

   \begin{enumerate}
      \item Suppose that $(X, \dist)$ is $\operatorname{CAT}(0)$. Then its horofunction boundary $\partial_H X$ is the same as its visual boundary, and every single point of $\partial_H X$ is saturated.
      
      \item Suppose that $(X, \dist)$ is the Teichm\"uller space $\Tc$ of a closed connected orientable sufrace $\Sigma$ of genus at least two, equipped with the Teichm\"uller metric. 
      Then its horofunction boundary $\partial_{GM} \Tc$ contains the space $\PMF$ of projective measured foliations on $\Sigma$ as a proper subset \cite{GM_boundary}. Moreover, the subset $\UE \subset \PMF$ of uniquely ergodic ones is topologically embedded in $\partial_{GM} \Tc$ \cite[Coro. 1]{Miyachi_UE}, and every point in $\UE$ is saturated \cite[Lem. 12.6]{Yang_conformal}.
      
      Let $\Ga < \Mod(\Sigma)$ be non-elementary and $\mu$ its PS-measure of dimension $\delta$ on $\partial_{GM} \Tc$. If $\sum_{\ga \in \Ga} e^{-\delta \dist(o, \ga o)} < +\infty$, $\mu(\La^{\rm con}(\Ga)) = 0$ by Theorem \ref{thm:conical limit set has full measure}.
      If $\sum_{\ga \in \Ga} e^{-\delta \dist(o, \ga o)} = +\infty$,  we have $\mu(\UE) = 1$ \cite[Thm. 1.14, Lem. 12.6]{Yang_conformal}. Therefore, in any case, the condition in Theorem \ref{thm.contractingwellbehaved} is verified.

   \end{enumerate}
\end{example}

In general, points in $\partial_H X$ may not be saturated, even in contracting limit sets. On the other hand, one can proceed the same argument as in our proof of the rigidity theorm (Theorem \ref{thm:main rigidity theorem}) in the so-called reduced horofunction boundary of $X$, which is obtained as the quotient of $\partial_H X$ under the equivalence relation $c \sim c'$ if and only if $\| c - c' \|_{\infty} < + \infty$. When the reduced horofunction boundary is metrizable (e.g. $X$ is a proper geodesic Gromov hyperbolic space), the same argument can be proceeded. In general, one should employ \cite[Prop. 5.1]{Coulon_ergodicity}. We omit this discussion in the current paper.

\subsection{Boundary of contracting subsets}
To prove Theorem \ref{thm.contractingPS} and Theorem \ref{thm.contractingwellbehaved}, we need to introduce more notation.
   Let $Y \subset X$ be a closed subset and $c \in \overline{X}$. A point $p \in Y$ is called a \emph{projection} of $c$ on $Y$ if
   $$
   c(p, y) \le 0 \quad \text{for all } y \in Y.
   $$
When $c \in X$, a point $y \in Y$ is a projection if and only if it is a nearest-point projection. 
The \emph{boundary at infinity} $\partial^+ Y$ is the set of all $c \in \partial_H X$ such that there is no projection of $c$ on $Y$.

Going back to a classical setting for a moment, a non-elementary discrete subgroup of $\Isom(\mathbb{H}^n)$ has infinitely many loxodromic elements with disjoint fixed points on $\partial_{\infty} \mathbb{H}^n$. The following is a similar phenomenon in this current setting.

\begin{proposition}[{\cite[Lem. 2.12]{Yang_SCC}, \cite[Prop. 3.15]{Coulon_PS}}]  \label{prop.coulon315}
   Let $\Ga < \Isom(X)$ be a non-elementary  discrete subgroup. If $\ga \in \Ga$ is a contracting isometry, then there exist infinitely many $g_i \in \Ga$ such that
   $$
   \partial^+ ( g_i \langle \ga \rangle o ) \cap \partial^+( g_j \langle \ga \rangle o) = \emptyset \quad \text{for all }i \neq j.
   $$
   In particular, $\partial^+ ( \langle g_i \ga g_i^{-1} \rangle o ) \cap \partial^+(  \langle g_j \ga g_j^{-1} \rangle o) = \emptyset$ for all $i \neq j$.
\end{proposition}

\subsection{Invisible locus}
We describe the locus which cannot be seen from a sequence of shadows. The following two lemmas can be proved by a slight modification of \cite[Proof of Prop. 4.9]{Coulon_PS}.

\begin{lemma} {\cite[Proof of Prop. 4.9]{Coulon_PS}} \label{lem.contractingvariety}
   Let $\{z_n\} \subset X$ be a sequence converging to   $z \in \partial_H X$. Let $g \in \Isom(X)$ be an $\alpha$-contracting isometry such that $z \notin \partial^+ (\langle g \rangle o)$. Suppose that $\{p_n\} \subset \langle g \rangle o$ is a sequence of projections of $z_n$ and that $p_n \to p \in \langle g \rangle o$.
   Then for any $R_n \to +\infty$, we have
   $$
   \overline{X} \smallsetminus O_{R_n}(z_n, o) \subset \{ c \in \overline{X} : c(p, g^k o) \le 4 \alpha \text{ for all } k \in \mathbb{Z}\} \quad \text{for all large } n.
   $$
\end{lemma}

\begin{lemma} \cite[Proof of Prop. 4.9]{Coulon_PS} \label{lem.contractingmultitranslates} \label{lem.translatevariety}
Let $g \in \Isom(X)$ be an $\alpha$-contracting isometry. For $i = 1, \ldots, m$, let $p_i \in \langle g \rangle o$ and set
$$
\tilde Z_i := \{ c \in \overline{X} : c(p_i, g^k o) \le 4 \alpha \text{ for all } k \in \mathbb{Z}\}.
$$
Then there exists $N > 0$ such that for all $n \in \mathbb{Z}$ with $|n| > N$, we have
$$
\left( \bigcup_{i = 1}^m \tilde Z_i \right) \cap g^n \left(  \bigcup_{i = 1}^m \tilde Z_i \right) = \emptyset.
$$
\end{lemma}

\subsection{Proof of Theorem \ref{thm.contractingPS}}
   As we observed above, $|c(x, y)| \le \dist(x, y)$ for all $c \in \overline{X}$ and $x, y \in X$. Hence,
   Property \ref{item:coycles are bounded} follows. Property \ref{item:almost constant on shadows} follows from  Observation \ref{obs.contractingcocycleshadow}. Property \ref{item:properness} and  Property \ref{item:shadow inclusion} are straightforward.

Fix a metric on $\overline{X}$ which generates the topology. Property \ref{item:empty Z intersection} is implied by Property \ref{item:baire}. To see Property \ref{item:baire}, let $\{\ga_n \} \subset \Ga$ and $R_n \to + \infty$ be sequences such that $[\partial_H X \smallsetminus \ga_{n}^{-1}\Oc_{R_n}(\ga_n)] \to Z$ with respect to the Hausdorff distance.  After passing to a subsequence, we may assume that $\ga_n^{-1} o \to z \in \partial_H X$. By Proposition \ref{prop.coulon315}, for any $h_1, \ldots, h_m \in \Ga$, there exists an $\alpha$-contracting isometry $g \in \Ga$ such that $z, h_1 z, \ldots, h_m z \notin \partial^+ ( \langle g \rangle o)$, for some $\alpha \ge 0$. Then Property \ref{item:baire} is a conseqeunce of Lemma \ref{lem.contractingvariety} and Lemma \ref{lem.translatevariety}. 
\qed

\subsection{Properties of contracting tails}
To show the well-behavedness, we employ some propeties of contracting tails obtained in \cite{Coulon_PS}.

\begin{proposition} {\cite[Lem. 4.15, Lem. 5.2]{Coulon_PS}} 
   \label{prop.contractingtailsshadows}
   Let $\alpha, L, R \ge 0$ with $L > R + 16 \alpha$. If $\|\ga_1\|_{\beta} \le \|\ga_2\|_{\beta}$ and $O_R (o, \ga_1 o) \cap O_R(o, \ga_2 o) \neq \emptyset$ for $\ga_1, \ga_2 \in \Cc(\alpha, L)$, then
   \begin{enumerate}
      \item $\abs{ \|\ga_2\|_{\beta} - (\|\ga_1\|_{\beta} + \|\ga_1^{-1}\ga_2\|_{\beta})} \le 4 R + 44 \alpha$;
      \item $O_R(o, \ga_2 o) \subset O_{R + 42 \alpha}(o, \ga_1 o)$.
   \end{enumerate}
\end{proposition}

Recall from Section \ref{sec:conical limit set} the notion of conical limit set for a subset of $\Ga$.
As a generalization of Hopf--Tsuji--Sullivan dichotomy, the following was obtained by Coulon \cite{Coulon_PS} (see also Yang \cite{Yang_conformal}).

   \begin{theorem} {\cite[Coro. 5.19]{Coulon_PS}} 
       \label{thm.contractinglimitfull}
      If $\sum_{\ga \in \Ga} e^{-\delta_{\Ga} \| \ga \|_{\beta}} = + \infty$,  then there exists $\alpha_0, R_0 \ge 0$ such that for any $\beta$-PS measure $\mu$ of dimension $\delta_{\Ga} $,
      $$
      \mu \left( \La_{R_0}(\Cc(\alpha_0, L)) \right) = 1 \quad \text{for all } L \ge 0.
      $$
   \end{theorem}

Property \ref{item:diam goes to zero ae} says that shadows converging to a generic point have diameter decaying to $0$. This can be observed from contracting tails. Recall that $c \in \partial_H X$ is saturated if for any $c' \in \overline{X} \smallsetminus \{c\}$, $\| c - c' \|_{\infty} = +\infty$.

\begin{lemma} {\cite[Coro. 5.14]{Coulon_PS}} \label{lem.contractingdecayshadow} 
   Let $\alpha, L, R \ge 0$ with $L > R + 13 \alpha$. Let $c \in \La_R(\Cc(\alpha, L))$ be saturated. For any open neighborhood $U \subset \overline{X}$ of $c$, there exists $T \ge 0$ such that for any $\ga \in \Cc(\alpha, L)$ with $\dist(o, \ga o) \ge T$,
   $$
   c \in O_R(o, \ga o) \Longrightarrow O_R(o, \ga o) \subset U.
   $$
\end{lemma}

\subsection{Proof of Theorem \ref{thm.contractingwellbehaved}}

By Thoerem \ref{thm.contractingPS}, it suffices to verify Properties \ref{item:intersecting shadows} and \ref{item:diam goes to zero ae}. First, note that for any $\alpha \ge 0$, the hierarchy $\{ \mathscr{H}(R) = \Cc(\alpha, R + 16 \alpha + 1 ): R \ge 0\}$ satisfies Property \ref{item:intersecting shadows} by Proposition \ref{prop.contractingtailsshadows}.

Hence, it suffices to show that Property \ref{item:diam goes to zero ae} holds for some $\alpha \ge 0$. We consider two cases separately. Suppose first that $\sum_{\ga \in \Ga} e^{-\delta \| \ga \|_{\beta}} < + \infty$. Then by Theorem \ref{thm.contractingPS} and Theorem \ref{thm:conical limit set has full measure},
$$
\mu (\La^{\rm con}(\Ga))= 0.
$$
Setting $M' := \partial_H X \smallsetminus \La^{\rm con}(\Ga)$, Property \ref{item:diam goes to zero ae} is vacuously true for any $\alpha \ge 0$.

Now suppose that $\sum_{\ga \in \Ga} e^{-\delta \| \ga \|_{\beta}} = + \infty$. By Theorem \ref{thm.contractinglimitfull}, there exist $\alpha_0, R_0 \ge 0$ such that
$$
\mu(\La_{R_0}(\Cc(\alpha_0, L))) = 1 \quad \text{for all } L \ge 0.
$$
We then consider the hierarchy $\{ \mathscr{H}(R) = \Cc(\alpha_0, R + 16 \alpha_0 + 1 ): R \ge 0\}$ and set
$$
M' := \left\{ c \in \bigcap_{L \ge 0}\La_{R_0}(\Cc(\alpha_0, L)) : c \text{ is saturated} \right\}
$$
which has the full $\mu$-measure by the hypothesis. To see Property \ref{item:diam goes to zero ae}, fix $R > 0$.  Then for any $c \in M'$, if $c \in \bigcap_{n = 1}^{\infty} \Oc_R(\ga_n)$
for some escaping sequence $\{ \ga_n \} \subset \mathscr{H}(R)$, then $\lim_{n \to \infty} \diam \Oc_R(\ga_n) = 0$ by Lemma \ref{lem.contractingdecayshadow}. This completes the proof.
\qed

\subsection{Proof of Theorem \ref{thm.Teichmueller is wellbehavedPS}} This follows immediately from Example \ref{ex.saturation}, Theorem \ref{thm.contractingwellbehaved}, and Theorem \ref{thm.contractinglimitfull}. 
\qed

\section{Random walks on relatively hyperbolic groups} \label{sec.randomwalks}

In this section we use results in~\cite{GGPY2021} to show that the stationary measures on the Bowditch boundary of a relatively hyperbolic group (Definition \ref{defn:RH}) are examples of PS-measures on well-behaved PS-systems. For word hyperbolic groups see the discussion in Section~\ref{ex.RW GH}.

For the rest of the section suppose $(\Gamma, \Pc)$ is relatively hyperbolic and suppose $\mathsf{m}$ is a probability measure on $\Gamma$ such that: 
\begin{enumerate}
\item The support of $\mathsf{m}$ generates $\Gamma$ as a semigroup, see Equation \eqref{eqn.RW generates Gamma}.
\item $\mathsf{m}$ has finite superexponential moment, see Equation \eqref{eqn.superexponential intro}.
\end{enumerate}

By the work of Maher--Tiozzo \cite[Thm. 1.1]{MaherTiozzo2018}, there exists a unique $\mathsf{m}$-stationary measure $\nu$ on $\partial (\Gamma, \Pc)$ and this measure has no atoms. Moreover, it is realized as the hitting measure for a sample path in a Gromov model for $(\Ga, \Pc)$. In particular, $\nu$ is $\Ga$-quasi-invariant.
 We consider the measurable cocycle defined by 
$$
\sigma_\mathsf{m}(\gamma, \cdot) = - \log \frac{d\gamma_*^{-1}\nu}{d\nu}
$$
so that $\nu$ is a $\sigma_{\mathsf{m}}$-PS measure of dimension $1$. 
More precisely, let $M' \subset \partial (\Ga, \Pc)$ be a $\Ga$-invariant subset of full $\nu$-measure on which the Radon--Nykodim derivative $\frac{d\gamma_*^{-1}\nu}{d\nu}$ is defined for all $\ga \in \Ga$. Then we set $\sigma_{\mathsf{m}}(\ga, x) = - \log \frac{d\gamma_*^{-1}\nu}{d\nu}$ for $x \in M'$ and $\sigma_{\mathsf{m}}(\ga, x)  = 0$ for $x \notin M'$.
Since the set of bounded parabolic points is countable and $\nu$ has no atoms, $\nu$ assigns full measure to the set of conical limit points.

In the rest of the section, fix a metric $\dist$ on $\Ga \sqcup \partial(\Ga, \Pc)$ that generates the topology described at the start of Section  \ref{sec.convergenceaction}. Also let $\dist_{G}$ be the Green metric on $\Ga$ associated to $\mathsf{m}$, which is a left $\Ga$-invariant asymmetric metric on $\Ga$, see Equation \eqref{eqn.green metric intro}.

\begin{theorem}\label{thm:random walks are well behaved} With the notation above, $( \partial(\Gamma, \Pc), \Gamma, \sigma_\mathsf{m}, \nu)$ is a well-behaved PS-system of dimension $1$ with respect to the trivial hierarchy $\mathscr{H}(R) \equiv \Ga$, with magnitude function $\norm{\cdot}_\mathsf{m} : = \dist_G(\id, \cdot)$ and shadows as in Equation~\eqref{eqn.shadowinconvergence}. Moreover, 
$$
\sum_{\gamma \in \Gamma} e^{-\norm{\gamma}_{\mathsf{m}}} = +\infty. 
$$
\end{theorem}

\subsection{Proof of Theorem \ref{thm:random walks are well behaved}}

As described above, the conical limit set has full $\nu$-measure. Then by Theorem~\ref{thm:conical limit set has full measure} and Observation~\ref{obs:conical limit points equivalent definition}, it suffices to prove the first assertion in  Theorem \ref{thm:random walks are well behaved}. 

For notational convenience, we write 
$$
\| \cdot \| := \| \cdot \|_{\mathsf{m}}.
$$

Properties \ref{item:empty Z intersection}, \ref{item:baire},  \ref{item:shadow inclusion},  and \ref{item:diam goes to zero ae} can be verified as in the proof of Theorem \ref{thm.convergencepssystem}. By~\cite[Prop. 7.8]{GT2020}, the Green metric $\dist_G$ is quasi-isometric to any word metric on $\Ga$ with respect to a finite generating set and hence Property \ref{item:properness} holds. 

Property \ref{item:coycles are bounded} follows from the fact that $\nu$ is a stationary measure and $\supp \mathsf{m}$ generates $\Gamma$ as a semigroup. In particular, since 
$$
\nu = \mathsf{m}^{*k} * \nu = \sum_{\gamma \in \Gamma} \mathsf{m}^{*k}(\gamma) \gamma_* \nu, 
$$
we have 
$$
\nu \geq \left(\max_{k \geq 1}  \mathsf{m}^{*k}(\gamma)\right) \gamma_* \nu 
$$
and 
$$
\gamma_* \nu \geq \gamma_* \left(\max_{k \geq 1}  \mathsf{m}^{*k}(\gamma^{-1})\right) (\gamma^{-1})_* \nu = \left(\max_{k \geq 1}  \mathsf{m}^{*k}(\gamma^{-1})\right)  \nu.
$$

It remains to verify Properties \ref{item:almost constant on shadows} and Property \ref{item:intersecting shadows}.
The following can be deduced from \cite[Coro. 1.8]{GGPY2021}.

\begin{theorem}\cite{GGPY2021} \label{thm.Ancona inequality} For every $\epsilon > 0$ there exists $C=C(\epsilon) > 0$ such that: if $\dist(\alpha, \beta) > \epsilon$, then 
\begin{equation}\label{eqn:Ancona inequality}
 \dist_G(\alpha, \beta) \le \dist_G(\alpha, \id) + \dist_G(\id, \beta) \le \dist_G(\alpha, \beta) + C.
\end{equation} 
\end{theorem}

\begin{remark} One always has  $\dist_G(\alpha, \beta) \le \dist_G(\alpha, \id) + \dist_G(\id, \beta)$ and so the non-trivial part of the above statement is the second inequality. 
\end{remark} 

We first prove Property \ref{item:almost constant on shadows}.

\begin{proposition} \label{prop.RW ps2}
   There exists a $\Ga$-invariant full $\nu$-measure subset $Y \subset \partial (\Ga, \Pc)$ where for any $R > 0$, there exists $C = C(R) > 0$ such that: if $x \in \ga^{-1} \Oc_R(\ga) \cap Y$  for some $\ga \in \Ga$, then 
   $$
   \abs{ \| \ga\|- \sigma_{\mathsf{m}}(\ga, x)} \le C.
   $$
\end{proposition}

\begin{proof}
   We consider the Martin boundary $\partial_M(\Ga, \mathsf{m})$, which is the horofunction boundary for the Green metric $\dist_G$.  First, for $\gamma \in \Gamma$, define $K_\gamma : \Gamma \rightarrow \Rb$ by $K_\gamma(g)= \frac{G_{\mathsf{m}}(g, \gamma)}{G_{\mathsf{m}}(\id, \gamma)}$, where $G_{\mathsf{m}}$ is the Green function for $\mathsf{m}$ (Equation \eqref{eqn.green metric intro}). Then the \emph{Martin boundary} $\partial_M (\Gamma, \mathsf{m})$ consists of functions $K : \Gamma \rightarrow \Rb$ where $K = \lim_{n \rightarrow \infty} K_{\gamma_n}$ for some escaping sequence $\{\gamma_n\} \subset \Gamma$. Then the set $\Gamma \sqcup \partial_M (\Gamma, \mathsf{m})$ has a topology making it a compact metrizable space and where an escaping sequence $\{\gamma_n\}$ converges to $K \in \partial_M (\Gamma, \mathsf{m})$ if and only if $K_{\gamma_n} \rightarrow K$ pointwise (see~\cite[Sect. 24]{Woess}).  Further the left action of $\Gamma$ on $\Gamma$ extends to a continuous action on $\Gamma \sqcup \partial_M (\Gamma, \mathsf{m})$ where $\gamma \cdot K = \frac{K \circ \gamma^{-1}}{K(\gamma^{-1})}$.

By~\cite[Coro. 1.7]{GGPY2021}, the identity map $\Gamma \rightarrow \Gamma$ extends to a continuous surjective equivariant map 
$$
\pi: \Gamma \sqcup \partial_M (\Gamma, \mathsf{m}) \rightarrow \Gamma \sqcup \partial(\Gamma,\Pc)
 $$
 where the pre-image of each conical limit point $x \in \partial (\Ga, \Pc)$ is a singleton $\{K_x\}$ and
 $$
 K_x = \lim_{\gamma \rightarrow x} K_\gamma.
 $$
 
 There exists a $\mathsf{m}$-stationary measure $\nu_0$ on $\partial_M (\Gamma, \mathsf{m})$ such that 
 $$
 \frac{d\gamma_*\nu_0}{d\nu_0}(K) = K(\gamma) 
 $$
 for $\nu_0$-a.e. $K$ (see~\cite[Thm. 24.10]{Woess}). Since $\pi$ is equivariant, $\pi_* \nu_0$ is a stationary measure on $\partial(\Gamma, \Pc)$ and so, by uniqueness, $\nu = \pi_* \nu_0$. Then 
 \begin{equation}\label{eqn:formula for sigma_m}
\sigma_\mathsf{m}(\gamma, x) =-\log \frac{d\gamma_*^{-1}\nu}{d\nu}(x) = -\log K_x(\gamma^{-1}) 
 \end{equation} 
for $\nu$-a.e. conical limit point $x$. Let $Y$ be a $\nu$-full measure set where every $x \in Y$ is conical and satisfies Equation~\eqref{eqn:formula for sigma_m}. Since $\nu$ is $\Gamma$-quasi-invariant, replacing $Y$ with $\bigcap_{\gamma \in \Gamma} \gamma Y$, we may assume that $Y$ is $\Ga$-invariant.

Now fix $R > 0$. Fix $\ga \in \Ga$ and 
$$ 
x \in \gamma^{-1} \Oc_R(\gamma) \cap Y = \left(\partial(\Gamma, \Pc) \smallsetminus \overline{B_{1/R}(\gamma^{-1})}\right) \cap Y.
$$ 
Then $x$ is conical and $\sigma_\mathsf{m}(\gamma, x) = -\log K_x(\gamma^{-1})$. Fix a sequence $\{\alpha_n \} \subset \Ga$ converging to $x$. Then $\dist( \gamma^{-1}, \alpha_n) > 1/(2R)$ for $n$ large. Hence by Theorem \ref{thm.Ancona inequality}
$$\begin{aligned}
\abs{ \norm{\ga} - \sigma_{\mathsf{m}}(\ga, x) } & = \abs{ \dist_G(\id, \ga) + \log K_x(\ga^{-1})} \\
& = \lim_{n \rightarrow \infty} \abs{ \dist_G(\ga^{-1}, \id) + \dist_G(\id, \alpha_n) - \dist_G(\ga^{-1}, \alpha_n) }
\end{aligned}
$$
is bounded by a constant which only depends on $R$. 
\end{proof}

To verify Property \ref{item:intersecting shadows}, we will use the following lemma whose proof follows \cite[Prop. 3.3 part (7)]{BCZZ2024a}.

\begin{lemma} \label{lem.RW PS7 pf lem}
   For any $\epsilon > 0$, there exists a finite subset $F \subset \Ga$ such that: if $\alpha, \beta \in \Ga$, $\norm{ \alpha } \le \norm{\beta}$, and $\beta^{-1} \alpha \not \in F$, then
   $$
   \dist (\beta^{-1}, \beta^{-1} \alpha) \le \epsilon.
   $$
\end{lemma}

\begin{proof}

By Theorem \ref{thm.Ancona inequality}, there exists $C = C(\epsilon) > 0$ such that if $\alpha, \beta \in \Ga$ and $\dist(\alpha^{-1}, \beta) > \epsilon$, then
$$
\dist_{G}(\alpha^{-1}, \id) + \dist_G (\id, \beta) \le \dist_G(\alpha^{-1}, \beta) + C,
$$
which is equivalent to
$$
\| \alpha \| + \| \beta \| - C \le \| \alpha \beta \|.
$$
Let $F := \{ \ga \in \Ga : \| \ga \| \le C \}$, which is finite by Property \ref{item:properness} shown above. Now if $\alpha, \beta \in \Ga$ satisfy $\| \alpha \| \le \| \beta \|$ and $\beta^{-1} \alpha \notin F$, then
$$
\| \beta \| + \| \beta^{-1} \alpha \| - C > \| \beta \| \ge \| \alpha \| = \| \beta \beta^{-1} \alpha \|.
$$
Therefore, $\dist(\beta^{-1}, \beta^{-1} \alpha) \le \epsilon$ as desired.
\end{proof}

We now prove the first half of Property \ref{item:intersecting shadows}.

\begin{proposition} \label{prop.RW ps7 first}
   For any $R > 0$, there exists $R'> 0$ such that: if $\alpha,\beta \in \Gamma$, $\norm{\alpha}\leq \norm{\beta}$, and $\Oc_R(\alpha) \cap \Oc_R(\beta) \neq \emptyset$, then 
$$
\Oc_R(\beta) \subset \Oc_{R'}(\alpha).
$$ 
\end{proposition}

\begin{proof} 

Suppose to the contrary that there exist $R > 0$ and sequences $\alpha_n, \beta_n \in \Ga$ such that $\| \alpha_n \| \le \| \beta_n \|$, $\Oc_R(\alpha_n) \cap \Oc_R(\beta_n) \neq \emptyset$, and $\Oc_R(\beta_n) \not\subset \Oc_n(\alpha_n)$ for all $n \ge 1$. This implies that for all $n \ge 1$,
$$
\alpha_n^{-1} \beta_n \left( \partial (\Ga, \Pc) \smallsetminus \overline{B_{1/R}(\beta_n^{-1})} \right) \not \subset \partial (\Ga, \Pc) \smallsetminus \overline{B_{1/n}(\alpha_n^{-1})}.
$$
Then the sequence $\{ \beta_n^{-1} \alpha_n \}$ is escaping; otherwise, $\beta_n^{-1} \alpha_n \overline{B_{1/n}(\alpha_n^{-1})} \subset \overline{B_{1/R}(\beta_n^{-1})}$ for all large $n \ge 1$, which contradicts our assumptions. 

By Lemma \ref{lem.RW PS7 pf lem}, $\dist(\beta_n^{-1}, \beta_n^{-1} \alpha_n) \to 0$ as $n \to \infty$. Hence, for $n \ge 1$ sufficiently large we have 
$$\begin{aligned}
\alpha_n^{-1} \Oc_{R}(\beta_n) & = \alpha_n^{-1} \beta_n \left( \partial (\Ga, \Pc) \smallsetminus \overline{B_{1/R}(\beta_n^{-1})}\right) \\
& \subset \alpha_n^{-1} \beta_n\left( \partial (\Ga, \Pc) \smallsetminus \overline{B_{1/(2R)}(\beta_n^{-1} \alpha_n)} \right) \subset \Oc_{2R}(\alpha_n^{-1} \beta_n).
\end{aligned}
$$
Since $\{\alpha_n^{-1} \beta_n\}$ is escaping as well, it follows from \cite[Prop. 5.1 part (2)]{BCZZ2024a} that $\diam \Oc_{2R}(\alpha_n^{-1} \beta_n) \to 0$ as $n \to \infty$, and hence $\diam \alpha_n^{-1} \Oc_R(\beta_n) \to 0$ as $n \to \infty$. 

Since $\Oc_{R}(\alpha_n) \cap \Oc_{R}(\beta_n) \neq \emptyset$ and $\alpha_n^{-1} \Oc_R(\alpha_n) = \partial(\Ga, \Pc) \smallsetminus \overline{B_{1/R}(\alpha_n^{-1})}$, it follows from $\lim_{n \to \infty} \diam \alpha_n^{-1} \Oc_R(\beta_n) = 0$  that
$$
\alpha_n^{-1} \Oc_R (\beta_n) \subset \partial (\Ga, \Pc) \smallsetminus \overline{B_{1/(2R)}(\alpha_n^{-1})} \quad \text{for all large } n \ge 1.
$$
Therefore, $\Oc_R(\beta_n) \subset \Oc_{2R}(\alpha_n)$ for all large $n \ge 1$, which is a contradiction. This finishes the proof.
\end{proof}

We prove the second half of Property \ref{item:intersecting shadows}.

\begin{proposition} \label{prop.RW ps7 second}
   For any $R > 0$, there exists $C > 0$ such that if $\alpha, \beta \in \Gamma$, $\norm{\alpha} \leq \norm{\beta}$, and $\Oc_R(\alpha) \cap \Oc_R(\beta) \neq \emptyset$, then 
   $$
   \abs{\norm{\beta} - (\norm{\alpha} + \norm{\alpha^{-1}\beta})} \le C.
   $$ 
\end{proposition} 

\begin{proof}
   Suppose to the contrary that there exist $R > 0$ and sequences $\alpha_n, \beta_n \in \Ga$ such that $\| \alpha_n \| \le \| \beta_n \|$, $\Oc_R(\alpha_n) \cap \Oc_R(\beta_n) \neq \emptyset$, and 
   $$
   \abs{\norm{\beta_n} - (\norm{\alpha_n}  + \norm{\alpha_n^{-1}\beta_n})} \ge n \quad \text{for all } n \ge 1.
   $$

By Theorem~\ref{thm.Ancona inequality}, we have 
$$
  \| \beta_n \| \le \| \alpha_n \| + \| \alpha_n^{-1} \beta_n \|
$$
for all $n \geq 1$. Hence, by assumption,  the sequence $\{ \alpha_n^{-1} \beta_n \}$ is escaping. Similarly, for all $n \ge 1$,
$$  \| \alpha_n^{-1} \beta_n \| - \| \alpha_n^{-1} \| \le  \| \beta_n \| \le \| \alpha_n \| + \| \alpha_n^{-1} \beta_n \|,$$
and hence the sequence $\{ \alpha_n \}$ is also escaping. Since $\| \alpha_n \| \le \| \beta_n \|$, $\{\beta_n\}$ is an escaping sequence as well.
Since $\{ \alpha_n^{-1} \beta_n \}$ is escaping, Lemma \ref{lem.RW PS7 pf lem} implies that  
$$
\lim_{n \to \infty} \dist(\beta_n^{-1}, \beta_n^{-1} \alpha_n) = 0.
$$

As Properties \ref{item:coycles are bounded}--\ref{item:empty Z intersection} have been verified, $(\partial (\Ga, \Pc), \Ga, \sigma_{\mathsf{m}}, \nu)$ is a PS-system. Hence, by Proposition \ref{prop.shadowlemma}, there exists $R_0 > 0$ such that $\nu(\Oc_{R_0}(\ga)) > 0$ for all $\ga \in \Ga$. 
Now by increasing $R > 0$, we may assume that $R > R_0$. By Proposition \ref{prop.RW ps7 first}, we can fix $R' > 0$ such that
$$
\Oc_R(\beta_n) \subset \Oc_{R'}(\alpha_n) \quad \text{for all } n \ge 1.
$$
Let $Y \subset \partial (\Ga, \Pc)$ be the subset in Proposition \ref{prop.RW ps2}. Since each $\Oc_R(\beta_n)$ has positive measure, for each $n \ge 1$ there exists a point
$$
x_n \in \Oc_{R}(\beta_n) \cap Y \subset \Oc_{R'}(\alpha_n) \cap Y.
$$
Moreover, since $\Oc_{R}(\beta_n) = \beta_n \left( \partial (\Ga, \Pc) \smallsetminus \overline{B_{1/R}(\beta_n^{-1})} \right)$  and $\lim_{n \to \infty} \dist(\beta_n^{-1}, \beta_n^{-1} \alpha_n) = 0$, we have
 $$
\dist( \beta_n^{-1} x_n, \beta_n^{-1} \alpha_n) \ge \dist(\beta_n^{-1} x_n, \beta_n^{-1}) - \dist (\beta_n^{-1} \alpha_n, \beta_n^{-1}) \ge \frac{1}{2R}
$$
for $n$ sufficiently large. Hence,
$$
\alpha_n^{-1} x_n \in \Oc_{2R}(\alpha_n^{-1} \beta_n) \cap Y.
$$

Now if $C = C(\max \{R', 2R\}) > 0$ satisfies Proposition \ref{prop.RW ps2}, then 
\begin{align*}
& \abs{ \| \alpha_n \| - \sigma_{\mathsf{m}}(\alpha_n, \alpha_n^{-1} x_n)}  \le C, \quad \abs{ \| \beta_n \| - \sigma_{\mathsf{m}}(\beta_n, \beta_n^{-1} x_n)}  \le C, \text{ and} \\ 
& \quad \quad \quad \abs{ \| \alpha_n^{-1} \beta_n \| - \sigma_{\mathsf{m}}(\alpha_n^{-1}\beta_n, \beta_n^{-1} x_n)}  \le C.
\end{align*}

Further, by the cocycle property
$$
\sigma_{\mathsf{m}}( \beta_n, \beta_n^{-1} x_n) = \sigma_{\mathsf{m}}( \alpha_n \alpha_n^{-1}\beta_n, \beta_n^{-1} x_n) = \sigma_{\mathsf{m}}(\alpha_n, \alpha_n^{-1} x_n) + \sigma_{\mathsf{m}}(\alpha_n^{-1} \beta_n, \beta_n^{-1} x_n).
$$
Combining altogether,
$$
\abs{\norm{\beta_n} - (\norm{\alpha_n} + \norm{\alpha_n^{-1}\beta_n})} \le 3C,
$$
which is a contradiction, finishing the proof.
\end{proof}
The proof of Theorem \ref{thm:random walks are well behaved} is now complete.
\qed

\section{Rigidity results for random walks}

In the following subsections we suppose that 
\begin{itemize}
\item $(\Gamma, \Pc)$ is a relatively hyperbolic group and
\item  $\mathsf{m}$ is probability measure on $\Gamma$ with finite superexponential moment as in Equation \eqref{eqn.superexponential intro} and whose support generates $\Gamma$ as a semigroup.
\end{itemize} 
Let $\nu_0$ be the unique $\mathsf{m}$-stationary measure on $ \partial(\Gamma, \Pc)$ and let $( \partial(\Gamma, \Pc), \Gamma, \sigma_\mathsf{m}, \nu_0)$ be the well-behaved PS-system in Theorem~\ref{thm:random walks are well behaved}. 

In the subsections that follow we will assume that $\Gamma$ is a subgroup of either  the isometry group of a Gromov hyperbolic space, the mapping class group of a surface, or a semisimple Lie group.

\subsection{Random walks on the isometry group of a Gromov hyperbolic space} In this subsection we further suppose that 
\begin{itemize}
\item $(X, \dist_X)$ is a proper geodesic Gromov hyperbolic space, and
\item  $\Gamma < \Isom(X)$ is a non-elementary discrete subgroup.
\end{itemize} 
In this setting, Kaimanovich proved that there exists a unique $\mathsf{m}$-stationary measure $\nu$ on the Gromov boundary $\partial_{\infty} X$, and is the hitting measure for a sample path \cite[Remark following Thm. 7.7]{Kaimanovich2000}.

A subset $Y \subset X$ is \emph{quasi-convex} if there exists $R > 0$ such that any geodesic joining two points in $Y$ is contained in the $R$-neighborhood of $Y$.
Then a discrete subgroup $\Gamma' < \Isom(X)$ is \emph{quasi-convex} if for any $o\in X$ the orbit $\Ga'(o) \subset X$ is quasi-convex (see \cite{Swenson_QC} for properties of such groups). Using the Morse Lemma, it is easy to see that a subgroup is quasi-convex if and only if any orbit map is a quasi-isometric embedding with respect to a word metric on the group with respect to a finite generating set.

\begin{theorem}\label{thm:random walks GH} If $\mu$ is a coarse Busemann PS-measure for $\Ga$ on $\partial_\infty X$ of dimension $\delta$, then 
   the following are equivalent:
   \begin{enumerate}
   \item The measures $\nu$ and $\mu$ are not singular.
   \item  The measures $\nu$ and $\mu$ are in the same measure class and the Radon--Nikodym derivatives are a.e. bounded from above and below by a positive number.
   \item For any $o \in X$, $$\sup_{\gamma \in \Gamma} \abs{ \dist_G(\id, \gamma) - \delta \dist_X(o, \gamma o)} < +\infty.$$ In particular, $\Ga$ is quasi-convex, $\delta$ is equal the critical exponent of $\Ga$, and $\sum_{\ga \in \Ga} e^{-\delta \dist_X(o, \ga o)} = + \infty$.
   \end{enumerate}
\end{theorem}

\begin{proof} 
   The implication $(2) \Rightarrow (1)$ is clear. We now prove $(1) \Rightarrow (3)$.
By~\cite[Remark following Thm. 7.7]{Kaimanovich2000}, the spaces $( \partial(\Gamma, \Pc), \nu_0)$ and $(\partial_\infty X, \nu)$ are both Poisson boundaries for $(\Gamma, \mathsf{m})$. Hence, there is a $\Ga$-equivariant isomorphism 
$$
f :( \partial(\Gamma, \Pc), \nu_0)  \rightarrow (\partial_\infty X, \nu).
$$
By assumption $\nu = f_* \nu_0$ is not singular with respect to $\mu$. As explained in Example \ref{ex.distance Gromov potential} and Theorem \ref{thm.convergencepssystem}, $\mu$ is a coarse PS-measure in a PS-system which has magnitude function 
$$
\gamma \mapsto \dist_X(o,\gamma o).
$$
Then by Theorem \ref{thm:main rigidity theorem},
\begin{equation*}
\sup_{\gamma \in \Gamma} \abs{ \dist_G(\id, \gamma) - \delta \dist_X(o, \gamma o)} < +\infty.
\end{equation*}
Moreover, since $\dist_G$ is quasi-isometric to a word metric on $\Ga$ with respect to a finite generating set by \cite[Prop. 7.8]{GT2020},  $\Ga$ is quasi-convex. 
 Since $\mu$ is of dimension $\delta$, $\delta$ is at least the critical exponent of the Poincar\'e series \cite[Coro. 6.6]{Coornaert_PS}. Together with $\sum_{\ga \in \Ga} e^{-\dist_{G}(\id, \ga)} = + \infty$ (Theorem \ref{thm:random walks are well behaved}), we have that $\delta$ is equal to the critical exponent and the Poincar\'e series diverges at $\delta$.

It remains to show $(3) \Rightarrow (2)$. Assuming (3), $\Ga$ is a word hyperbolic group and the orbit map $\ga \in \Ga \mapsto \ga o \in X$ is a quasi-isometric embedding with respect to a word metric on $\Ga$ as mentioned above. Hence we can assume that $\Pc=\emptyset$ and so $\partial(\Gamma, \Pc)$ coincides with the Gromov boundary $\partial_\infty \Gamma$. Further, the orbit map continuously extends to  $f : \partial_{\infty} \Ga \to \partial_{\infty} X$ which is a $\Ga$-equivariant homeomorphism onto its image. Since both $\nu$ and $\nu_0$ are hitting measures, we have $f_* \nu_0 = \nu$. Since $\sum_{\ga \in \Ga} e^{-\delta \dist_X (o, \ga o)} = + \infty$, Theorem \ref{thm.convergencepssystem}, Observation \ref{obs:conical limit points equivalent definition}, and Theorem \ref{thm:conical limit set has full measure}, imply that $\mu (f (\partial_{\infty} \Ga)) = 1$. Hence, we can take a pull-back of the Busmann cocycle on $\partial_{\infty} X$ and $\mu$ to $\partial_{\infty} \Ga$.
Since the Busemann cocycle on $\partial_{\infty} X$ is expanding (Example \ref{ex.distance Gromov potential}), the same is true for the pull-back. Therefore, (2) follows from \cite[Prop. 13.1 and 13.2]{BCZZ2024a}.
\end{proof} 

We now restate and prove Corollary \ref{cor.singularity hitting PS symmetric space intro}.

\begin{corollary} \label{cor.singularity hitting PS symmetric space}
Suppose $X$ is a negatively curved symmetric space. If $\Gamma$ is not a cocompact lattice in $\Isom(X)$, then $\nu$ is singular to the Lebesgue measure class on $\partial_{\infty} X$.
\end{corollary}
 
\begin{proof}
Suppose that $\nu$ is non-singular to the Lebesgue measure class on $\partial_{\infty} X$. Since the Lebesgue measure class contains a Busemann PS-measure for $\Ga$ (cf. \cite[Lem. 6.3]{Quint_PS}), it follows from Theorem \ref{thm:random walks GH} that $\Ga$ is convex cocompact. Since $\nu$ is supported on the limit set on $\Ga$, the limit set has a positive Lebesgue measure class. By the classical Hopf--Tsuji--Sullivan dichotomy \cite{roblin}, the Lebesgue measure class gives a unique PS-measure supported on the limit set. Therefore, $\partial_{\infty} X$ is the limit set of $\Ga$, and hence $\Ga$ must be a cocompact lattice.
\end{proof}

\subsection{Random walks on mapping class groups and Teichm\"uller spaces} 
\label{subsec.RT Teich}

Let $\Mod(\Sigma)$ denote the mapping class group of a closed connected orientable surface $\Sigma$ of genus at least two and let $(\Tc, \dist_{\Tc})$ is the Teichm\"uller space of $\Sigma$ equipped with the Teichm\"uller metric. 

We continue to assume that $\Gamma$ and $\mathsf{m}$ satisfy the assumptions at the start of the section. In this subsection we further suppose that 
\begin{itemize}
   \item $\Ga < \Mod(\Sigma)$ is a non-elementary subgroup.
\end{itemize}

Thurston compactified $\Tc$ with the space $\PMF$ of projective measured foliations on $\Sigma$ \cite{Thurston_geometry_dynamics}.
In this setting, Kaimanovich--Masur showed that there exists a unique $\mathsf{m}$-stationary measure $\nu$ on $\PMF$, and is the hitting measure for a sample path in $\Tc$ and supported on the subset $\UE \subset \PMF$ of uniquely ergodic foliations \cite[Thm. 2.2.4]{KM_MCG}. Since $\UE$ is topologically embedded in the Gardiner--Masur boundary $\partial_{GM}\Tc$ \cite{Miyachi_UE}, $\nu$ can also be regarded as a measure on $\partial_{GM}\Tc$, where PS-measures are defined.

\begin{theorem} \label{thm.RW Teich body}
   If $\mu$ is a Busemann PS-measure for $\Ga$ on $\partial_{GM} \Tc$ of dimension $\delta$ and the measures $\nu$, $\mu$ are not singular, then:
   \begin{enumerate}
   \item For any $o \in \Tc$,
   $$\sup_{\gamma \in \Gamma} \abs{ \dist_G(\id, \gamma) - \delta \dist_{\Tc}(o, \ga o)} < +\infty.$$ In particular, $\delta$ is the critical exponent of $\Ga$ and $\sum_{\ga \in \Ga} e^{-\delta \dist_{\Tc} (o, \ga o)} = + \infty$.
   \item If   $\dist_w$ is a word metric on $\Gamma$ with respect to a finite generating set, then  the map 
   $$
   \gamma \in (\Gamma, \dist_w) \mapsto \gamma o \in (\Tc, \dist_{\Tc})
   $$
   is a quasi-isometric embedding. 
   \end{enumerate} 
   \end{theorem}
   
   \begin{proof}
   By~\cite[Thm. 2.2.4]{Kaimanovich2000} the space $(\PMF, \nu)$ is a Poisson boundary for $(\Gamma, \mathsf{m})$ and by~\cite[Remark following Thm. 7.7]{Kaimanovich2000}, the space  $( \partial(\Gamma, \Pc), \nu_0)$ is a Poisson boundary for $(\Gamma, \mathsf{m})$. 
   Hence there is an isomorphism 
   $$
   f : ( \partial(\Gamma, \Pc), \nu_0)  \rightarrow (\PMF, \nu).
   $$
   Since $\nu(\UE) = 1$, we can view $f$ as a map into $\UE \subset \partial_{GM}\Tc$.
   
   By assumption $\nu = f_* \nu_0$ is not singular with respect to $\mu$. By Theorem \ref{thm.contractingPS}, $\mu$ is a PS-measure in a PS-system which has magnitude function 
   $$
   \gamma \mapsto \dist_{\Tc}(o, \ga o).
   $$
   Then by Theorem \ref{thm:main rigidity theorem},
   $$
   \sup_{\gamma \in \Gamma} \abs{ \dist_G(\id, \gamma)-\delta \dist_{\Tc}(o, \ga o)} < +\infty.
   $$
   Since $\mu$ is of dimension $\delta \ge 0$, $\delta$ is at least the critical exponent of the Poincar\'e series
      (\cite[Prop. 4.23]{Coulon_PS}, \cite[Prop. 6.8]{Yang_conformal}). Since  $\sum_{\ga \in \Ga} e^{-\dist_{G}(\id, \ga)} = + \infty$ by Theorem \ref{thm:random walks are well behaved}, we have that $\delta$ is equal to the critical exponent and the Poincar\'e series diverges at $\delta$,
   showing (1).

   By~\cite[Prop. 7.8]{GT2020} the Green metric is quasi-isometric to $\dist_w$. Therefore, (2) follows. 
   \end{proof} 
   
We can now restate (as a corollary) and prove Theorem~\ref{thm.singularity hitting PS MCG multitwist intro}.

\begin{corollary}\label{cor.singularity hitting PS MCG multitwist}  If $\Gamma$ contains a multitwist, then the $\mathsf{m}$-stationary measure $\nu$ is singular to every Busemann Patterson--Sullivan measures on $\partial_{GM} \Tc$. 
\end{corollary}

\begin{proof} By Farb--Lubotzky--Minsky \cite{FLM_MCG}, every infinite order element $g \in \Mod(\Sigma)$ has positive stable translation length on its Cayley graph, i.e., $$\limsup_{n \to \infty} \frac{\dist_w(\id, g^n)}{n} > 0$$
   for any word metric $\dist_w$ on  $\Mod(\Sigma)$ with respect to a finite generating set. On the other hand, an infinite order mapping class has zero stable translation length on $\Tc$ if and only if one of its power is a multitwist. So the result follows from Theorem \ref{thm.RW Teich body}.
   \end{proof}

   For a special class of subgroups, we prove the converse of Theorem~\ref{thm.RW Teich body}.  A subgroup $\Ga'<\Mod(\Sigma)$ is \emph{parabolically geometrically finite (PGF)} if
   \begin{itemize}
      \item $(\Ga', \Pc')$ is relatively hyperbolic for some $\Pc' = \{ P_1, \dots, P_n \}$ where each $P_i < \Ga'$ contains a finite index, abelian subgroup consisting entirely of multitwists;
      \item the coned off Cayley graph of $(\Ga', \Pc')$ embeds $\Ga'$-equivariantly and quasi-isometrically
      into the curve complex of $\Sigma$.
   \end{itemize}
   See \cite[Def. 1.10]{DDLS_PGF} for details. When $\Pc' = \emptyset$, the group $\Ga'$ is convex cocompact. This is equivalent to the original definition of \cite{FM_CC} as shown by \cite{KL_shadows,Hamenstadt_CC}.

   \begin{theorem} \label{thm.RW Teich PGF}
   Suppose  $\Ga$ is PGF.  If $\mu$ is a Busemann PS-measure for $\Ga$ on $\partial_{GM} \Tc$ of dimension $\delta$, then 
   the following are equivalent:
   \begin{enumerate}
   \item The measures $\nu$ and $\mu$ are not singular, 
   \item  The measures $\nu$ and $\mu$ are in the same measure class and the Radon--Nikodym derivatives are a.e. bounded from above and below by a positive number, 
   \item For any $o \in \Tc$, $$\sup_{\gamma \in \Gamma} \abs{ \dist_G(\id, \gamma) - \delta \dist_{\Tc}(o, \gamma o)} < +\infty.$$ In particular, $\Ga$ is convex cocompact, $\delta$ is the critical exponent of $\Ga$, and $ \sum_{\ga \in \Ga} e^{-\delta \dist_{\Tc} (o, \ga o)} = + \infty$.
   \end{enumerate}
   \end{theorem}

   \begin{proof}
   The implication $(2) \Rightarrow (1)$ is clear and $(1) \Rightarrow (3)$ follows from Theorem \ref{thm.RW Teich body}. Now suppose (3). Then $\Ga$ is word hyperbolic and the orbit map $\ga \mapsto \ga o$ continuously extends to a $\Ga$-equivariant map $f : \partial_{\infty} \Ga \to \UE$ which is a homeomorphism onto its image,
   after replacing $o \in \Tc$ with another point if necessary \cite[Thm. 1.1, Prop. 3.2]{FM_CC}. Hence, $\nu = f_* \nu_0$ since both $\nu$ and $\nu_0$ are hitting measures. Since $\sum_{\ga \in \Ga} e^{-\delta \dist_{\Tc}(o, \ga o)} = + \infty$, Theorem \ref{thm.contractinglimitfull} implies that $\mu(f(\partial_{\infty} \Ga)) = 1$. Hence, we can take the pull-back of the measure $\mu$ to $\partial_{\infty} \Ga$ via $f$, which is a PS-measure for the cocycle $\sigma_{\Tc}$ given in Proposition \ref{prop.CCMCG expanding} below. In Proposition \ref{prop.CCMCG expanding} below we will verify that $\sigma_{\Tc}$ is an expanding cocycle. Therefore, (2) follows from \cite[Prop. 13.1 and 13.2]{BCZZ2024a}.
   \end{proof}

   \begin{proposition} \label{prop.CCMCG expanding}
      Suppose $\Ga<\Mod(\Sigma)$ is convex cocompact. Let $f : \partial_{\infty} \Ga \to \UE \subset \partial_{GM} \Tc$ be the $\Ga$-equivariant embedding induced from a quasi-isometric embedding $\ga \in \Ga \mapsto \ga o \in \Tc$ for some $o \in \Tc$.
      Then  the cocycle $\sigma_{\Tc} : \Ga \times \partial_{\infty} \Ga \to \R$ given by
      $$
      \sigma_{\Tc}(\ga, x) := f(x)(\ga^{-1}o, o)
      $$
      is an expanding cocycle with magnitude $\gamma \mapsto \dist_{\Tc}(o, \ga o)$.
   \end{proposition}

   \begin{proof}

      It is clear that $\sigma_{\Tc}$ is a cocycle and $\lim_{n \to \infty}  \dist_{\Tc}(o, \ga_n o) = + \infty$ for any escaping sequence $\{ \ga_n \} \subset \Ga$.  Moreover, since $f(\partial_{\infty} \Ga) \subset \UE$, $\sigma_{\Tc}$ is continuous \cite{Miyachi_UE}. Recalling the metric $\dist$ on $\Ga \sqcup \partial_{\infty} \Ga$ from Section \ref{sec.convergenceaction}, it remains to show that for any $\epsilon > 0$, there exists $C> 0$ such that
      $$
      \dist_{\Tc}(o, \ga o) - C \le \sigma_{\Tc}(\ga, \ga^{-1}x) \le \dist_{\Tc}(o, \ga o) + C
      $$
      whenever $x \in \ga \left( \partial_{\infty} \Ga \smallsetminus \overline{B_{\epsilon}(\ga^{-1})} \right)$, where $B_{\epsilon}$ is the open $\dist$-ball of radius $\epsilon$  centered at $\ga^{-1}$. 

      Let $d_w$ be a word metric on $\Ga$ with respect to a finite generating set.
      Fix $\epsilon > 0$. It is easy to see that there exists $R_0 > 0$ such that for any $\ga \in \Ga$ and $x \in \ga \left(\partial_{\infty} \Ga \smallsetminus \overline{B_{\epsilon}(\ga^{-1})} \right)$, any geodesic ray $[\id, x) \subset \Ga$ with respect to $d_w$ intersects the $\dist_w$-ball of radius $R_0$ centered at $\ga$.

      Let $\ga \in \Ga$ and $x \in \ga \left(\partial_{\infty} \Ga \smallsetminus \overline{B_{\epsilon}(\ga^{-1})} \right)$. Fix a geodesic ray $[\id, x) \subset \Ga$ and for each $n \ge 1$, let $\ga_n \in [\id, x)$ be such that $\dist_w(\id, \ga_n) = n$. By \cite{Miyachi_UE},
      $$
      \sigma_{\Tc}(\ga, \ga^{-1} x) = \lim_{n \to \infty} \dist_{\Tc} (\ga_n o, o) - \dist_{\Tc}(\ga_n o, \ga o).
      $$
      Fix $k \ge 1$ with $\dist_w(\ga, \ga_k) < R_0$. Since the orbit map $\Ga \to \Tc$ is a quasi-isometric embedding, we have 
      $$\dist_{\Tc}(\ga o, \ga_k o) < R$$
       for some $R > 0$ determined by $R_0$.

      For each $n \ge 1$, let $L_n \subset \Tc$ be the geodesic from $o$ to $\ga_n o$. Since $\Ga(o) \subset \Tc$ is quasi-convex \cite{FM_CC}, there exists $C_0 > 0$ such that $L_n$ is contained in the $C_0$-neighborhood of $\Ga(o)$ for all $n \ge 1$. Hence, the nearest-point projection $L_n' \subset \Ga(o)$ of $L_n$ is a quasi-geodesic. Since the orbit map is a quasi-isometric embedding, it follows from the Morse Lemma for $(\Ga, d_w)$ that for some uniform $C_1 > 0$, the quasi-geodesic $\{ \ga_1 o, \dots, \ga_n o\} \subset \Tc$ is contained in the $C_1$-neighborhood of $L_n$, for all $n \ge 1$.

      Now for all $n \ge k$,
      $$
      \dist_{\Tc}(\ga o, L_n) < R + C_1
      $$
      and hence 
      $$
      \abs{\left( \dist_{\Tc} (\ga_n o, o) - \dist_{\Tc}(\ga_n o, \ga o) \right) - \dist_{\Tc}(o, \ga o)} < 2(R + C_1).
      $$
      Taking $n \to \infty$, we have $\abs{  \sigma_{\Tc}(\ga, \ga^{-1} x) -     \dist_{\Tc}(o, \ga o) } \le 2(R + C_1)$, completing the proof with $C := 2(R + C_1)$.
   \end{proof}

\subsection{Random walks on discrete subgroups of Lie groups} \label{subsec.RW Lie gp} We continue to assume that $\Gamma$ and $\mathsf{m}$ satisfy the assumptions at the start of the section. In this subsection we suppose that 
\begin{itemize}
\item $\Gsf$ is connected semisimple Lie group without compact factors and with finite center, and 
\item  $\Gamma < \Gsf$ is a Zariski dense discrete subgroup.
\end{itemize}

Recall that $\Fc = \Gsf/\Psf$ is the Furstenberg broundary. Guivarc'h and Raugi showed that there exists a unique $\mathsf{m}$-stationary measure $\nu$ on $\F$, and it is the hitting measure for a sample path \cite{GR_Furstenberg}.

As a higher rank analogue of critical exponent, Quint introduced the notion of growth indicator on $\Ga$ \cite{Quint_Divergence}. Fixing any norm $\norm{ \cdot }$ on $\fa$, the \emph{growth indicator of $\Gamma$} is the function $\psi_{\Ga} : \fa \to \Rb \cup \{-\infty\}$ defined as follows: for $u \neq 0$,
   $$
   \psi_{\Ga}(u) := \norm{u} \inf_{\Cc \ni u} \left\{ \text{critical exponent of } s \mapsto \sum_{\ga \in \Ga} e^{-s \norm{\kappa(\ga)}} \right\}
   $$
   where the infimum is over all open cones in $\fa$ containing $u$, and $\psi_{\Ga}(0) = 0$. A functional $\phi \in \fa^*$ is \emph{tangent} to the growth indicator of $\Ga$ if $\phi \geq \psi_{\Ga}$ on $\fa$ and there exists non-zero $u \in \fa^+$ with $\phi(u) = \psi_{\Ga}(u)$.

\begin{theorem}\label{thm:random walks Zdense} If $\mu$ is a coarse $\phi$-PS measure for $\Ga$ on $\Fc$ of dimension $\delta$ and  the measures $\nu$, $\mu$ are not singular, then:
\begin{enumerate}
\item $\sup_{\gamma \in \Gamma} \abs{ \dist_G(\id, \gamma) - \delta \phi(\kappa(\gamma))} < +\infty$. In particular, $\sum_{\ga \in \Ga} e^{-\delta \phi(\kappa(\gamma))} =~+ \infty$ and $\delta \phi$ is tangent to the growth indicator of $\Ga$.

\item If  $\dist_w$ is a word metric on $\Gamma$ with respect to a finite generating set, $(X,\dist_X)$ is the symmetric space associated to $\Gsf$, and $x_0 \in X$, then  the map 
$$
\gamma \in (\Gamma, \dist_w) \mapsto \gamma x_0 \in (X,\dist_X)
$$
is a quasi-isometric embedding.
\end{enumerate} 
\end{theorem}

\begin{proof}
By~\cite[Thm. 10.7]{Kaimanovich2000} the space $(\Fc, \nu)$ is a Poisson boundary for $(\Gamma, \mathsf{m})$ and by~\cite[Remark following Thm. 7.7]{Kaimanovich2000}, the space  $( \partial(\Gamma, \Pc), \nu_0)$ is a Poisson boundary for $(\Gamma, \mathsf{m})$. Hence there is an isomorphism $f : ( \partial(\Gamma, \Pc), \nu_0)  \rightarrow (\Fc, \nu)$. 

By assumption $\nu = f_* \nu_0$ is not singular with respect to $\mu$. By Theorem \ref{thm.irreduciblePS}, $\mu$ is a coarse PS-measure in a PS-system which has magnitude function 
$$
\gamma \mapsto \phi(\kappa(\gamma)).
$$
Then by Theorem \ref{thm:main rigidity theorem},
$$
\sup_{\gamma \in \Gamma} \abs{ \dist_G(\id, \gamma)-\delta \phi(\kappa(\gamma))} < +\infty,
$$
showing the first part of (1).  Since  $\sum_{\ga \in \Ga} e^{-\dist_{G}(\id, \ga)} = + \infty$ by Theorem \ref{thm:random walks are well behaved}, we have 
$$
\sum_{\ga \in \Ga} e^{-\delta \phi(\kappa(\gamma))} = + \infty.
$$
Then \cite[Lem. 3.1.3]{Quint_Divergence} implies that $\delta \phi(u) \le \psi_{\Ga}(u)$  for some $u \neq 0$. Finally, the existence of the coarse $\phi$-PS measure $\mu$ of dimension $\delta$ implies that $\delta \phi \ge \psi_{\Ga}$ by \cite[Thm. 8.1]{Quint_PS} and so $\delta \phi$ is tangent to the growth indicator of $\Ga$. Note that while \cite[Thm. 8.1]{Quint_PS} assumes the PS-measure is non-coarse, the same proof works for coarse PS-measures as well.

To show (2), let $S \subset \Gamma$ be the finite symmetric generating set which induces $\dist_w$. By~\cite[Prop. 7.8]{GT2020} the Green metric is quasi-isometric to $\dist_w$ and so there exist $a > 1$ and $b > 0$ such that 
$$
a^{-1} \dist_w(\gamma_1,\gamma_2) - b \leq \phi(\kappa(\gamma_1^{-1}\gamma_2)) \leq a \dist_w(\gamma_1,\gamma_2) + b
$$
for all $\gamma_1,\gamma_2 \in \Gamma$. Then
$$
\dist_w(\gamma_1,\gamma_2) \leq a \phi(\kappa(\gamma_1^{-1}\gamma_2))+b \leq a\norm{\phi} \norm{\kappa(\gamma_1^{-1}\gamma_2)}+b = a\norm{\phi} \dist_X(\gamma_1o,\gamma_2o)+b
$$
and 
$$
\dist_X(\gamma_1x_0, \gamma_2x_0) \leq C \dist_w(\gamma_1, \gamma_2)
$$
where $C : = \max_{s \in S} \dist_X(x_0, sx_0)$. So (2) follows. 
\end{proof}

We now restate and prove Corollary \ref{cor:RW in G wh and unipotent}.

\begin{corollary} \label{cor:RW in G wh and unipotent body} If $\Gamma$ is word hyperbolic (as an abstract group) and contains a unipotent element of $\Gsf$, then $\nu$ is singular with respect to every coarse Iwasawa PS-measure on $\Fc$. 
\end{corollary} 

\begin{proof} Suppose for a contradiction that $\nu$ is non-singular to some coarse $\phi$-PS measure $\mu$ of dimension $\delta$. Fix a word metric $\dist_w$ on $\Gamma$ with respect to a finite generating set and $x_0 \in X$. By Theorem~\ref{thm:random walks Zdense},    the map 
$$
\gamma \in (\Gamma, \dist_w) \mapsto \gamma x_0 \in (X,\dist_X)
$$
is a quasi-isometry. However, if $u \in \Gamma$ is a unipotent element of $\Gsf$, then 
$$
\lim_{n \rightarrow \infty} \frac{1}{n} \dist_X(u^n x_0, x_0) = 0
$$
while 
$$
\lim_{n \rightarrow \infty} \frac{1}{n} \dist_w(u^n, \id) >0
$$
since $\Gamma$ is word hyperbolic and $u \in \Gamma$ has infinite order (hence is loxodromic). So we have a contradiction. 
\end{proof}

\bibliographystyle{alpha}
\bibliography{geom}

\end{document}